\documentclass[11pt]{amsart}
\usepackage{epsfig,latexsym,amsfonts,amssymb,amsmath,amscd}
\usepackage{amsthm}
\usepackage[mathscr]{eucal}
\usepackage{mathrsfs}
\usepackage{oldgerm,units,color}
\usepackage{comment}

\makeatletter
\def\section{\@startsection{section}{1}%
  \z@{1.1\linespacing\@plus\linespacing}{.8\linespacing}%
  {\normalfont\Large\scshape\centering}}
\makeatother

\theoremstyle{plain}

\newtheorem*{thmA}{Theorem B}
\newtheorem*{thmB}{Theorem A}

\newtheorem*{conj*}{Root Groups Conjecture}
\newcommand{\etype}[1]{\renewcommand{\labelenumi}{(#1{enumi})}}
\newtheorem*{thm1.2}{(1.2) Theorem}
\newtheorem*{thm1.3}{(1.3) Theorem}
\newtheorem*{thm1.4}{(1.4) Theorem}
\newtheorem*{prop*}{Proposition}

\def\fun{cross\ }

\newtheorem{prop}{Proposition}[section]

\newtheorem{thm}[prop]{Theorem}

\newtheorem{ques}[prop]{Question}
\newtheorem{cor}[prop]{Corollary}
\newtheorem{lemma}[prop]{Lemma}

\theoremstyle{definition}
\newtheorem{Def}[prop]{Definition}
\newtheorem*{Def*}{Definition}

\newtheorem{Defsnot}[prop]{Definitions and notation}
\newtheorem{example}[prop]{Example}

\newtheorem{note}[prop]{Note}

\newtheorem*{notation*}{Notation}
\newtheorem{remark}[prop]{Remark}

\newtheorem*{remark*}{Remark}
\newtheorem{hyp1}[prop]{Hypothesis}

\newcommand{\ff}{\mathbb{F}}

\newcommand{\uu}{\mathbb{U}}

\newcommand{\ga}{\alpha}
\newcommand{\gb}{\beta}
\newcommand{\gc}{\gamma}

\newcommand{\gd}{\delta}

\newcommand{\gre}{\epsilon}
\newcommand{\gl}{\lambda}

\newcommand{\gro}{\omega}

\newcommand{\gvp}{\varphi}
\newcommand{\gr}{\rho}

\newcommand{\gt}{\tau}

\newcommand{\sminus}{\smallsetminus}
\newcommand{\lan}{\langle}
\newcommand{\ran}{\rangle}

\def\eroman{\etype{\roman}}

\def\Eref{Example~\ref}

\newcommand{\hal}{\frac{1}{2}}
\newcommand{\half}{\frac{1}{2}}

\numberwithin{equation}{section}

\begin{document}
\title[Weakly primitive axial algebras]{Weakly primitive axial algebras}
\author[Louis Halle Rowen , Yoav Segev]
{Louis Halle Rowen$^*$\qquad Yoav Segev}

\address{Louis Rowen\\
         Department of Mathematics\\
         Bar-Ilan University\\
         Ramat Gan\\
         Israel}
\email{rowen@math.biu.ac.il}
\address{Yoav Segev \\
         Department of Mathematics \\
         Ben-Gurion University \\
         Beer-Sheva 84105 \\
         Israel}
\email{yoavs@math.bgu.ac.il}
\thanks{$^*$The first author was supported by the ISF grant 1994/20 and the Anshel Pfeffer Chair}



\begin{abstract}
In  earlier work we studied the structure of primitive axial
algebras of Jordan type (PAJ's), not necessarily commutative, in terms
of their primitive axes. In this paper we weaken primitivity and permit several pairs of (left and right) eigenvalues satisfying  more general fusion rules,
bringing in interesting new examples such as the band semigroup
algebras and other commutative and noncommutative examples. Also we broaden our investigation and describe 2-generated algebras in which only one of the generating axes is weakly primitive and
  satisfies the fusion rules, on condition that its zero-eigenspace is one dimensional. We also characterize when both axes
 satisfy the fusion rules (weak PAJ's), and describe precisely the 2-dimensional axial algebras. In contrast to the previous case,   there are  weak PAJ's of   dimension~$> 3$ generated by two axes.

\end{abstract}
\date{\today}

 \maketitle

\tableofcontents
\section{Introduction}

In recent years, interest was renewed in commutative algebras
generated by semisimple idempotents. These algebras were introduced as  ``axial algebras'' in  the  paper of
\cite{HRS},
in which they investigated the simplest nontrivial case, which they called ``primitive axial algebras of
  Jordan type." In previous papers \cite{RoSe1,RoSe2,RoSe3, RoSe4}, we
generalized this notion to the noncommutative setting. For an algebra $A$ which is not necessarily commutative, an
   {\it axis} is a left and right semisimple idempotent $a\in A$, such that   $a(ya)=(ay)a$ for all $y \in A.$
  The axis $a$ is ``primitive" if the
1-eigenspace of $a$ both from the left and from the right is $\ff a.$ In the theory of axial algebras, the notion of  ``fusion rules'' is central, and is inspired by Albert's theory \cite{A1,A} of idempotents in Jordan algebras. The paper \cite{HRS} deals mostly with $2$-generated commutative algebras whose two generating  axes have at most one additional fixed eigenvalue  other than 0 and 1.  In previous work~\cite{RoSe2,RoSe3},  extending the terminology of \cite{HRS} by defining a PAJ to be   a not necessarily commutative, primitive axial algebra of Jordan  type, with one left and one
right eigenvalue other than 0 and 1, we characterized PAJ's
 modulo the commutative case.

In \cite{RoSe4} we also determined the idempotents in 2-generated PAJ's, in order to answer basic questions about axes, leaving the
general case for later.

In this paper we cast the net further, extending the fusion rules to handle more than one extra eigenvalue, and
weakening the condition of ``primitive'' to ``weakly primitive,'' in order to
include some other interesting noncommutative examples  such as band semigroup
algebras, but we still obtain essentially the same theory. At times
we remove all conditions, and address the question
of when an idempotent indeed is a weakly primitive axis, and when it satisfies the fusion rules.

\begin{Defsnot}\label{not1}$ $
\begin{enumerate}
\item $A$ always denotes an  algebra {\it (possibly not commutative)},
over a field~$\ff$ of characteristic $\ne 2.$ 

\item
For  $y\in A,$ write $L_y$ for the left  multiplication map
$z\mapsto yz$, and  $R_y$ for the right  multiplication map
$z\mapsto zy$. $A_\mu(L_y)$ denotes the $\mu$-eigenspace of $L_y$,
for $\mu\in \ff$, which is permitted to be $0$; similarly for $A_\nu(R_y).$

 \item
 A {\it left axis}
$a$ is an 
idempotent  for which $A$ is a direct sum of its
left eigenspaces with respect to $L_a$. A  {\it right axis} is
defined analogously. A {\it  (2-sided)  axis} $a$  is a left axis
which is also a    right axis, for which $L_a R_a = R_a L_a$.

If $S_L: = S_L(a)$ is the set of left eigenvalues and $S_R:= S_R(a)$ is  the set of right eigenvalues  of an axis $a$, then $A$ decomposes as a direct sum of subspaces
\begin{equation}
    \label{dec18} A = \oplus _{\mu\in S_L, \ \nu\in S_R} A_{\mu,\nu}(a), \end{equation} where
$A_{\mu,\nu}(a):= A_{\mu}(L_a)\cap A_{\nu}(R_a), \quad \mu\in S_L, \ \nu\in S_R.$
 The eigenvalues $(1,1)$ (and possibly  $(0,0)$) play a special role.
We write
\[
S(a):=S_L(a)\times S_R(a)\setminus\{(0,0), (1,1)\},
\]
and  when $a$ is understood,
we write $S$ for $S(a)$.

$A_{\mu,\nu }(a)$, for $(\mu,\nu)\in S(a)$, will be called a {\it two-sided eigenspace}, and its elements  {\it two-sided eigenvectors}.
 An element in $A_{\mu,\nu }(a)$ will specifically be called a
$(\mu,\nu)$-{\it eigenvector} of $a,$ with {\it eigenvalue} $(\mu,\nu)$.
We sometimes write
$A_\mu(a)$ for $A_{\mu,\mu}(a).$
 Thus, $A$~decomposes as a direct sum of subspaces
\begin{equation}
    \label{dec17} A= A_1(a) \oplus A_0(a) \oplus A_S(a),
\end{equation}
where  $A_S(a)=  \oplus  \{A_{\mu,\nu}(a): (\mu,\nu)\in S(a)\}.$    (We also permit the case that  $A_S(a) = 0$; in particular one could have $S(a)=\emptyset$.)

In the presence of \eqref{dec17} we call $a$ an {\it $S$-axis}, and for $y\in A$ we write $$y = y_1 +y_0 +  y_S,$$ where $y_i\in A_i(a)$, $i=0,1,$ and $y_S\in A_S(a),$ so that $y_S=\sum_{(\mu,\nu)\in S}y_{\mu,\nu}.$

 Let
 \begin{align*}
     &S^\circ(a) = \{ (\mu,\nu)\in S(a): A_{\mu,\nu}(a)\ne 0\}.\\
     &S^{\dagger}(a)=\{(\mu,\nu)\in S^{\circ}(a): \mu\ne\nu\}.\\
     &\bar S(a) = \{(\mu,\nu) \in S^\dag(a) :A_{\mu,\nu}^2(a)\ne 0\}.
\end{align*}

\item
  We say that  $a$ is a {\it $\mu$-left axis} when  $S_L \subseteq \{0,1,\mu\}$ with $\mu\in \ff\sminus\{0,1\}$. 
 A  {\it $\nu$-right axis} is defined analogously.

\item \label{dec112}   A {\it $|1|$-axis} is an $S$-axis  $a$ for which $|S| =1$.
In particular when $S = \{(\mu,\nu)\}$ we delete the brackets and call $a$ a $(\mu,\nu)$-axis.

A~{\it $[\mu,\nu]$-axis} is an axis which is either a  $(\mu,\nu)$-axis
  or a $(\nu,\mu)$-axis.


\item An $S$-axis $a$ is  {\it
 commutative} if $(\mu,\nu)\in S^\circ(a)$ implies $\mu=\nu.$ (Note that this implies that $a$ commutes with all elements in $A.$)
 Otherwise we say that   $a$ is {\it
not commutative}. 


\item
An $S$-axis  $a$ satisfies the ``fusion rules''  if the following hold:
\begin{enumerate}
\item $A_{\mu,\nu}(a)A_{\mu',\nu'}(a)=0, \quad \forall (\mu,\nu)\ne (\mu',\nu') \in S.$

\item $A_{\mu,\nu}(a)^2 \subseteq  A_0(a)+  A_{1}(a), \quad \forall (\mu,\nu) \in S.$

\item $ (A_0(a)+  A_{1}(a) ) (A_0(a)+  A_{1}(a) ) \subseteq  A_0(a)+  A_{1}(a) .$

 \item
 $(A_{1}(a)+A_0(a)  A_{\mu,\nu}(a),\,
 A_{\mu,\nu}(a)(A_{1}(a)+A_0(a)) \subseteq A_{\mu,\nu}(a),$ for all all $(\mu,\nu)\in S.$

\end{enumerate}

In particular, the decomposition \eqref{dec17} is a
noncommutative   grading of $A$ as a $\mathbb Z_2$-graded algebra, i.e.,
\begin{equation}\label{dec119}
    A=\overbrace{A_{1}(a) \oplus A_{0}(a)  }^{\text {$(+)$-part}} \oplus
(\overbrace{\oplus _{(\mu,\nu)\in S}\, A_{\mu,\nu}(a)}^{\text {$(-)$-part}}),
\end{equation}

 Note that (8)(a) is vacuous for a $(\mu,\nu)$-axis.

Note that  axes need not satisfy any fusion
rule.


\medskip

\item
 $X\subseteq A$ denotes a set of  axes. $\langle\langle X \rangle\rangle$ denotes the
subalgebra of $A$ generated by $X$. We say that $X$ {\it generates}
$A$ if $A= \langle\langle X \rangle\rangle.$

\item
  In this paper, an  {\it axial algebra} $A$ is an  algebra  generated by a set~$X$ of axes.
$A$ is {\it finitely generated} if $X$ is a finite set, and $A$ is
{\it $n$-generated} if $|X|=n.$




\item An axis $a$ is {\it weakly primitive} if $A_1(a) = \ff a.$
 An axis $a$ is  {\it left primitive}
if $A_1(L_a) = \ff a.$ Similarly for  {\it right primitive}. An axis
which is both left and right primitive is called {\it
primitive}.

\item For a given weakly primitive axis $a\in A,$ and any  $y\in A$, we write $y_1 = \ga_y a.$
\end{enumerate}
\end{Defsnot}

Thus, a $(0,1)$-axis is left primitive, and a   $(1,0)$-axis is right primitive.

\subsection{The main results}$ $

Our main results are on algebras $A=\lan\lan a,b\ran\ran,$ where $a,b$ are idempotents.  In \S2, we assume that $\dim A=2,$ in which case we can get basically get complete information about $A.$  Let us state   the two main results, characterizing when $A$ is a  weak PAJ.

 In this subsection we assume that $a$ is a weakly primitive axis satisfying the fusion rules such that $\dim A_0(a)\le 1,$ and that $b$ is an axis satisfying the fusion rules.

\begin{thmB}[Theorems~\ref{yrho}, \ref{comm1}, and \ref{gr+xi=1}]
Assume that $A$ is commutative.  Then $\dim A\le 4,$ and  one of the following holds.
\begin{enumerate}\eroman
\item
$A$ is an HRS algebra (defined below in Remark~\ref{rem03}.)

\item
$\dim A=3,\ \dim A_1(b)=2,$ and either
\begin{enumerate}
\item
$\dim A_{\half,\half}(b)=1,$ and $A$ is as in Example~\ref{4dim2}.
\item
$\ \dim A_0(b)=1,$ and $A$  is as in Example \ref{3dim01}.
\end{enumerate}

\item
$b$ is weakly primitive and $\dim A_0(b)=1,$   $\dim A=4,$ and $A$ is as in Example \ref{egc4}.

\item  $\dim A_1(b)=\dim A_0(b)=2,$ and $A$ is as in Example \ref{4dim}.

\end{enumerate}
\end{thmB}

\begin{thmA}[Theorem~\ref{Yoav1}]\label{mnon}
Assume that $A$ is not commutative.   Then one of the following holds.
\begin{enumerate}\eroman
\item
$A_0(a)=0,$ and $A$ is as in Example \ref{b0zeroa}(2).

\item
$\dim A_0(a)=1,$ and either
\begin{enumerate}
\item
$\dim A_1(b)=2,$ and $A$ is as in Example \ref{noncom2}.

\item
$b$ is weakly primitive and  $A$ is either as in Example \ref{E1} or  as in Example \ref{egnon}.
\end{enumerate}
\end{enumerate}
\end{thmA}

 Thus, the only  HRS algebras which extend to    PAJ's which are not commutative are   $\ff,$ $\ff \times \ff,$ $B(\half,1)$ and its $2$-dimensional quotient, and $B(2,1).$

\subsection{ Description of the main auxilary results.
}$ $
\medskip

Next, we   describe the auxiliary results leading to the main theorems, which also are of independent interest since they permit us to study axes in an axial algebra which do not satisfy the fusion rules.
\medskip

\noindent
{\bf Lemma \ref{pr01} and Theorem~\ref{dime2}}. 
{\it The general case when $\dim A=2.$}
\medskip

 From now on let   $a$ be a weakly primitive axis satisfying the fusion rules. We build up our hypotheses on the idempotent $b$ step by step.
 \medskip

\noindent
{\bf Theorem~\ref{Ser02}} {\it  Necessary and sufficient criteria for an element    $b\in A$ to be an idempotent that satisfies $(ba)b = b(ab).$  }

\medskip
Next, we also assume that $b$  is an axis.

\medskip

\noindent
{\bf Theorem~\ref{Ser2}} {\it Necessary and sufficient criteria for    $b\in A$  to be an axis.}

\medskip

\noindent
{\bf Theorem~\ref{b0zero}}  {\it Characterization of the two classes for which $b$ is an axis, and $A_0(a)=0$.}
\medskip

{\it We further assume that $a$ is a weakly primitive axis satisfying the fusion rules,such that $\dim A_0(a)\le 1$}.

\noindent
{\bf Theorem~\ref{Ser04}}
{\it Characterization of the eigenvectors of $b$ contained  in the subspace $ V(a) = \{ y\in A:  \ga_y= 0 = y_0\},$ and their eigenvalues, under the hypothesis that $b$ is an axis.}
\medskip

\noindent
{\bf Theorem~\ref{eigcom}}
{\it Determination of the $(\gr,\gr)$-eigenvalues of $b,$ for  $\gr\ne\half,$ and the eigenspace $A_{\gr,\gr}(b),$ under the hypothesis that $b$ is an axis.}

 Finally, we also assume that $b$ is an axis satisfying the fusion rules. We have a full characterization of the PAJ's for which $\dim A_0(a) =1:$

\subsection{Preliminary observations}

\begin{remark}\label{b}
    Note that if $b\in A$ is an idempotent, and $T\subset \ff\times\ff$ is a finite subset such that $A=\oplus_{(\rho,\xi)\in T} A_{\rho,\xi}(b),$ then $b$ is an axis, i.e., $L_bR_b=R_bL_b.$  Indeed, it is enough to check this on an element $x\in A_{\rho,\xi}(b),$  and for such $x$ this is obvious.
\end{remark}

\begin{lemma}\label{dec7} For any $S$-axis $a$ and any $y \in A$,
\begin{enumerate}\eroman
\item $y_0, y_1, y_{\mu,\nu}\in \langle\langle a,y
\rangle\rangle$, for each $(\mu,\nu)$ in $S.$

 \item If $S^\circ = \emptyset,$    then $\dim  \langle\langle a,b\rangle\rangle =2.$ 
\end{enumerate}
\end{lemma}
\begin{proof}  (i) Write $y = y_1 + y_0 + \sum y_{\mu,\nu}.$ Then $a(a(\dots (ay))) = y_1 + \sum \mu^i y_{\mu,\nu}$ and $((ya)\dots a)a = y_1 + \sum \nu^j y_{\mu,\nu},$   so we can solve for all $y_{\mu,\nu}$ except for $y_{0,1}$ and $y_{1, 0}$ in terms of $y, ay, a(ay), \dots,  ya, (ya)a,\dots .$
Then we get $y_1$ from $aya$, and $y_0$ from $aya -y.$ Finally we get $y_{0,1}$ and $y_{1, 0}$  from  $a y -y$ and $ya -y.$

(ii)  $b = \ga_b a +b_0$. Thus $ab = \ga_b a =ba$, proving the  assertion.
\end{proof}

\begin{remark} \label{Miyfhol0} In \cite{RoSe3} we showed (with considerable
effort) that in any  algebra, which is not commutative, generated by two primitive axes $a$
and $b$ each having at most one left and one right eigenvalue other than 0 and 1, satisfying the fusion rules, that $a$ and $b$ are $[\mu,\nu]$-axes
  with $\mu+\nu=1$, by \cite[Theorem~2.5, Proposition~2.9, and
Example~2.6]{RoSe1}.
\end{remark}

\begin{lemma}\label{rempr}  Suppose $a$ is an $S$-axis satisfying the fusion rules.
For any subset $S' $ of $S,$
$$A_1(a) + A_0(a) +   \sum _{ (\mu,\nu)\in S'} A_{\mu,\nu}(a)$$ is a subalgebra of $A,$ in which the axis~$a$  satisfies the fusion rules.
\end{lemma}
\begin{proof}
    By the fusion rules, all products of components are in $A_1(a) + A_0(a)$ or in $ A_{\mu,\nu}(a)$ for some $ (\mu,\nu)\in S'$. The last assertion is a fortiori.
\end{proof}

We now are ready for notions about an axial algebra.

\begin{Def} An axial algebra $A =  \langle\langle X \rangle\rangle$
is
 {\it  weakly primitive} (resp.~{\it left primitive}, resp.~{\it right primitive}, resp.~{\it primitive})  if each axis of $X$ is
weakly primitive (resp.~left primitive, resp.~right primitive, resp.~primitive).

 A  {\it weak PAJ} is a
weakly primitive axial algebra generated by axes satisfying the fusion rules.

 A  {\it  PAJ} is a
 primitive axial algebra generated by axes satisfying the fusion rules.

$A$ is   of  {\it type} $S$, if each axis of $X$ is an  $S$-axis.


As a special case, when  $S$ is  deleted, it means  $S= \{(\mu,\nu)\}$, in accordance with~\cite{RoSe4}. A~{\it CPAJ} is a commutative
PAJ.
\end{Def}

\begin{remark}\label{rem03}
Some well-known classes of algebras are axial algebras.$ $
\begin{enumerate}
    \item Jordan algebras generated by primitive idempotents are CPAJ's of type $\half$, in the sense of \cite{HRS}.
 \item \cite{RoSe3} characterized all  CPAJ's $\lan\lan a,b\ran\ran$   with $|S(a)|, |S(b)|\le 1$, showing that they   all are in the list of  \cite[Theorem~1.1]{HRS}.
  We shall call these axial algebras {\it HRS algebras}.
 \item For another noncommutative (but associative) example
recall that a \textbf{band}  is a semigroup all of whose elements are idempotents. Thus, its semigroup algebra $B$ is an axial algebra, but only with left and right eigenvalues $0$ and $1$.
On the other hand, the fusion rules fail for nontrivial band algebras, so to get a PAJ   we will have to modify the multiplication table, cf.~\Eref{side}.
\end{enumerate}
\end{remark}


The theory of PAJ's with one left and one right eigenvalue $\ne 0,1$ was expounded in \cite{RoSe4}.
 One major
objective of this paper is to extend these results to weakly
primitive axial algebras with at least one axis with several eigenvalues, satisfying   more general fusion rules.


\subsection{Seress' Lemma}$ $

Let us strengthen some of the fusion rules.

\begin{Def}\label{Se} An  $S$-axis $a$  is    \textbf{left Seress}
 if $$A_{0}(a)A_1(a) \subseteq A_{0}(a),\qquad
 A_1(a) A_{0}(a) =0, \qquad A_1(a)^2 = A_1(a),$$
$$
A_{\mu,\nu}(a)(A_1(a)+ A_{0}(a))\subseteq A_{\mu,\nu}(a), \quad \forall (\mu,\nu)\in S.$$

Right Seress is defined analogously. Seress means left Seress and right Seress.
\end{Def}

\begin{remark}
  Any  weakly primitive  $S$-axis  satisfying the \fun\ fusion rules  is Seress, since $A_1(a)=\ff a$.
\end{remark}

\begin{lemma}[{Seress' lemma}]\label{Seress}
  Any left Seress  $S$-axis $a$   satisfies
\begin{equation}
    \label{Serr} a(yz) = (ay)z + a(y_0z_0)
\end{equation} for any $ y \in A$ and $z\in \ff a +
A_0(a).$ If $a$ satisfies the fusion rules then $$a(yz) +\ff a = (ay)z + \ff a. $$
\end{lemma}
\begin{proof}$ (ay)z =  (y_1+\sum \mu y_{\mu,\nu})z = y_1z_1 +\sum  \mu  y _{\mu,\nu}z_1+\sum \mu  y _{\mu,\nu}z_0,$
 whereas \begin{equation} \begin{aligned} a(yz) & = a(y_1z_1 + y_1z_0 + y_0 z_1 +y_0z_0 +
\sum \mu y_{\mu,\nu}z_1 +\mu \sum  y_{\mu,\nu} z_0)\\ & = a(y_1z_1) + a(y_0z_0)+\sum  \mu
y_{\mu,\nu}z_1 + \sum \mu y_{\mu,\nu} z_0  .\end{aligned}\end{equation}
Note that $a(y_1z_1) = y_1z_1$ since $y_1z_1 \in A_1(a).$

The last assertion follows since $a(y_0z_0) \in a(\ff a + A_0(a)) \subseteq \ff a.$
 \end{proof}

Equation~\eqref{Serr} leaves us with the term $a(y_0z_0) $, and next we search for a
situation where it is 0.

\subsection{Frobenius forms}$ $

The requirement $A_0(a)^2 \subseteq A_0(a) $ is conspicuously missing in Definition~\ref{Se}. It can be proved in the following situation.

\begin{Def} A non-zero bilinear form
$(\cdot\, ,\, \cdot)$ on an algebra $A$ is called {\bf Frobenius} if
the form associates with the algebra product, that is,
\[(x,yz)=(xy,z)\]
for all $x,y,z\in A$. \end{Def}

It was shown in \cite[Theorem 6.3]{RoSe4} that every  PAJ generated by 1-axes  has a
symmetric Frobenius form.  We shall extend this to weak PAJ's,
in the sequel to this paper.

The relevance to Seress' lemma is:

\begin{lemma}\label{Ser3} Suppose $A$ has a symmetric
 Frobenius form $(\phantom{w},\phantom{w} )$. For any
weakly primitive axis~$a$,

\begin{enumerate}\eroman
\item
 $(a,y) = (a,y_1)$ for all $y \in A.$

\item $A_0(a)^2 \subseteq  A_0(a)$ if $a$ satisfies the fusion rule of Definition~\ref{not1}(8)(c) and $(a,a)\ne 0.$ In this case, $ a(yz) = (ay)z $
  for any $ y \in A$ and $z\in \ff a +
A_0(a).$\end{enumerate} \end{lemma}
\begin{proof}(i)  For $z \in  A_{\mu,\nu}(a),$
$$\mu(a,z) = (a,az) = (a^2,z) =
(a,z).$$
Hence $(a,z) = 0$ unless $\mu = 1$.    Similarly,  $(a,z) = (z,a)=0$ unless $\nu = 1$.
We cannot have  $\mu = \nu = 1.$ So $(a,z) = 0$.
  Also, for $z\in A_0(a),$ we have $(a,z) = (a,az) =0.$
We conclude that
$(a,y) = 0 + (a, y_1) .$

(ii) For $z_1,z_2 \in  A_{0}(a),$ we have $z_1z_2 \in A_1(a)
+A_0(a).$ But $(a,z_1 z_2) = (az_1,z_2) = 0$. Since $A_1(a) = \ff a,$ $z_1z_2 \in A_0(a)$
by (i). The last assertion then follows from  \eqref{Serr}.
\end{proof}

\begin{ques}\label{exten} Is the requirement $(a,a)\ne 0$ in Lemma~\ref{Ser3} superfluous?
\end{ques}

\begin{example}
    Just as axial algebras of Jordan type were inspired by Jordan algebras, one  has {\it noncommutative Jordan algebras}, introduced by Albert. Given an algebra $A$ over a field $F$ of characteristic $\ne 2,$ define $A^+$ to have the new multiplication $\circ$ given by $x\circ y = \half( xy+yx)$). Albert  \cite[p.~575]{A} called $A$ {\it Jordan admissible} if $A^+ $ is a Jordan algebra.
Flexible, Jordan admissible algebras $A$,   called   {\it noncommutative Jordan algebras,}
    were studied by Schafer \cite{S}, who proved that they have a symmetric Frobenius form, for which  where $(x,x)\ne 0$ for any idempotent $x$ and $(z,z)=0$ for all nilpotent $z$ (the old terminology being ``trace-admissible'').

      If $J$ is noncommutative Jordan then any $(\mu,\nu)$-eigenvector of $J$ becomes a $\half(\mu+\nu)$-eigenvector of $J^+$, so $\mu+\nu =1,$ since the only nontrivial eigenvalue in a Jordan algebra is $\half.$ Note that this is the case for Example~\ref{egnon1} but not for
      Example~\ref{egnon} since 2 is an eigenvalue. \end{example}
\section{Axial algebras of dimension  2}$ $

Algebras of dimension 2 are such a special case that we can examine them thoroughly, to see how close
they are to being axial, and weak PAJ's. We have some phenomena  of weakly primitive axes, even for
dimension 2.

\begin{lemma}[The general case for dimension 2]\label{pr01} Let  $A$ be an algebra generated by
nontrivial idempotents $a,b$ satisfying $ab = \rho a + \xi b$, and $ba = \xi' a+ \rho'b $.

\begin{enumerate}\eroman
\item \begin{enumerate}
    \item $a$ has left eigenvalue $\xi$ and right eigenvalue $\rho'$.

If $\xi\ne 1,$ then the left decomposition of $b$ with respect to $a$ is
 $$b =-\frac \rho {\xi
-1} a + \left(b + \frac \rho {\xi -1} a \right).$$
(In particular $b_0 = 0.$)

 \item Symmetrically, $b$ has left eigenvalue $\xi'$ and  right eigenvalue $\rho$.
\end{enumerate}

\item   \begin{enumerate}
    \item  $a$ is left semisimple,
unless $\xi = 1$ and $\rho\ne 0.$

 \item  Similarly, the idempotent $a$ is right semisimple, unless $\rho' = 1$ and $\xi'\ne 0.$
\end{enumerate}

 \item \begin{enumerate}
     \item A semisimple idempotent $a$ is an axis, if and only if the following equation holds:
  \begin{equation}\label{seer3}\rho(1-\rho') = \xi'(1 - \xi).\end{equation}

  \item A semisimple idempotent  $b$ is an axis iff
\begin{equation}\label{seer21}\rho'(1-\rho) = \xi(1 - \xi').\end{equation}

 \item  $a$ and $b$ are axes,
 if and only if  either
 \begin{itemize}
     \item $\rho =\xi'$ and $\rho' = \xi,$ so $A$ is commutative, or
     \item $\rho +  \xi = 1 = \rho'+\xi',$ so $A$ is not commutative unless $\rho =  \xi = \rho'=\xi'= \half.$
     \end{itemize}
       \item  We cannot have $\xi = \rho' =1,$  or
 $\xi' = \rho =1.$
  \end{enumerate}

  These solutions are significant in view of \cite[Theorem
B]{RoSe1}.

\item \begin{enumerate}
    \item  In the commutative case, the axis $a$ satisfies the fusion rules if and only if $\xi-1+2\rho\xi=0.$ Similarly, the axis $b$ satisfies the fusion rules if and only if $\rho-1+2\rho\xi=0.$ Thus if both $a$ and~$b$ satisfy the fusion rules, then
$\xi =\rho \in \{\half, -1\}.$

Consequently the only weakly primitive 2-dimensional axial algebras generated by two axes satisfying the fusion rules are the well known HRS algebras    ($\rho = \xi \in \{\half ,-1\}$).

\item  In the noncommutative case, if both $a,b$ are axes, then $A_{\xi,\rho'}(a) = \ff (b-a)= A_{\xi',\rho}(b);$  hence $a$ and $b$ are weakly primitive and
Seress. The axis $a$ satisfies the fusion rules, if
and only if $\xi = \xi'$ (and thus $\rho = \rho'$). This matches
\cite[Theorem~2]{RoSe1}, which says that $A$ is a PAJ
 if and only if $\xi = \xi'$ and $\rho = \rho'$ with  $\xi + \rho = 1,$  as well as $\xi \notin
\{0,1\}.$ In this case $A$ has no other
nontrivial idempotents.

 When $\xi \ne \xi'$, the only
nontrivial idempotent other than the axes $a$ and $b$ is $e=\frac
1{\xi'-\xi}(a-b)$, which is left and right semisimple, but not an axis.
\end{enumerate}

\end{enumerate}
\end{lemma}
\begin{proof}
    (i)
We need
to get a left eigenvalue $\xi$ for $a$.  If $\xi = 1$ then $a$ is a $\xi$-eigenvector, so we may assume that $\xi \ne 1.$ If
$\rho = 0,$ then $b$ is a $\xi$-eigenvector of $a.$ So we may assume that $\rho \ne 0.$ We solve
$$ \xi ( a + \gb b)   =
a ( a + \gb b) =a +\gb( \rho a+ \xi b) = (1 + \gb \rho) a  + \xi \gb
b.$$ First solving $1 + \gb \rho = \xi$ yields
  \begin{equation}\label{seer0}
  \gb = \frac{\xi-1}\rho.
\end{equation}
 Thus $a$ has  left eigenvalue $\xi$, and, by symmetry, $a$ has right eigenvalue $\rho'$.

 (ii) By (i), $a$ is left semisimple unless $\xi = 1.$ Assume $\xi=1.$ If $\rho = 0$ then  $ab= b$, so $a$ is the left identity element of $A$, which
is left semisimple, so assume that $\rho\ne 0.$
Then $ab = \rho a + b,$ implying $(L_a-1)b =ab-b=   \rho a,$ and
$$ (L_a-1)^2b =a(ab-b)-(ab-b)  = \rho  (a-a) = 0,$$ so $b$ is a generalized
eigenvector  of multiplicity 2.

(iii)(a)
We need $(ab)a = a(ba).$
We solve
$$(\rho a + \xi b)a = a( \xi' a
+ \rho'b) ,$$ i.e.,
$$\rho a + \xi ( \xi' a
+ \rho'b)  =  \xi' a + \rho' (\rho a + \xi b),$$ or
\begin{equation}\label{seer2}\rho + \xi  \xi' = \xi'  + \rho' \rho ,\end{equation} yielding \eqref{seer3}.

\indent \indent (b) Reverse the roles of
$a$ and $b$.

\indent \indent (c)
  Indeed if either condition  holds then clearly  \eqref{seer3} and  \eqref{seer21} hold. Conversely, assume that   \eqref{seer3} and  \eqref{seer21} hold. Subtracting \eqref{seer21} from \eqref{seer3} yields $$\rho -\rho'
= \xi' - \xi,$$  or $\rho' = \xi +
\rho - \xi'.$ Plugging into \eqref{seer3} we get $\rho( \xi' + 1- \xi -\rho) = \xi'(1 - \xi)$,
which says
$$\rho (1 -
\xi -\rho) =  \xi'(1 - \xi -\rho).$$ In other words,

  \begin{equation}\label{seer4} \xi' = \rho \text{ or }\xi + \rho = 1.
  \end{equation}

  But if $\xi'=\rho $ then \eqref{seer3} or \eqref{seer21}    shows that $\rho' = \xi .$ Likewise if $\rho = 1-\xi \ne 0$ then  \eqref{seer3} and \eqref{seer21} shows $\rho' = 1-\xi'.$

(d)  If  $\xi = \rho' =1,$ then  $\rho =\xi' = 0,$ so $ab=ba=b,$ contrary to $a\ne 1.$
\medskip

(iv)(a) Since $A$ is commutative,
\[
\ff a \ni (\rho a + ( \xi-1)b)^2 = \rho^2 a +  ( \xi-1)^2 b + 2\rho   ( \xi-1)(\rho a+ \xi b)
\]
so
\begin{equation}\label{7i}
(\xi-1)^2+2\rho(\xi-1)\xi=0.
\end{equation}
If $\xi=1,$ then $\rho=0,$ but then $a$ is the identity of $A,$ contrary to hypothesis.  So $\xi-1+2\rho\xi=0.$
Since $\xi'=\rho$ and $\rho'=\xi,$ if $b$ satisfies the fusion rules $\rho-1+2\rho\xi=0.$  Thus if both $a,b$ satisfy the fusion rules, $\xi=\rho.$  Then we get $\rho-1+2\rho^2=0,$ so $\rho\in\{-1,\half\}.$

 \medskip

\indent \indent (b)
  $\xi+\rho =1 = \xi' + \rho',$   in view of (iii)(c). To show that $A_{\xi,\rho'}(a) = \ff (b-a)= A_{\xi',\rho}(b),$ note that
 $a(b-a) = \rho a + \xi b - a = \xi (b-a)$, and similarly
$(b-a)a = \rho'(b-a).$ Then by (iii)(c), $A_1(a) = \ff a,$ and $A_0(a) =  0,$ and
likewise for $b.$  Hence $a$ and $b$ are weakly primitive. For the fusion rules, we need to consider $(b- a)^2.$  More generally we compute $(b-\gb
a)^2$ for $\gb \in \ff.$

$ab+ba = \rho a + \xi b +  \xi' a + \rho' b = (\xi' + 1 -\xi)a + (\xi+
1-\xi')b$, so
  \begin{equation}\label{seer}
\begin{aligned} (b-\gb a)^2 & = b+ \gb^2 a -\gb(ab+ba) \\ &=  (1-\gb(1-\xi'+ \xi))b -\gb(\xi ' -
\xi +1 -\gb )a . \end{aligned}
\end{equation}
Taking $\gb =1,$ yields $(b- a)^2 = (\xi - \xi')(b- a)\in \ff(b- a), $ which is
$0$ if and only if $\xi = \xi'$.

  $(b-\gb a)^2 \in \ff (b-\gb a)$ only when the two terms on
the right of \eqref{seer}, $1-\gb(1+\xi- \xi')$ and $\xi'  - \xi +1 -\gb$, are
equal. Thus
 $$ 1 -\gb +\gb(\xi'-\xi) = 1 -\gb +(\xi'  - \xi),$$ so  \begin{equation}\label{seer1}(\gb -1)(\xi'  -
\xi) =0.\end{equation}

When  $\xi' \ne \xi,$ then $\gb =1$, and the only
nontrivial idempotent other than $a$ and $b$ is $e=\frac
1{\xi'-\xi}(a-b)$.

Let us check the left and right eigenvectors for $e$.
$$(a-b)(a+b) = a-b +(\rho a + \xi b)- (\xi' a
+ \rho'b) = (\rho +\rho')(a -  b),$$

so $(a-b)((a+b) - \frac{\rho +\rho'}{\xi'-\xi}(a-b)) =0,$ i.e., $
(a+b) - \frac{\rho +\rho'}{\xi'-\xi}(a-b)$ is a right 0-eigenvector
of $e$, and symmetrically $$(a+b)(a-b) =  a-b -(\rho a + \xi b)+
(\xi' a + \rho'b) = (\xi +\xi')(a -  b),$$ so $(a+b) - \frac{\xi
+\xi'}{\xi'-\xi}(a-b)$ is a left 0-eigenvector of $e$. Hence $e$ is left and right semisimple.

Next, $(a+b)((a-b)(a+b)) = (\rho +\rho')(a -  b)^2$  whereas
$$((a+b)(a-b))(a+b) = (\xi +\xi')(a -  b)^2.$$ Thus, $e$ is  an axis if and only
if $\rho +\rho' = \xi +\xi' = 2-(\rho +\rho'),$ so $\rho +\rho'=1.$ But then $\xi = \rho'$, so $A$ is commutative, a contradiction.
\end{proof}

We can use Lemma~\ref{pr01} to study a  nice infinite class of primitive axial algebras (not necessarily of dimension 2).

\begin{example}\label{pr1} Define
$\uu_E(\mu,\nu)$ as the algebra generated by a set of idempotents $E
= \{ x_i: i \in I\}$, satisfying $x_ix_j = \nu x_i + \mu x_j$ for
all $i\ne j \in I,$ and $\mu \ne \nu$ or $\mu = \nu \notin \{0,1\}.$

\begin{enumerate}\eroman

\item $E$ clearly spans $\uu_E(\mu,\nu)$.

\item Let $a = x_1$. Then \begin{equation}\label{eq7} a  \sum \gc_j x_j = \left(\gc_1+ \mu (\sum _{j\ne 1}\gc_j)\right) a + \nu \sum _{j\ne 1} \gc_j x_j,
\end{equation}
       for $\gc_j\in \ff.$ Hence, we have the following possible left eigenvalues for~$a$:
\begin{enumerate}
    \item  $1$, with  left eigenspace   $\ff a$,   or       \item        $\nu,$  with
     left  eigenspace $\{ \sum \gc_j x_j : \mu (\sum_{j\ne 1} \gc_j) = (\nu-1)\gc_1\}$ (which has codimension 1 in $A$) when $\mu\ne 0$, and $\sum_{j\ne 1} \ff x_j$ when $\mu= 0$, $\nu \ne 1$.
\end{enumerate}
When $\mu =0$ and $\nu = 1$, then $a$ is the left identity.

Symmetrically, the right eigenvalues for~$a$ are $1$ and $\mu.$
\item
It follows that each idempotent $x_i$ is semisimple. By Lemma~\ref{pr01}(iii), each $x_i$ is an axis if and only if $\mu =\nu$ or $\mu+\nu = 1$.

\item If 1 is not a left eigenvalue then there is only one left eigenvalue, so for $a$ to be semisimple we must have $ax_i =0$
for each $i$. Likewise if~1 is not a right eigenvalue then $x_i a=0$
for each $i$. In other words, for $\mu,\nu \ne 0$ we have $\ga_{x_j}\ne 0$ for each $i$.

 \item (The commutative   case) Assume that  $\mu = \nu.$ The base of the $\mu$-eigenspace of $a$ is $\{\mu a + (\mu-1) x_j : j >1\}$.

 For  $|I|=2,$ when  $\mu\in \{\half, -1$\}, we have the HRS algebras.

 For $|I|>2$ and for nonzero $j\ne j',$
\begin{equation}
    \begin{aligned}
    (\mu  a + &(\mu-1) x_j)(\mu a + (\mu-1) x_j')\\ & =  \mu ^2 a +\mu(\mu-1) ax_j +\mu(\mu-1) ax_{j'} +(\mu-1)^2 x_{j}x_{j'} \\ &=\mu^2(2\mu-1) a+ \mu^2( \mu-1)(x_j+x_{j'})+\mu ( \mu-1)^2(x_j+x_{j'}) \\ & = (1-2\mu)a +\mu ( \mu-1)(2\mu-1)(x_j+x_{j'}).
    \end{aligned}
\end{equation}

Thus  the fusion rules are satisfied only for $\mu= \half,$ since the cases $ \mu = \nu\in \{ 0,  1\}$ have been excluded.
 For all other $\mu$ we have primitive
  axes which fail the fusion rules.


\item (The noncommutative   case) Recall from (iii) when $E$ is a set of axes that $\mu+\nu = 1$. \begin{enumerate}
    \item For $\mu \notin \{0, 1\},$
 in \cite[Lemma~5.5]{RoSe4} $\uu_E(\mu,\nu)$  is denoted as $\uu_E(\mu)$,   and
  is shown there to be a
 PAJ.

\item We are left with $\mu= 0$ or $\nu = 0.$
\begin{itemize}
    \item For $\mu = 0$ and $\nu = 1$ we get $x_i x_j =  x_i,$ and $E$ is a
semigroup called the {\it left-zero band}  \cite{W}. Thus $A:=
\uu_E(0,1)$ is the semigroup algebra of the  left-zero band
 $E$. Hence $A$ is associative.   $A_1({x_i})= \ff x_i$. Also $Z:= A_{0,1}(x_i) = \sum_j \ff
 (x_i-x_j)$ is independent of~$i$, and so $A_0(x_i)=0$.

$x_iZ =0,$ for each $i$, implies $AZ =0.$  Hence  $Z$  is a square zero
ideal. The fusion rules
 clearly hold, and every $x_i$ is a left, but not right, primitive $(0,1)$-axis. $A$ is a left primitive axial algebra of
type $(0,1)$.

Any element of the form $x_i+z$ for $z \in Z$ is idempotent,
satisfying $x_j(x_i+z) = x_jx_i+x_jz = x_j$ and $$(x_i+z)x_j =
 x_ix_j+ z x_j= x_i+z,$$ so $x_i+z$ is an $(0,1)$-axis.

Every nonzero idempotent has the form $x_i+z$ for $z \in Z$, since
$(\gb x_i +z)^2 = \gb^2 x_i +\gb z,$ implying $\gb=1.$

\item
Symmetrically, for $\mu = 1$ and  $\nu = 0$ we get $x_i x_j =  x_j,$
which   defines the {\it right-zero band} semigroup, and
  $\uu_E(0,1)$ is the semigroup algebra of the right
zero-band semigroup~$E$, where each nontrivial idempotent is a right
primitive $(0,1)$-axis satisfying the fusion rules.
\end{itemize}
\end{enumerate}
\end{enumerate}
 \end{example}

\begin{remark} We get nothing new by mixing these relations in Example~\ref{pr1}, taking the algebra generated by a set of idempotents $E
= \{ x_i: i \in I\}$, satisfying $x_ix_j = \mu_{i,j} x_i + \nu_{i,j} x_j$ for
all $i,j \in I,$ where $\mu_{i,j},\nu _{i,j}\in \{0, 1\}$ with $\mu_{i,j}+\nu _{i,j} =1.$

 For example, take $ab =b$ and $ba = b.$ Then $a$ and $b$ are
Seress, with $b$ primitive, but $a$ is the unit element. Note that
this algebra is really $\ff b \times \ff (a-b),$ and one can check that this kind of modification
always yields ``decomposable'' presentations of this sort.
\end{remark}

Let us summarize the general situation for dimension 2, reformulating Lemma~\ref{pr01}.

\begin{thm}\label{dime2}$ $\eroman
\begin{enumerate}
    \item Suppose $A=  \langle\langle a,b\rangle\rangle $ has dimension $2$, where $a,b$ are nontrivial axes.
Then $a$ and $b$ are weakly primitive axes, and writing $ab = \rho a
+ \xi b$ and $ba = \xi' a + \rho' b$,
the
eigenvalues of $a$ are $1$ and $(\xi,\rho')$ and the eigenvalues of $b$ are $1$ and $(\xi',\rho).$

  The idempotents of $A $ are $b$, and all elements of the form
\begin{enumerate}
    \item $\frac 1{1 +
\gb }(a +\gb b)$  when $\xi '+
\rho =
 \rho' + \xi =1$ and  $\gb \ne
-1$, or

\item $
    \frac 1{1 +
\gb(\xi '+ \rho)}(a +\gb b),$
  when $\xi' + \rho \ne 1$ and $\gb = \frac{\rho'+
 \xi-1}{\xi '+ \rho -1}$.
\end{enumerate}

\item Either

\begin{enumerate}\eroman

\item (The commutative case)
$\xi' =\rho$ and $ \rho' = \xi  $. In this   $a$ has eigenvalues   1 and  $\xi,$  and $b$ has eigenvalues  1 and  $\rho.$
The axes satisfy  the fusion rules if and only if $\xi = \rho \in \{ 0, \half, -1\}.$ (If $\xi = \rho= 0$ then $ab = ba =0,$  i.e., $A\cong \ff \times \ff.$)

\smallskip
or
\smallskip

\item  (The noncommutative case)
$\xi + \rho =1$ and $\xi' + \rho' = 1,$  and $\rho, \rho'\ne \half$.

In  case $(\xi,\rho'), (\xi',\rho)\notin\{(1,1)\},$ the axis $a$ satisfies the fusion rules if and only if $\xi = \xi'$ (implying
$\rho = \rho'$), in which case $b$ also satisfies the fusion rules, and $A\cong   \uu_E(\xi,\rho)$.

In case $(\xi,\rho)=(1,0),$ then $A$ is a left-zero band and in case $(\xi,\rho)=(0,1),$
then $A$ is a right-zero band (see Example \ref{pr1}(vi)).

 \end{enumerate}

\end{enumerate}
\end{thm}

\begin{proof}

(i) The eigenvalues of~$a$ and $b$
come from Lemma~\ref{pr01}(i).  $$(a +\gb b)^2 =  a + \gb^2 b + \gb (ab+ba) = (1 + \gb(\xi' + \rho))a + \gb(\gb +\rho'+ \xi)b. $$
 We want this to be proportional to $a + \gb b$,  i.e., $$1 + \gb(\xi' + \rho) =
\gb +\rho'+
 \xi,$$
or $$\gb ({\xi' + \rho-1})=  {\rho'+
 \xi-1} .$$  Both sides are 0 if and only if $\xi' + \rho = 1$ and either $\gb =0$ or \begin{equation}\label{ugen9}\xi' + \rho =
 \rho' + \xi =1.\end{equation} In case $\gb =0$ we get the idempotent $a.$
In the case of \eqref{ugen9}, dividing by $1+\gb,$ we get an idempotent as in (a) for any value of
 $\gb\ne -1$. Otherwise we get
  $$\gb = \frac{\rho'+
 \xi-1}{\xi '+ \rho -1},$$ yielding (b).

(ii)(a) Of course $A$ is commutative; by (i), the eigenvalues of $a$ are  $1$ and~$\xi$, and the eigenvalues of $b$ are  $1$ and $\rho$.  By Lemma~\ref{pr01}(iv), if $a$ satisfies the fusion rules,
then $\xi=\xi'=\rho=\rho',$ in which case $b$  also  satisfies the fusion rules.  So $A$ is a
CPAJ, so $\xi\in\{0,\half,-1\},$ by \cite{HRS}.

\indent \indent (b)  Since Case (iia) does not occur, we must have $\xi\ne\rho'$ and $\xi'\ne\rho,$ by equations \eqref{seer3} and \eqref{seer21}.
Then
the first assertion holds by Lemma~\ref{pr01}(iii)(c).  If $(\xi,\rho')=(1,1),$
then $a$ is the identity of $A,$ a contradiction.  Similarly for~$(\xi',\rho).$ The  axes $a$ and $b$ satisfy the fusion rules by Lemma~\ref{pr01}(iv)(b).
\end{proof}

\begin{cor}\label{ugen1} Hypotheses as in Theorem~\ref{dime2}, there are infinitely many idempotents of $A=  \langle\langle a,b
\rangle\rangle $, if and only if $A \cong  \uu_E(\mu,\nu)$ for  $\mu +
\nu = 1.$ Then the idempotents are $\{(1-\gc) a + \gc_0b: \gb \in
\ff\}$, and these are all weakly primitive $(\mu, \nu)$-axes   satisfying the fusion rules,  and  are  primitive axes when $\mu \ne 0,1.$
\end{cor}
\begin{proof} For $A$ commutative  there are only finitely many idempotents unless $\mu = \nu = \half,$
by \cite[Lemma~5.3]{RoSe4}, in which case we have  $A \cong  \uu_E(\half,\half)$.
For $A$ not commutative   there are infinitely many idempotents if and only if  $\xi '+
\rho =
 \rho' + \xi =1$,
 by Theorem~\ref{dime2},
 i.e., $\mu ' =\mu $ and $\nu ' =\nu$.
\end{proof}
\begin{remark}
    We have recovered   the known  PAJ's when $\mu \ne 0,1$. Otherwise we have the left-zero or
right-zero band algebra.
\end{remark}

\begin{hyp1}
Throughout the remainder of this paper we assume that $a$ is a weakly primitive axis satisfying the fusion rules with $\dim A_0(a)\le 1.$
\end{hyp1}

\section{2-generated axial algebras with a weakly primitive axis}$ $

Since the case $\dim A = 2$ was already handled in Theorem~\ref{dime2}, we assume from now on that $\dim A \ge 3.$
 Here is an easy example of a commutative axial algebra in which the axis
$a$  is not weakly primitive.

\begin{example} \label{try01} Let  $ A$
be the 3-dimensional commutative algebra spanned by idempotents
$a,b$ and an element~$y,$   satisfying the relations
\[
ab  = y =  ay   =by,\quad y^2=0.
\]

Clearly $A_1(a) = \ff a + \ff y,$ $A_0(a) = \ff (b-y),$ and $(b-y)^2 =
b -2y,$ so the axis $a$ is  not weakly primitive, and  the fusion rules fail.
\end{example}

To avoid such noise, we assume henceforth that $a$ is weakly primitive.

\subsection{The case of $\lan \lan a,b \ran \ran$  when the axis $a$ is weakly primitive}\label{str}$ $

\begin{note}\label{assmp}
 Throughout the remainder of this section, $A = \lan \lan a,b \ran \ran$ is an algebra of dimension  $\ge 3$, where   $a$ is a  weakly primitive $S$-axis of $A$ satisfying the fusion rules, cf.~Definition~\ref{not1}(8), and $b =\ga_b a +b_0 +\sum _{ (\mu,\nu)\in S}b_{\mu,\nu}$ is an idempotent.   We take $ S^\circ: = S^\circ(a),$  $S^{\dagger}: = S^{\dagger}(a)$ and $\bar S:=\bar S(a),$ cf.~Definition~\ref{not1}(3).    We shall go as far as we can without assuming that $b$ also satisfies the fusion rules.
  Also define \begin{equation}
      \label{Vdef} V = \{ y\in A:  \ga_y= 0 = y_0\}, \qquad    V_{\mu,\nu}(a)=V\cap A_{\mu,\nu}(a),
  \end{equation} $$  V(a)=\sum _{(\mu,\nu)\in S^\circ(a)}V\cap A_{\mu,\nu}(a).$$
\end{note}

\begin{lemma}\label{rempr1}   $b_{\mu,\nu}\ne 0$ for each $(\mu,\nu)\in S^\circ$.
\end{lemma}

\begin{proof}
   The fusion rules for $a$ and the fact that $b_{\mu,\nu}^2=0,$ imply that $Ab_{\mu,\nu}\subseteq (\ff a+A_0(a))b_{\mu,\nu}\subseteq \ff b_{\mu,\nu}.$  Similarly $b_{\mu,\nu}A\subseteq \ff b_{\mu,\nu}.$
 The remaining part follows.
\end{proof}

Here are our main structure theorems.
We can proceed fairly far just with the property that $(ba)b = b(ab).$
\begin{thm}\label{Ser02} Suppose $(ba)b = b(ab).$ Then
    \begin{enumerate}\eroman
  \item \begin{equation}\label{61}\sum _{(\mu,\nu)\in S} (\mu-\nu) b_{\mu,\nu}^2 =0,\end{equation} and \begin{equation}\label{62}
\nu b_{\mu,\nu}b_0- \mu b_0b_{\mu,\nu}  = \ga_b (\mu+\nu-1)(\mu-\nu)b_{\mu,\nu} , \, \forall (\mu,\nu)\in S.
\end{equation}

        \item $b_0 b_{\mu,\nu}, b_{\mu,\nu}b_0 \in \ff b_{\mu,\nu}$ for all $(\mu,\nu)\in S,$ with the precise formulas holding, writing $\varepsilon:= \varepsilon_{\mu,\nu} = \mu +\nu:$
   \begin{equation}\label{622}
   \begin{aligned}
        ((1 -\ga_b)\mu +\ga_b(1- \varepsilon) \nu)b_{\mu,\nu}  &= \varepsilon b_{\mu,\nu}b_0,\\ \quad  ((1 -\ga_b)\nu +\ga_b(1- \varepsilon) \mu)b_{\mu,\nu}  &= \varepsilon b_0 b_{\mu,\nu}.
        \end{aligned}
\end{equation}
Furthermore if $\varepsilon = 0$ then $(\mu,\nu) \notin S^\circ.$

\item
\begin{enumerate}
\item
$
(1-\ga_b\varepsilon )b_{\mu,\nu}=b_{\mu,\nu}b_0+b_0b_{\mu,\nu}.
$
\item
$b_{\mu,\mu}b_0=b_0b_{\mu,\mu}.$

\item
If $\varepsilon =1=\ga_b,$ then  $b_{\mu,\nu}b_0 = 0=b_0b_{\mu,\nu}.$

\item
If $\ga_b\ne 1$ and $\mu\ne\nu,$ then either $b_{\mu,\nu}b_0 \ne 0$ or $b_0b_{\mu,\nu}\ne 0.$

\item
If $\mu=0$ and $\nu\ne 1,$ then $b_{\mu,\nu}b_0\ne 0,$ and if $\nu=0$ and $\mu\ne 1,$ then $b_0b_{\mu,\nu}\ne 0.$
\end{enumerate}

\item
 \begin{equation}\label{6224} \begin{aligned}
     b_0(b_{\mu,\nu} b_0)  &= (b_0 b_{\mu,\nu}) b_0    = \\& = \frac{((1 -\ga_b)\mu +\ga_b(1- \varepsilon) \nu)((1 -\ga_b)\nu +\ga_b(1- \varepsilon) \mu)b_{\mu,\nu}} {\varepsilon ^2}.
 \end{aligned}\end{equation}

\item   $  (\ga _b a+b_0 )b_{\mu,\nu} = \frac{  (1 -\ga_b)\nu +\ga_b  \mu}{\varepsilon}b_{\mu,\nu} , $ and  $  b_{\mu,\nu}(\ga _b a+b_0 ) =\frac{  (1 -\ga_b)\mu +\ga_b  \nu}{\varepsilon}b_{\mu,\nu}.$


 \item     $b_{\mu,\nu}$ is a $\left(\frac{  (1 -\ga_b)\nu +\ga_b  \mu}{\varepsilon},\frac{  (1 -\ga_b)\mu +\ga_b  \nu}{\varepsilon} \right)$ eigenvector of~$b,$ iff    $b_{\mu,\nu}^2 =0.$

 For the last assertion, we exploit    $b$ being idempotent.

 \item $  b_0^2    - b_0   =  (\ga_b-   \ga_b^2 )a    - \sum  b_{\mu,\nu}^2$.
    \end{enumerate}

     Conversely, if \eqref{61} and  \eqref{62} hold for an idempotent $b$, then  $(ba)b = b(ab)$.

\end{thm}
\begin{proof} (i) and (ii).  We may assume that $S^{\circ}\ne 0$. As in
\cite[Lemma~2.4(iv)]{RoSe1}, we compute $b(ab)=(ba)b.$
\[
\begin{aligned}
b(ab) = b&\left(\ga_b a+  \sum _{(\mu,\nu)\in S} \mu b_{\mu,\nu} \right)
 =  \ga_b ba +  \sum _{\mu,\nu} \mu b b_{\mu,\nu}\\
& =\ga_b \left(\ga_b a+  \sum _{\mu,\nu} \nu b_{\mu,\nu}\right) +\sum _{\mu,\nu} \mu \left(\ga_b a+b_0 +  \sum _{\mu',\nu'} b_{\mu',\nu'}
\right) b_{\mu,\nu}\\
& =\ga_b ^2 a + \sum _{\mu,\nu} \sum _{\mu',\nu'} \mu b_{\mu',\nu'}b_{\mu,\nu} + \sum _{\mu,\nu} \ga_b(\mu^2 +\nu )b_{\mu,\nu}+ \sum _{\mu,\nu}  \mu
b_0  b_{\mu,\nu}
\\
& =\ga_b ^2 a + \sum _{\mu,\nu} \mu b_{\mu,\nu}^2 + \sum _{\mu,\nu} \ga_b(\mu^2 +\nu )b_{\mu,\nu}+ \sum _{\mu,\nu}  \mu
b_0  b_{\mu,\nu},
\end{aligned}
\]

since $ b_{\mu,\nu}b_{\mu',\nu'}=0$ for $   (\mu,\nu) \ne (\mu',\nu').$
On the other hand,
\[
\begin{aligned}(ba)b &= \left(\ga_b a+ \sum _{\mu,\nu} \nu b_{\mu,\nu}\right) b = \ga_b ab+ \sum _{\mu,\nu} \nu b_{\mu,\nu} b\\
& =  \ga_b (\ga_b a+  \sum _{\mu,\nu}  \mu b_{\mu,\nu} ) + \sum _{\mu,\nu}  \nu b_{\mu,\nu}(\ga_b a+b_0+ \sum _{\mu',\nu'}b_{\mu',\nu'})\\
& =\ga_b^2 a + \sum _{\mu,\nu} \sum _{\mu',\nu'} \nu b_{\mu,\nu}b_{\mu',\nu'} + \sum _{\mu,\nu} \ga_b( \mu +\nu ^2)b_{\mu,\nu}+
 \sum _{\mu,\nu} \nu b_{\mu,\nu}b_0
 \\ & =\ga_b^2 a + \sum _{\mu,\nu}\nu b_{\mu,\nu}^2 + \sum _{\mu,\nu} \ga_b( \mu +\nu ^2)b_{\mu,\nu}+
 \sum _{\mu,\nu} \nu b_{\mu,\nu}b_0.
\end{aligned}
\]

Matching components in  $ A_{1}(a) \oplus
A_{0}(a)$ yields \begin{equation}\label{65}
    \sum _{\mu,\nu}  \mu b_{\mu,\nu}^2 = \sum _{\mu,\nu}  \nu
 b_{\mu,\nu}^2 .
\end{equation}

i.e., $\sum _{\mu,\nu} (\mu-\nu) b_{\mu,\nu}^2=0,$ which is \eqref{61}.

Matching components in $A_{\mu,\nu}(a)$ yields $$\ga_b (\mu^2+\nu)b_{\mu,\nu} +\mu b_0b_{\mu,\nu} = \ga_b(\nu^2 +\mu)b_{\mu,\nu} +\nu b_{\mu,\nu}b_0,$$   yielding \eqref{62}.

Next,
\begin{equation}\label{bmn7} \begin{aligned}
    \ga_b a & + b_0 + \sum  _{(\mu,\nu)\in S} b_{\mu,\nu}  = b   = b^2  =  \ga_b^2 a  +   b_0^2 + \sum  b_{\mu,\nu}^2 \\& + \sum (\ga_b \mu b_{\mu,\nu} + \ga_b \nu b_{\mu,\nu})  + \sum (b_0 b_{\mu,\nu} + b_{\mu,\nu}b_0).
\end{aligned}   \end{equation}

    Matching components in  $b_{\mu,\nu}$ yields
    \begin{equation}\label{b2match} b_{\mu,\nu}= \ga_b (\mu+\nu)b_{\mu,\nu} + b_0 b_{\mu,\nu} + b_{\mu,\nu}b_0,
    \end{equation} or
    \begin{equation}
        \label{bmn77} (1-  \ga_b (\mu+\nu))b_{\mu,\nu} = b_0 b_{\mu,\nu} + b_{\mu,\nu}b_0
    \end{equation}

    Reformulating \eqref{62} yields
 $$\varepsilon b_{\mu,\nu}b_0 =\mu (b_0 b_{\mu,\nu}+ b_{\mu,\nu}b_0)
 +\ga_b (\varepsilon -1)(\mu-\nu)b_{\mu,\nu} ,$$
 i.e.

Matching with \eqref{bmn77} yields
   $$\varepsilon b_{\mu,\nu}b_0 =\mu  (1-  \ga_b \varepsilon)b_{\mu,\nu}
 +\ga_b (\varepsilon-1)(\mu-\nu)b_{\mu,\nu} ,$$

so

\begin{equation}
    \label{6220}
((1 -\ga_b)\mu +\ga_b(1- \varepsilon) \nu)b_{\mu,\nu}  = \varepsilon b_{\mu,\nu}b_0 .\end{equation}

Also applying symmetry,  we get  \eqref{622}.
If $\varepsilon =0,$ then we have $$(1 -\ga_b (\mu -\nu))b_{\mu,\nu} = 0=(1 -\ga_b (\nu -\mu))b_{\mu,\nu},$$ so $\ga_b(\mu-\nu) =0$,   implying $b_{\mu,\nu} =0$.

(iii) (a) Adding both equations \eqref{622} and cancelling $\varepsilon$ gives the assertion.

(b)\&(c) This is immediate from \eqref{622}.

(d) Suppose $\ga_b\ne 1, \mu\ne\nu,$ and $b_{\mu,\nu}b_0=b_0b_{\mu,\nu}=0.$  By (a), $\ga_b=\frac{1}{\varepsilon},$  so $\varepsilon\ne 1.$ Then by \eqref{622}
$
(1 -\frac{1}{\varepsilon})\mu +\frac{1}{\varepsilon}(1- \varepsilon) \nu)=0,
$
which implies $\mu=\nu,$ a contradiction.

(e) This is clear.

(iv) \begin{equation}
    \begin{aligned}
        \varepsilon^2& b_0(b_{\mu,\nu} b_0) = \varepsilon b_0 ((1 -\ga_b)\mu +\ga_b(1- \varepsilon) \nu)b_{\mu,\nu}\\& = ((1 -\ga_b)\mu +\ga_b(1- \varepsilon) \nu)((1 -\ga_b)\nu +\ga_b(1- \varepsilon) \mu)b_{\mu,\nu} \\ & = \varepsilon^2 (b_0 b_{\mu,\nu}) b_0,
    \end{aligned}
\end{equation} proving $b_0(b_{\mu,\nu} b_0)= (b_0 b_{\mu,\nu}) b_0.$

(v) $\ga_b a b_{\mu,\nu}  = \ga_b \mu b_{\mu,\nu} .$  Hence using \eqref{622}, \begin{equation}\label{625}
    \begin{aligned}
        \ga_b a b_{\mu,\nu}  +  b_0 b_{\mu,\nu} &=  \frac{(\ga_b \mu \varepsilon+ (1 -\ga_b)\nu +\ga_b(1- \varepsilon) \mu)b_{\mu,\nu}}{\varepsilon} \\& = \frac{  (1 -\ga_b)\nu +\ga_b  \mu}{\varepsilon}b_{\mu,\nu} ,
    \end{aligned}
\end{equation} and the other assertion follows symmetrically.

(vi) $b b_{\mu,\nu} = (\ga_b a+b_0)b_{\mu,\nu}+b_{\mu,\nu}^2,$ so apply (v).

(vii) Match components in  $\ff a +A_0(a) $ in \eqref{bmn7}, to get
 $$  \ga_b a  + b_0   =  \ga_b^2 a  +   b_0^2 + \sum  b_{\mu,\nu}^2 ,$$ and rearrange.

The converse follows by reversing the proof of (i).
\end{proof}



For later use, let us record   important annihilators of $b_{\mu,\nu}.$

\begin{cor}\label{half11}$ $
   Suppose that $(ba)b=b(ab),$ let $(\mu,\nu)\in S^{\circ},$ and set $\varphi_1:=(1 -\ga_b)\mu +\ga_b(1- \varepsilon) \nu,$ $\varphi_2:=(1 -\ga_b)\nu +\ga_b(1- \varepsilon) \mu,$   $w_1=\frac{\varphi_1}{\varepsilon}a-\nu b_0$ and $w_2=\frac{\varphi_2}{\varepsilon}a-\mu b_0,$ where $\varepsilon=\mu+\nu.$ Then 
    \begin{enumerate}\eroman
\item $b_{\mu,\nu}w_1=0=w_2b_{\mu,\nu}.$

     \item For $b_0 \ne 0$, $w_1=0$ $($resp.~$w_2=0)$ iff $\nu=0$ $($resp.~$\mu=0)$ and $\ga_b=1.$


    \end{enumerate}
\end{cor}
\begin{proof}
     (i)  $\varepsilon b_{\mu,\nu}w_1=\varphi_1 b_{\mu,\nu} a -\varepsilon  \nu b_0b_{\mu,\nu}= \varepsilon \nu b_0 b_{\mu,\nu}-\varepsilon \nu b_0 b_{\mu,\nu}=0$ by \eqref{622}, and likewise $ 0=w_2b_{\mu,\nu}$.

     (ii) For $b_0 \ne 0$,
 $w_1=0$ iff $\nu=0$ and $\varphi_1=0$, iff $\nu=0$ and $\ga_b=1.$
\end{proof}
%


\begin{cor}\label{S1}
    Suppose that $(ba)b=b(ab),$ and that $|S^\circ| =1 ,$ i.e., $S^\circ= \{(\mu,\nu)\}.$ Then either $a$ is commutative or  $b_{\mu,\nu}^2=0,$ in which case $A =\ff a + \ff b  +\ff b_{\mu,\nu} = \ff a +\ff b_0 +\ff b_{\mu,\nu}$ with the multiplication table given as in Theorem~\ref{Ser02}(vi).
\end{cor}
\begin{proof}
    By \eqref{61}, $$0= (\mu-\nu)b_{\mu,\nu}^2,$$ implying $\mu = \nu$ or $b_{\mu,\nu}^2=0.$
\end{proof}

Both possibilities  occur as HRS algebras.
Next, we proceed with the other properties of an axis.


\subsection{The case of $\lan \lan a,b \ran \ran$  for axes $a$ and $b$}\label{stra}$ $

\begin{note}\label {baxis}
For the remainder of this section, our goal is to describe axial algebras $A = \lan \lan a,b \ran \ran$ of dimension  $\ge 3$, where   {\bf $a$ is a  weakly primitive $S$-axis satisfying the fusion rules  and $b$ is an axis}.
\end{note}

\begin{thm}\label{Ser2} 
 For  $ (\mu,\nu)\in S^\circ $ the following four equations hold, letting $\varepsilon: = \mu +\nu \ne 0$ (by Theorem~\ref{Ser02}(ii)):

\begin{equation}\label{6641}
  (\mu -\nu)\Big( \frac{1-2\ga_b}\varepsilon  \Big) b_{\mu,\nu}^2 = b_{\mu,\nu}^2b_0-b_0b_{\mu,\nu}^2  . \end{equation}


\begin{equation}\label{664} \,  b_{\mu',\nu'} b_{\mu,\nu}^2= b_{\mu,\nu}^2 b_{\mu',\nu'}, \quad \forall (\mu',\nu') \in S.\end{equation}

\begin{equation}\label{6642} b_{\mu,\nu}b_0^2 -b_0^2 b_{\mu,\nu} =(\mu -\nu)\frac{1-2\ga_b+\ga_b^2\varepsilon}{\varepsilon} b_{\mu,\nu} .\end{equation}

  \begin{equation}\label{6643} b_0  b_0^2 -  b_0^2  b_0 = (1- 2\ga_b) \sum _{(\mu',\nu')\in S} \frac{\mu' -\nu'}{\mu'+\nu'} b_{\mu',\nu'}^2.\end{equation}

\end{thm}
\begin{proof}  Recall~\eqref{622}:
      $$     ((1 -\ga_b)\mu +\ga_b(1- \varepsilon) \nu)b_{\mu,\nu}  = \varepsilon b_{\mu,\nu}b_0  , \quad  ((1 -\ga_b)\nu +\ga_b(1- \varepsilon) \mu)b_{\mu,\nu}  = \varepsilon b_0 b_{\mu,\nu}.$$
 We compute
        $$b_{\mu,\nu} b  = b_{\mu,\nu}(\ga_b a +b_0 + \sum _{(\mu',\nu')\in S} b_{\mu',\nu'}) = \ga _b \nu b_{\mu,\nu} + b_{\mu,\nu} b_0 +b_{\mu,\nu}^2 . $$
\begin{equation}\label{mm1}
    \begin{aligned}
    &\textstyle{b(b_{\mu,\nu} b) = (\ga_b a +b_0 + \sum _{(\mu',\nu')\in S} b_{\mu',\nu'} )(\ga _b \nu b_{\mu,\nu} + b_{\mu,\nu} b_0 +b_{\mu,\nu}^2 )}
        \\ & = \ga_b ^2 \mu \nu b_{\mu,\nu}  + \ga _b  \mu b_{\mu,\nu}  b_0  +  \ga_b a b_{\mu,\nu}^2 +\ga _b \nu b_0 b_{\mu,\nu}+ b_0(b_{\mu,\nu}b_0)+    b_0 b_{\mu,\nu}^2 \\ & \quad   + \ga _b \nu  b_{\mu,\nu}^2 + b_{\mu,\nu}(b_{\mu,\nu} b_0)  +\sum _{(\mu',\nu')\in S} b_{\mu',\nu'} b_{\mu,\nu}^2  .
    \end{aligned}
\end{equation}

Symmetrically,
\begin{equation}\label{mm2}
    \begin{aligned}
    (bb_{\mu,\nu})b    = \ga_b ^2& \mu \nu b_{\mu,\nu} + \ga _b \nu b_0  b_{\mu,\nu}+ \ga_b b_{\mu,\nu}^2 a  + \ga _b \mu b_{\mu,\nu} b_0  +( b_0 b_{\mu,\nu})b_0 \\& + b_{\mu,\nu}^2b_0  + \ga_b  \mu  b_{\mu,\nu}^2     +(b_0 b_{\mu,\nu}  ) b_{\mu,\nu}  +\sum _{(\mu',\nu')\in S}  b_{\mu,\nu}^2b_{\mu',\nu'} .  \end{aligned}
\end{equation}

But $b(b_{\mu,\nu} b) = (b b _{\mu,\nu})b.$ Matching components
in $a$ and $A_0(a)$ and noting $ a b_{\mu,\nu}^2 = b_{\mu,\nu}^2 a $ shows $$ \ga _b \nu  b_{\mu,\nu}^2  + b_{\mu,\nu}(b_{\mu,\nu} b_0) + b_0b_{\mu,\nu}^2=  \ga _b \mu  b_{\mu,\nu}^2 +( b_0 b_{\mu,\nu})b_{\mu,\nu} + b_{\mu,\nu}^2b_0,$$
or

\begin{equation}\label{66414}
     \ga _b (\mu -\nu)  b_{\mu,\nu}^2 =   b_{\mu,\nu}(b_{\mu,\nu} b_0) -( b_0 b_{\mu,\nu})b_{\mu,\nu} +b_0b_{\mu,\nu}^2-b_{\mu,\nu}^2b_0.
\end{equation}

But \eqref{622} says
\begin{equation}
    \begin{aligned}
      \varepsilon(   b_{\mu,\nu}(b_{\mu,\nu} b_0) &-( b_0 b_{\mu,\nu})b_{\mu,\nu})\\ & =   b_{\mu,\nu}\left( ((1 -\ga_b)\mu +\ga_b(1- \varepsilon) \nu)\right)b_{\mu,\nu}\\& \qquad -\left((1 -\ga_b)\nu +\ga_b(1- \varepsilon) \mu\right)b_{\mu,\nu}b_{\mu,\nu}\\&  =  {(1-2\ga_b+\ga_b\varepsilon)(\mu-\nu)} b_{\mu,\nu}^2.
    \end{aligned}
\end{equation}
 Plugging into \eqref{66414}  yields

 \begin{equation}
     \ga _b (\mu -\nu)  b_{\mu,\nu}^2 =   (\mu-\nu)\left(\frac {(1-2\ga_b)}\varepsilon+\ga_b\right)  b_{\mu,\nu}^2 +b_0b_{\mu,\nu}^2-b_{\mu,\nu}^2b_0.
\end{equation}

Thus  \begin{equation}\label{66411}
      (\mu -\nu)\Big( \frac{1-2\ga_b}\varepsilon  \Big) b_{\mu,\nu}^2 = b_{\mu,\nu}^2b_0-b_0b_{\mu,\nu}^2 ,
\end{equation}
which is \eqref{6641}.

Next,
matching components in $b_{\mu',\nu'} $ for $(\mu',\nu') \ne (\mu,\nu)$ in \eqref{mm1} and \eqref{mm2}  yields \begin{equation}
    b_{\mu',\nu'} b_{\mu,\nu}^2= b_{\mu,\nu}^2b_{\mu',\nu'} ,
\end{equation}
which  is \eqref{664}.

Matching components in $b_{\mu,\nu} $ in \eqref{mm1} and \eqref{mm2}  yields
$$  \ga _b \mu  b_{\mu,\nu}b_0 +\ga _b \nu  b_0 b_{\mu,\nu}
   +b_0 (b_{\mu,\nu}   b_0) +  b_{\mu,\nu}  b_{\mu,\nu}^2 = $$ $$ = \ga _b \nu  b_0 b_{\mu,\nu}+\ga _b \mu  b_{\mu,\nu}b_0 +( b_0 b_{\mu,\nu})b_0 +  b_{\mu,\nu}^2 b_{\mu,\nu},  $$
which again is \eqref{664} since $ b_0 (b_{\mu,\nu}   b_0) =  ( b_0 b_{\mu,\nu})b_0$ by \eqref{6224}.

\begin{equation}\label{mm11}
    \begin{aligned}b(b_0 b) & = (\ga_b a +b_0 + \sum _{(\mu',\nu')\in S} b_{\mu',\nu'} )(b_0^2 +\sum_{(\mu',\nu')\in S} b_0  b_{\mu',\nu'}    )
        \\ & = \ga_b
         a b_0^2
        +\sum _{(\mu',\nu')\in S} \ga_b\mu' b_0  b_{\mu',\nu'} +b_0 {b_0^2}
      +\sum_{(\mu',\nu')\in S}  b_0  (  b_0 b_{\mu',\nu'})  \\& +\sum   _{(\mu',\nu')\in S}( b_{\mu',\nu'} b_0^2 )  +\sum _{(\mu',\nu')\in S}  b_{\mu',\nu'} (b_0b_{\mu',\nu'})
    \end{aligned}
\end{equation}
since $A_{\mu',\nu'}(a) A_{\mu,\nu}(a) =0 $ for $(\mu',\nu') \ne  (\mu,\nu).$
Symmetrically,
\begin{equation}\label{mm12}
    \begin{aligned} (bb_0)b &   = \ga_b b_0^2 a
        +\sum _{(\mu',\nu')\in S} \ga_b \nu'   b_{\mu',\nu'}b_0  + {b_0^2} b_0
     \\& +\sum_{(\mu',\nu')\in S}  (    b_{\mu',\nu'} b_0 )  b_0   +\sum   _{(\mu',\nu')\in S}( b_0^2 b_{\mu',\nu'} ) +\sum _{(\mu',\nu')\in S}   ( b_{\mu',\nu'} b_0)b_{\mu,\nu}
    \end{aligned}
\end{equation}
   Matching components in $ \ff b_{\mu,\nu }$  in \eqref{mm11} and \eqref{mm12}  yields
\begin{equation}
    \label{66421}\ga_b \mu  b_0  b_{\mu,\nu}+ b_0 (  b_0  b_{\mu,\nu } )   + b_{\mu,\nu}b_0^2  =   \ga_b \nu   b_{\mu,\nu}b_0+(   b_{\mu,\nu}  b_0 )  b_0 +b_0^2 b_{\mu,\nu}.
\end{equation}

But \eqref{622} says
\begin{equation}
    \begin{aligned}
         (b_{\mu,\nu} b_0)b_0&   =    \left(\frac { (1 -\ga_b)\mu +\ga_b(1- \varepsilon) \nu}{\varepsilon }\right)b_{\mu,\nu}b_0\\& =\left(\frac { (1 -\ga_b)\mu +\ga_b(1- \varepsilon) \nu}{\varepsilon }\right)  ^2b_{\mu,\nu},
    \end{aligned}
\end{equation}
and symmetrically
 $b_0  (  b_0 b_{\mu,\nu}) =\left(\frac { (1 -\ga_b)\nu +\ga_b(1- \varepsilon) \mu}{\varepsilon }\right)  ^2b_{\mu,\nu}      $, so subtracting yields
\begin{equation}\label{6228}
    \begin{aligned}
         &(b_{\mu,\nu} b_0)b_0- b_0  (  b_0 b_{\mu,\nu})\\
         & \textstyle{=\left(\frac { (1 -\ga_b)\mu +\ga_b(1- \varepsilon) \nu}{\varepsilon }\right)  ^2b_{\mu,\nu}- \left(\frac { (1 -\ga_b)\nu +\ga_b(1- \varepsilon) \mu}{\varepsilon }\right)  ^2b_{\mu,\nu}} \\
         &\textstyle{ = \left(\frac { (1 -\ga_b)\mu +\ga_b(1- \varepsilon) \nu}{\varepsilon }  - \frac { (1 -\ga_b)\nu +\ga_b(1- \varepsilon) \mu}{\varepsilon } \right)}\\
         &\textstyle{\quad\left(\frac { (1 -\ga_b)\mu +\ga_b(1- \varepsilon) \nu}{\varepsilon }  + \frac { (1 -\ga_b)\nu +\ga_b(1- \varepsilon) \mu}{\varepsilon }\right)b_{\mu,\nu}}
         \\ &\textstyle{ =
        \left(\frac {(1-2\ga_b+\ga_b\gre)(\mu-\nu)}{\varepsilon}\right)(1-\ga_b\gre)
        b_{\mu,\nu} =  (\mu-\nu) \frac {(1-\ga_b \varepsilon)(1-2\ga_b+\ga_b\gre) } {\varepsilon  }   b_{\mu,\nu}}
    \end{aligned}
\end{equation}
Note  that
\eqref{622} also implies
\begin{equation}\label{6229}
    \begin{aligned}
        \mu b_0  b_{\mu,\nu}   - \nu b_{\mu,\nu}b_0 & = \frac{((1 -\ga_b)\nu +\ga_b(1- \varepsilon)\mu) \mu -((1 -\ga_b)\mu +\ga_b(1- \varepsilon) \nu)\nu}{\varepsilon} b_{\mu,\nu}\\& =\ga_b(1- \varepsilon)\frac{\mu ^2 -\nu^2}{\varepsilon} b_{\mu,\nu}=  \ga_b(1- \varepsilon)( \mu-\nu)b_{\mu,\nu}.
    \end{aligned}
\end{equation}
  Plugging  \eqref{6228} and \eqref{6229} into \eqref{66421},        yields
\begin{equation}
\ga_b^2(1- \varepsilon)( \mu-\nu )b_{\mu,\nu} + b_{\mu,\nu}b_0^2  =
(\mu-\nu) \frac {(1-\ga_b \varepsilon)(1-2\ga_b+\ga_b\gre)} {\varepsilon  }     +b_0^2 b_{\mu,\nu},
\end{equation}

or \begin{equation}
    \begin{aligned}
         b_{\mu,\nu}b_0^2 -b_0^2 b_{\mu,\nu}& =
         (\mu -\nu)\frac{(1-\ga_b \varepsilon)(1-2\ga_b+\ga_b\gre) -\ga_b^2(1-\gre)\varepsilon}{\varepsilon} b_{\mu,\nu}\\
         & = (\mu -\nu)\frac{1-2\ga_b+\ga_b^2\gre}{\varepsilon} b_{\mu,\nu} ,
    \end{aligned}
\end{equation}
    which is \eqref{6642}.

   Matching parts  in \eqref{mm11} and \eqref{mm12}  of $ \ff a +A_0(a)$  yields
   $$b_0  b_0^2  + \sum _{(\mu',\nu')\in S}  b_{\mu',\nu'} (b_0b_{\mu',\nu'}) =   b_0^2  b_0 +\sum _{(\mu',\nu')\in S}  ( b_{\mu',\nu'} b_0)b_{\mu',\nu'},$$
   since $ab_0^2 = b_0^2 a, $

   yielding
      $$b_0  b_0^2 -  b_0^2  b_0 = \sum _{(\mu',\nu')\in S}   ( b_{\mu',\nu'} b_0)b_{\mu',\nu'}-  b_{\mu',\nu'} (b_0b_{\mu',\nu'}),$$
or, since by   \eqref{622}, $$((1 -\ga_b)\mu +\ga_b(1- \varepsilon)\nu )-((1 -\ga_b)\nu +\ga_b(1- \varepsilon) \mu)=(1-2\ga_b+\ga_b\gre)(\mu-\nu),$$ we get
 \begin{equation}
     \begin{aligned}
         &b_0  b_0^2 -  b_0^2  b_0\textstyle{ = \sum _{(\mu',\nu')\in S}  (\mu' -\nu')\frac{1-2\ga_b+\ga_b(\mu'+\nu')}{(\mu'+\nu')} b_{\mu',\nu'}^2} \\ &\textstyle{ = (1- 2\ga_b) \sum _{(\mu',\nu')\in S} \frac{\mu' -\nu' }{\mu'+\nu'}  b_{\mu',\nu'}^2 +\ga_b \sum _{(\mu',\nu')\in S}  (\mu' -\nu') b_{\mu',\nu'}^2}\\ &\textstyle{ =  (1- 2\ga_b)\sum _{(\mu',\nu')\in S} \frac{\mu' -\nu'}{\mu'+\nu'} b_{\mu',\nu'}^2,}
     \end{aligned}
 \end{equation}
    by \eqref{61}.
   \end{proof}

We already have a key result.

\begin{lemma}\label{b0sq}
 If $A_0(a)$ is commutative  and $\ga_b\ne\half,$
  then $\bar S=\emptyset,$ i.e., $b_{\mu,\nu}^2=0$ for all $(\mu,\nu)\in S^{\dagger}.$
  \end{lemma}
  \begin{proof} By the fusion rules for $a,$
      $b_0$ commutes  with $b_{\mu',\nu'}^2,$ for all $(\mu',\nu')\in S.$
 Hence (ii) holds by \eqref{6641}.
  \end{proof}

 \begin{lemma}\label{nonz}$ $
 \begin{enumerate}\eroman
     \item $ \mu b_0b_{\mu,\nu} =\nu b_{\mu,\nu}b_0$
      holds, if and only if
  $\mu = \nu$ or $\mu+\nu = 1$  or $\ga_b=0$.

 \item  If $b_0$ commutes with $A_0(a)$ (in particular if $b_0 =0$ or $A_0(a)=\ff b_0)$ then either $\ga_b=\half$ or $\sum _{(\mu',\nu')\in S} \frac{\mu' -\nu'}{\mu'+\nu'} b_{\mu',\nu'}^2 =0. $

 \end{enumerate}

\end{lemma}
\begin{proof}
(i)  Immediate from \eqref{62}.
 \medskip

(ii) By \eqref{6643}, $(1- 2\ga_b) \sum _{(\mu',\nu')\in S} \frac{\mu' -\nu'}{\mu'+\nu'} =0.$
\end{proof}

\begin{remark}\label{Acomm}
    When $A$ is commutative, Theorem~\ref{Ser2} is superfluous, as are \eqref{61} and \eqref{62}.

      Continuing,  our only constraint is \eqref{622}, which controls the multiplication of $b_{\mu,\nu}$ with $\ff b_0$, and Theorem~\ref{Ser02}(vii), which could be viewed as a formula for $b_0^2.$ Otherwise, we have no control over $A_0(a)$, which could have arbitrarily large dimension, and we have no information about the product of its elements.
\end{remark}

\subsection{The case $b_0=0$}$ $
\medskip

We are ready to characterize all axial algebras~$A$ in the case that $b_0=0.$ They are rather straightforward.

\begin{example}\label{b0zeroa} Here are two classes  with $ b_0=0.$
    \begin{enumerate}
        \item ($A$   commutative). $A = \ff a \oplus \ff b_{\mu,\mu},$ with $\mu=\frac{1}{2\ga_b},$  $b = \ga_b a +b_{\mu,\mu}$
        and $b_{\mu,\mu}^2 = \ga_b(1-\ga_b)a$. Then
        \begin{equation}
            \label{bsqu} b^2 = \ga_b^2 a + \ga_b(1-\ga_b)a + 2\ga_b \mu b_{\mu,\mu} =\ga_b a +b_{\mu,\mu} =b,
        \end{equation}
       and the fusion rules hold for $a.$  Then  $\dim A=~2$, with $$ab = \ga_b a +\mu b_{\mu,\mu}=  \ga_b a +\mu (b-\ga_b a) = \ga_b(1-\mu)a +\mu b.$$
          By Lemma~\ref{pr01}(i), where $\xi = \mu$ and $\rho = \ga_b(1-\mu)$ the second eigenvalue of $b $ is $\ga_b(1-\mu).$

          Until now, $\ga_b$ is arbitrary.
          By Lemma~\ref{pr01}(iv),  $b$ satisfies the fusion rules if and only if $0 =\ga_b(1-\mu) -1 +2\frac{1}{2\ga_b}\ga_b(1-\mu) = \ga_b - \ga_b\mu  -\mu. $ Thus $\frac{\ga_b}{1+\ga_b} = \mu = \frac{1}{2\ga_b}$, implying $2\ga_b^2 = 1+\ga_b,$ which has the roots $\ga_b = 1$ or $\ga_b = -\half.$ Thus $\mu =\half$ or $\mu = -2$. This matches the results in \cite{HRS}.

         \item ($A$   not commutative). $\mu+\nu =1$ for all $(\mu,\nu)\in S^\circ,$ and $S^{\dagger}\ne\emptyset.$
         \[
         \textstyle{A=\ff a+\sum_{(\mu,\nu)\in S^{\circ}}\ff b_{\mu,\nu},}
         \]
         where  $b = a +\sum_{(\mu,\nu)\in S^{\circ}} b_{\mu,\nu},$ and $b_{\mu,\nu}^2 = 0 $ for each $(\mu,\nu)\in S^{\circ}.$

         Clearly each $b_{\mu,\nu}$ is also  a $(\mu,\nu)$-eigenvector of $b$, so $b$ automatically satisfies the fusion rules.  $I:=\sum_{(\mu,\nu)\in S^{\circ}}\ff b_{\mu,\nu}$ is a square zero ideal in $A,$ with $A/I \cong \ff a. $
    \end{enumerate}
\end{example}

Now we shall show that these are the only examples for $b_0=0$.

\begin{thm}\label{b0zero}
    Suppose $ b_0=0.$ Then  $A_0(a)=0$.
    Furthermore,
     \begin{enumerate}\eroman

          \item If $S^{\dagger}=\emptyset,$ then $A$ is as in Example~\ref{b0zeroa}(i).
            \item If $S^{\dagger}\ne\emptyset,$ then $A$ is as in Example~\ref{b0zeroa}(ii).
        \end{enumerate}
\end{thm}

\begin{proof} Notation as in Corollary~\ref{half11}, $ w_i= \varphi_i a $ for $i=1,2,$ so   $0 = \varphi_1  = \varphi_2 $ by  Corollary~\ref{half11}(i).  Then
   \begin{equation}
       \label{b0z}  \begin{aligned}
0=\varphi_1+\varphi_2&=(1-\ga_b)\varepsilon+\ga_b(1-\varepsilon)\varepsilon\\
0=\varphi_1-\varphi_2&=(1-\ga_b)(\mu-\nu)+\ga_b(1-\varepsilon)(\nu-\mu)
     \end{aligned}
   \end{equation}
     Thus, cancelling $\varepsilon$ we get $1-\ga_b\epsilon=0.$ Hence $\ga_b \ne 0,$ so Lemma~\ref{nonz}(i)  implies $\mu = \nu$ or $\mu + \nu=1$. If $\mu = \nu$ then $\ga_b\mu = \frac{1}{2}.$

     (i) $b = \ga_b a +b_{\mu,\mu}$ since $S^\dagger=\emptyset$. Now \eqref{bsqu} rearranged says $$\ga_b a +b_{\mu,\mu} =b =b^2 = \ga_b^2 a +  b_{\mu,\mu}^2 + 2\ga_b \mu b_{\mu,\mu},$$
     implying  $b_{\mu,\mu}^2 =  (\ga_b- \ga_b^2)a.$

     (ii) For $\mu\ne\nu,$
     cancelling $\mu-\nu$ in \eqref{b0z}, we get $1-2\ga_b+\ga_b\varepsilon=0.$ Since $\ga_b\varepsilon=1,$ we see that $\ga_b=1,$ and then also $\varepsilon=1.$ By   \eqref{6641}, $b_{\mu,\nu}^2 =0.$

     Now, for $\mu = \nu$ we have $\mu = \frac{1}{2\ga_b}= \half.$
     Thus $b = a + b_{\half,\half}+\sum_{(\mu,\nu)\in S^{\dagger}} b_{\mu,\nu}.$
Furthermore $b = b^2 = a^2 + b_{\half,\half}^2 +b_{\half,\half}+\sum_{(\mu,\nu)\in S^{\dagger}} b_{\mu,\nu}, $ proving $b_{\half,\half}^2=0.$ This yields Example~\ref{b0zeroa}(ii).
\end{proof}

 \subsection{Axial algebras having a 1-axis}\label{3d}$ $


 Now we can establish the situation when $|S^\circ |=1,$ and $b_0\ne 0.$ Any  PAJ generated  by two 1-axes has dimension $\le 3,$  and is an HRS algebra if it is commutative, cf.~\cite{RoSe2,RoSe3}.

\begin{lemma}\label{ax1}
    If $a$ is a $( \mu,\nu)$-axis, and $b = \ga_b a  + b_0 +b_{\mu,\nu}$, 
    then $\mu,\nu \ne 0.$
\end{lemma}
\begin{proof}
    If $\mu = 0$ then $ab = \ga_b a +\mu b_{\mu,\nu} = \ga_b a $, so $a$ is a left eigenvector for~$b.$ Since $b$ is an axis, $a$ is a right eigenvector for $b.$
    But $ba = \ga_b a  +\nu b_{\mu,\nu},  $
so $\nu = \mu = 0,$ which is impossible. Symmetrically, $\nu \ne 0.$\end{proof}
\begin{example} \label{try1}

  Let  $ A:= A_{\operatorname{exc},3}(\{a,b\};\mu,\nu,\mu',\nu') $
denote the 3-dimensional algebra spanned by idempotents $a,b$ and an
element~$y,$ where   $(\mu,\nu), (\mu',\nu')  \notin \{(0,0),(1,1)\},$ satisfying the relations $y^2=0$ and
 \begin{equation}\label{dec1115} \begin{aligned} &
ab  =   \mu y,\quad  ay= \mu y, \quad ya = \nu
y,   \\ &   ba =     \nu
y,  \quad by = \nu'y, \quad    yb = \mu'
y .
\end{aligned}\end{equation}

\begin{enumerate}\eroman
 \item  One checks easily that the eigenspaces of $a$ are $\ff a = A_1(a),$ $\ff (b-y) = A_0(a),$ and $\ff y = A_{\mu,\nu}(a).$
  \item   $a$ is  a   weakly primitive $( \mu,\nu)$-axis by Remark~\ref{b}.

\item
 $b-y$ is an idempotent, and $\mu'+\nu'$ =1.
(Indeed
 $(b-y)^2- (b-y)\in \ff y \cap (A_0(a) + \ff a)=0.$
Then
    $ y-(\mu'+\nu')y= (b-y)^2- (b-y) =0, $
implying $\mu'+\nu'=1.$)

  \item
       $b(ab) = b\mu y  = \mu \nu' y$
       whereas $(ba)b = \nu yb = \mu' \nu y ,$
implying $b(ab) =  (ba)b $ if and only if
  \begin{equation}
      \label{p2} \mu \nu'   = \mu' \nu.
  \end{equation}
By Lemma~\ref{ax1}, $\mu,\nu \ne 0.$

\item {\it From now on, assume that  \eqref{p2} holds.}
Then $b(\nu' a - \nu y) =0$ and $(\mu' a - \mu y)b =0.$ But \eqref{p2}  implies that the vectors $\nu' a - \nu y$ and $\mu' a - \mu y $ are proportional.  Let   $W = \ff(\mu' a - \mu y)$;
hence
the eigenspaces of $b$ are $\ff b = A_1(b),$ $W = A_0(b),$ and $\ff y = A_{\mu',\nu'}(b),$
     implying  $b$ also  is a weakly primitive axis.

\item  The algebra $A_{\operatorname{exc},3}(\{a,b\};\mu,\nu,\mu',\nu')$  is a PAJ, iff $\mu'=\mu, \nu'=\nu$ and $\mu+\nu=1.$

Indeed, to check the fusion rules for $b,$ first note that $A_0(b)\ff y = \ff(\mu'a-\mu y)y \subseteq \ff y. $  Also $$(\mu'a-\mu y)^2=(\mu')^2a-\mu'\mu (ay+ya)=(\mu')^2a-\mu'\mu (\mu+\nu)y,$$ which we need to be in $A_0(b).$  Hence $b$ satisfies the fusion rules if and only if $\mu +\nu =1$  which, in addition to \eqref{p2} and (iii) shows $\mu = \mu'$ and $\nu = \nu'$. In other words, $a$ and $b$ satisfy the fusion rules, if and only if  $\mu    = \mu' $
and $ \nu = \nu'$ with  $\mu +\nu =1$.

\item  $ A_{\operatorname{exc},3}(\{a,b\};\mu,\nu) $ denotes  $ A:= A_{\operatorname{exc},3}(\{a,b\};\mu,\nu,\mu,\nu) ,$    which is a PAJ  iff $\mu+\nu = 1,$ to be generalized to  the ``$S$-exceptional  algebra'' in Example~\ref{E1} below.
\end{enumerate}
\end{example}

Here is the situation, complementing \cite{HRS}, assumptions as in~Note~\ref{assmp}.

\begin{prop}\label{prop nc}
Let $A=\lan\lan a,b\ran\ran,$ as in Note \ref{baxis}. Suppose that  $a$ is a
$(\mu,\nu)$-axis with $\mu\ne\nu$. Let $y =  b_{\mu,\nu}.$   Then

     $$A = \ff a \oplus \ff b \oplus  \ff y.$$ with multiplication table
     \begin{equation}\label{dec11151} \begin{aligned} &
ab  = \ga_b  a +  \mu y,\quad  ay= \mu y, \quad ya = \nu
y,   \\ &   ba =  \ga_b  a +\nu
y,  \quad by = \rho y, \quad    yb = \xi
y, \text{ and }\gr+\xi=1.\\
&y^2=0.
\end{aligned}\end{equation}

Conversely, if $A=\ff a+\ff b+\ff y,$ with $a, b$ idempotents and multiplication table as in \eqref{dec11151}, then $a$ is a weakly primitive $(\mu,\nu)$-axis in $A$ satisfying the fusion rules.

One of the following holds, which also determines when   $b$ is an axis:
\begin{enumerate}\eroman
\item  $\ga_b =0$.  $ A \cong A_{\operatorname{exc},3}(\{a,b\};\mu,\nu,\rho ,\xi),$ with $\rho,\xi$ given below in the proof. Further $b$ is an axis iff $\mu\gr=\xi\nu,$ as in Example \ref{try1}, and the information about $A$ is given there.

\item   $\ga_b =1,$
and either
\begin{itemize}
    \item[(1)]
$(\gr,\xi)=(0,1), \mu=0,$ $A_1(b)=\ff b+\ff (a+\nu y),$ and $A_{\gr,\xi}(b)=\ff y,$ or $(\gr,\xi)=(1,0),$ $A_1(b)=\ff b+\ff(a+\mu y),$ and $A_{\gr,\xi}=\ff y,$  or
\item[(2)]
$\gr,\xi\ne 0, \frac{\mu}{\gr}=\frac{\nu}{\xi}$  and $A_1(b)=\ff b+\ff( a + \frac{\mu}{\gr} y),$ and $A_{\gr,\xi}(b)=\ff y.$
\end{itemize}
In both cases $b$ is not weakly primitive and $b$ satisfies  the fusion rules iff $\mu+\nu=1,$ in which case $\gr=\mu$ and $\xi=\nu.$

\item   $\ga_b\ne 0,1$.  In this case $b$ is an axis iff the third eigenvalue of $b$ (other than $1$ and $(\gr,\xi)$) is $(\ga_b,\ga_b),$ and
\[
A_{\ga_b,\ga_b}(b)=\ff (\ga_b a+\gc y),\text{ and }(2\ga_b-1)\gc=\mu+\nu,\text{ where }\gc=\frac{\mu-\nu}{\gr-\xi}.
\]
The axis $b$ is  weakly primitive, but  $b$ does not satisfy the fusion rules.

\end{enumerate}
\end{prop}
\begin{proof}
Since $b$ is an axis (see Note \ref{baxis}),  \eqref{61} implies $y^2= 0,$ so Theorem~\ref{Ser02}(vi) implies
that $y\in A_{\gr,\xi}(b),$ as in the theorem; in particular $\gr+\xi=1.$ 

 Note that $\ff a+\ff b+\ff y$ is closed under multiplication and contains $a$ and $b,$ so it equals $A.$

Conversely,  \eqref{dec11151} implies that $a$ is an axis and $a$ satisfies the fusion rules. Indeed, let $b_0:=b-\ga_ba-y.$  Then
\[
ab_0=ab-\ga_ba-\mu y=\ga_ba+\mu y-\ga_ba-\mu y=0,
\]
and similarly $b_0a=0,$ so $A=\ff a+\ff b_0+\ff y,$
proving $a$ is an axis. We check the fusion rules for $a.$
\begin{gather*}
b_0^2=b+\ga_b^2a-2\ga_b^2a-\ga_b(\mu+\nu)y-by-yb+\ga_b(\mu+\nu)y\\
=\ga_ba+b_0+y-\ga_b^2a-y=b_0+(\ga_b-\ga_b^2)a\in \ff a+\ff b_0.
\end{gather*}
Next $b_0y=(b-\ga_ba-y)y=\gr y-\ga_b\mu y\in\ff y,$ and similarly $yb_0\in\ff y.$  So $a$ satisfies the fusion rules.




  Suppose first that
 $\ga_b =0$. Then \eqref{dec11151} shows $ A \cong A_{\operatorname{exc},3}(\{a,b\};\mu,\nu,\rho ,\xi), $
 and by Example \ref{try1}, $b$ is an axis iff  $\mu\gr=\xi\nu.$

Assume $\ga_b\ne 0.$ $b$ has two eigenvectors: $b$ and $y,$ with $y\in A_{\gr,\xi}(b).$ We look for the third eigenvector $z = \gc_1 a +\gc_2 b + \gc_3 y $ of $b.$ The third eigenvector $z$ of $b$ is independent of $b$ and $y$. Hence, $\gc_1\ne 0;$ we normalize $\gc_1 =1$.

We have
\begin{equation}\label{dec1125}\begin{aligned}zb=(a +\gc_2 b + \gc_3 y )b & = (\ga_b a + \mu y) +\gc_2 b + \gc_3 \xi y \\ &= \ga_b a  +\gc_2 b + (\mu +\gc_3 \xi )y.\end{aligned}\end{equation}

\begin{equation}\label{dec11251}\begin{aligned} bz = b( a +\gc_2 b + \gc_3 y ) & =  (\ga_b a + \nu y) +\gc_2 b + \gc_3 \rho y \\ &=    \ga_ba  +\gc_2 b + (  \nu +\gc_3  \rho  )y.\end{aligned}\end{equation}

We want the   eigenvalue $\gro=(\gro_1,\gro_2)$ of $z.$ If $\gro_2 = 0$ then $\ga_b=0,$ a contradiction.  Similarly $\gro_1\ne 0.$

Suppose that $\gro_2 = 1$. In this case matching the coefficient of $a$ in \eqref{dec1125} implies  $\ga_b =1,$ and $\gc_3=\mu+\gc_3\xi.$  This shows that $\gro_2=1,$ and   the coefficient of $y$ in \eqref{dec11251}, yields $\gc_3 = \nu + \gc_3\gr.$ If $\gr=0,$ then $\gc_3=\nu, \xi=1,$ and $\mu=0.$   Similarly, (ii)(1) holds if $\xi=0.$


Suppose that  $ \gr,\xi\ne 0.$  Then  $\gc_3 = \frac{ \mu }{1-\xi}  = \frac{ \mu }{\rho},$  and $\gc_3=\frac{\nu}{\xi}.$  In this case $A_1(b)=\ff b+\ff (a+\frac{\mu}{\gr}y).$

$b$ satisfies the fusion rules if and only if $a+(\mu+\nu)\frac{\mu}{\gr}y =(a+\frac{\mu}{\gr}y)^2 \in \ff b + \ff (a+\frac{\mu}{\gr}y),$ which is true if and only if $\mu+\nu = 1.$


So we may assume that $\gro_1,\gro_2 \ne 0, 1.$  Then by \eqref{dec1125} and \eqref{dec11251}, $\gc_2=0$ and $(\gro_1,\gro_2)=(\ga_b,\ga_b).$  We have
\begin{equation}\label{+}
\begin{aligned}
\ga_b  \gc_3  &=  \mu +\gc_3 \xi,\\
\ga_b\gc_3 &=\nu+\gc_3\rho.
\end{aligned}
\end{equation}
Adding we get $2\ga_b\gc_3=\mu+\nu+\gc_3,$ so $\mu+\nu=(2\ga_b-1)\gc_3.$  Subtracting we get $\mu-\nu=(\gr-\xi)\gc_3.$  Thus (iii) holds, and since $z^2\notin \ff b,$ we see that $b$ does not satisfy the fusion rules. \end{proof}



To show the necessity of the hypotheses that $b$ is an axis, consider the following modification of the exceptional algebra:

\begin{example}
    $A = \ff a +\ff b_0 + \ff y +\ff y',$ where $b=   b_0 +y$, $b_0 ^2 = b_0 , \mu+\nu=1,$  $ay = \mu y,$ $ay' = \mu y',$  $y a =\nu y$, $y' a =\nu y'$, $ b _0 y= \nu y+  y'$, $ yb_0=\mu y- y',$   $b_0 y' = \gb_1 y +\gb_1' y',$  $ y'b _0 = \gb_2 y+ \gb_2' y'$, and
    $y^2 = (y')^2=yy' = y'y = 0.$ Note that $  A_{\mu,\nu}(a)= \ff y +\ff y'$, so $a$ is a primitive $(\mu,\nu)$-axis when $\mu,\nu \notin \{0,1\}$.
    Then $b^2 =b_0^2+ b_0y  +yb_0
    = b_0 +(\mu+\nu )y
    +y'-y'=b,$ so $a$ satisfies the fusion rules, and $b$ is idempotent. Note that $\dim A = 4.$

\end{example}

In the  case $A$ is commutative, we have  explicit information given in \cite[Proposition~2.12]{RoSe1} when the axis $b$ also is  primitive and satisfies the fusion rules. See Examples~\ref{4dim2}, \ref{3dim01}, and Theorem~\ref{yrho} for the situation when $b$ is not primitive.
Since we are unable to obtain results when $\dim A_0(a)\ge 2:$

\subsection{Preliminary consequences of the hypothesis $A_0(a)=\ff b_0$}\label{A01}$ $ 

 Having characterized the case $b_0=0,$ we assume from now on that $b_0 \ne 0.$

\begin{lemma}
    \label{Ser23}  The following are equivalent for an axis $b$:
\begin{enumerate}\eroman
    \item   $A_0(a) = \ff b_0,$
       \item
$b_0^2,b_{\mu,\nu}^2 \in \ff a +\ff b_0$ for all $(\mu,\nu)\in S,$
   \item
$A =  \ff a +\ff b_0+\sum_{(\mu,\nu)\in S} \ff b_{\mu,\nu}$, so $\dim A= |S^\circ| +2$.
\end{enumerate}
 In this case, $A_{\mu,\nu}(a)=\ff b_{\mu,\nu}$ for each $(\mu,\nu)\in S^\circ.$
\end{lemma}
\begin{proof} $(i)\Rightarrow (ii)$
By the fusion rules.

     $ (ii)\Rightarrow (iii) $
By Theorem~\ref{Ser02}(ii), the space
    $\ff a +\ff b_0 +\sum_{(\mu,\nu)\in S} \ff b_{\mu,\nu}$
    is closed under multiplication and contains $a$ and $b$, so is all of $A.$

 $ (iii)\Rightarrow (i) $ follows at once, as does the the last assertion.
\end{proof}

\begin{hyp1}\label{hyp1} We assume from now on that $\dim A_0(a)=1,$ so that $A_0(a)=\ff b_0.$
\end{hyp1}

\begin{cor}  \label{Ser24}
    The following are equivalent:
    \begin{enumerate} \eroman
   \item   $A$ is commutative.
     \item   $ab = ba$.
     \item $S^\dagger = \emptyset.$
    \end{enumerate}
\end{cor}
\begin{proof}
$(i)\Rightarrow (ii) $ Obvious.

$(ii) \Rightarrow (iii)$ $\mu=\nu$ for each $(\mu,\nu)\in S^\circ$.

$(iii) \Rightarrow (i)$ Let $y= \ga_y a+ y_0 + \sum y_{\mu,\mu}$ and $z = \ga_z a+ z_0 + \sum z_{\mu,\mu}.$
Then $$yz = \ga_y\ga_z a + y_0z_0 +\sum (\ga_y\mu z_{\mu,\mu} +y_0z_{\mu,\mu}) +\sum (\ga_z\mu y_{\mu,\mu} +y_{\mu,\mu}z_0)+\sum y_{\mu,\mu}z_{\mu,\mu}.$$
Since $A_0(a)=\ff b_0,$ $y_0z_0 = z_0y_0.$   By Lemma~\ref{Ser23},  $A_{\mu,\mu}(a) = \ff b_{\mu,\mu},$ so $y_{\mu,\mu}z_{\mu,\mu} = z_{\mu,\mu} y_{\mu,\mu}$.   By \eqref{62},
$\mu b_{\mu,\mu}b_0- \mu b_0b_{\mu,\mu}  = \ga_b (\mu+\mu-1)(\mu-\mu)b_{\mu,\mu} =0,$ for all $(\mu,\mu)\in S,$ implying $b_0b_{\mu,\mu} = b_{\mu,\mu}b_0.$ Consequently, $yz=zy.$
\end{proof}

Here is a special case in which  the fusion rules are  satisfied.

\begin{example}\label{E1}[The ``$S$-exceptional  algebra'']
Suppose  $S\subseteq \{ (\mu,\nu) \in \ff \times \ff :   \mu+\nu =1\}$, and let $W$ denote the vector space $\ff a + \ff b+ \sum_ {(\mu,\nu)\in S} \ff b_{\mu,\nu}$. We define  multiplication on~$W$ according to the following   rules,  with the sums taken over all $ (\mu,\nu)\in S  $:

The elements  $a $ and $b$ are idempotents,
\begin{equation}\label{eq111}
   ab = \sum \mu b_{\mu,\nu}, \qquad ba = \sum \nu b_{\mu,\nu},
\end{equation}

\begin{equation}\label{eq112}
   a b_{\mu,\nu} =  \mu b_{\mu,\nu} = b_{\mu,\nu} b, \qquad  b_{\mu,\nu}a =  \nu b_{\mu,\nu} = bb_{\mu,\nu},
\end{equation}
\begin{equation}\label{eq113}
    b_{\mu,\nu}  b_{\mu',\nu'}= 0, \quad \forall (\mu,\nu), (\mu',\nu') \in S.
\end{equation}

As in Lemma~\ref{dec7}(i), each $b_{\mu,\nu} \in \lan\lan a,b\ran\ran,$ so we see that $W$ is the algebra $A = \lan\lan a,b\ran\ran $.
Taking $b_0 = b -\sum  _{\mu,\nu}b_{\mu,\nu}$, we see that $A$ has eigenspaces $A_1(a) = \ff a,$ $A_0(a) = \ff b_0,$
and $A_{\mu,\nu}(a) = \ff b_{\mu,\nu}$; $a$ is a weakly primitive axis of type~$S$, and $\ga_b = 0$.    Furthermore
$$b_0^2 = b +\sum  _{\mu,\nu} b_{\mu,\nu}^2 -\sum  _{\mu,\nu} (\nu+\mu) b_{\mu,\nu}  =b -\sum  _{\mu,\nu}   b_{\mu,\nu} =b_0, $$   implying  $a$ satisfies the fusion rules.

Note that \eqref{eq113} also holds with respect to $b$ since the $b_{\mu,\nu}$ also are  eigenvectors for $b.$
Thus, by symmetry, $b$ is a weakly primitive  $T$-axis satisfying the fusion rules, where $T= \{(\nu,\mu) : (\mu,\nu) \in S(a)\},$  and thus $A$ is a weak PAJ, which we call the  {\it $S$-exceptional  algebra}.

In fact if $(0,1), (1,0) \notin S$ then  $A$ is a PAJ.

\end{example}

  \begin{remark}$ $
Example~\ref{E1} provides more examples of weak PAJ's of arbitrary dimension $|S|+2$, under our definition in this paper. In previous papers \cite{RoSe2,RoSe3,RoSe4} we had considered only $|1|$-axes, cf.~Example~\ref{ax1}(vii).
 \end{remark}

Let us record some more information.

 \begin{lemma}\eroman
     \label{fu1} Suppose that   $b$ satisfies the fusion rules, with $b_{\mu,\nu}^2 = 0,$ and let $\varepsilon=\mu+\nu.$
\begin{enumerate}

    \item
        $b_{\mu,\nu}A_{\gr,\xi}(b)=A_{\gr,\xi}(b)b_{\mu,\nu}=0,$  for every $(\gr,\xi)\in S^\circ(b).$

       \item  Let $(\rho,\xi)\in S^\circ(b),$ and  let $y = u_y + y'\in A_{\gr,\xi}(b),$ where $$u_y=\ga_ya+\gc_0 b_0\in \ff a + \ff b_0, \qquad y' =\sum_{(\mu',\nu')\in S^{\circ}}\gc_{\mu',\nu'}b_{\mu',\nu'}.$$ The following assertions hold:
       \begin{enumerate}
           \item Either $\ga_y=\gc_0=0$, i.e., $u_y=0$ and $y\in V$, or $\mu=\nu$ or $\ga_b=1.$
           \item Let $u:=(\ga_b-\frac{1}{\varepsilon})a+b_0.$ Then  $u_y\in\ff u.$


           %
              \end{enumerate}
 \end{enumerate}
\end{lemma}
 \begin{proof}
 (i) Note that by Theorem \ref{Ser02}(vi), $b_{\mu,\nu}\in A_{\gr_1,\xi_1}(b),$ for some $(\gr_1,\xi_1).$
 If $b_{\mu,\nu}\notin A_{\gr,\xi}(b),$ the assertion follows from the fusion rules of $b.$ Suppose $b_{\mu,\nu}\in A_{\gr,\xi}(b).$  Note that $Ab_{\mu,\nu}\subseteq\ff b_{\mu,\nu},$ so for $y\in A_{\gr,\xi}(b),$ $$yb_{\mu,\nu}, b_{\mu,\nu}y\in \ff b_{\mu,\nu}\cap (\ff b+A_0(b))=0.$$

 (ii) Let $w_1, w_2$ be as in Corollary \ref{half11}(i). We first show:
 \begin{equation}\label{w}
            \text{If }w\in \{w_1,w_2\}\text{ is such that }w\ne 0,\text{  then }u_y \in \ff w.
 \end{equation}
We assume $w=w_1\ne 0.$ The argument when $w=w_2$ is similar.
 We have $b_{\mu,\nu}w_1 = 0.$ By (i), $b_{\mu,\nu}y=0,$ and since $b_{\mu,\nu}b_{\mu',\nu'}=0,$ for all  $(\mu',\nu')\in S^{\circ},$ also $b_{\mu,\nu}u_y =0.$ Hence, if $u_y\notin \ff w_1,$ then $b_{\mu,\nu}A=0.$ In particular, $b_{\mu,\nu}a=0,$ so $\nu=0.$  But also $b_{\mu,\nu}b_0=0,$ so, by Theorem \ref{Ser02}(ii), $\varphi_1=0,$ so $w_1=0,$ a ~contradiction.  Hence $u_y\in\ff w_1.$

  (a) Assume that $\ga_b\ne 1.$ By Corollary ~\ref{half11}(ii), $w_i\ne 0,$ so $u_y\in \ff w_i,$ for $i=1,2.$
Suppose also that $u_y\ne 0.$  Then $w_2\in\ff w_1,$  since  $w_1,w_2\in\ff u_y.$  But then, comparing coefficients of $b_0\ne 0$ in Corollary ~\ref{half11},  $\mu w_1=\nu w_2,$ and comparing coefficients of $a$, $\varphi_1\mu=\varphi_2\nu,$ i.e.,  $(1 -\ga_b)\mu^2 +\ga_b(1- \varepsilon)\mu \nu = (1 -\ga_b)\nu^2 +\ga_b(1- \varepsilon)\mu \nu,$ so, since $\varepsilon=\mu+\nu\ne 0,$  $\mu = \nu.$

(b)  If $u_y=0,$ this is clear.

Suppose $\mu=\nu.$  Then $\frac{1}{\mu}w_1=\frac{1-\ga_b\varepsilon}{\varepsilon}a-b_0=(\frac{1}{\varepsilon}-\ga_b)a-b_0$.  But $u_y\in\ff w_1,$
by \eqref{w}.  Hence (b) holds in this case as well.

Suppose $\ga_b=1.$  Then, $\nu\ne 0$ since $w_1\ne 0.$ Then $\frac{1}{\nu}w_1=\frac{1-\varepsilon}{\varepsilon}a-b_0=(\frac{1}{\epsilon}-\ga_b)a-b_0,$ and again (b) holds.
Now (a) completes the proof.


\end{proof}





\subsection{The eigenspaces of $b$}\label{bfus}$ $

So far we have described 2-generated axial algebras  $A = \lan \lan a,b \ran \ran$, where   $a$ is a  weakly primitive $S$-axis satisfying the fusion rules, and $b$ is an axis, with the further restriction  $\dim A_0(a)= 1$, i.e.,  $A_0(a) = \ff b_0.$  Now we turn to the eigenspace decomposition of $A$ with respect to the axis $b$.

Since we determined the structure of $A$ in Corollary~\ref{S1} when $|S^\circ| =1,$ we assume that $|S^\circ| \ge 2.$ For  $(\mu,\nu)\in S^{\circ}$ we write $\varepsilon_{\mu,\nu}=\mu+\nu.$ Also, for $y\in A,$ we write $ y = \ga_y a +\gc_0 b_0 + \sum_{(\mu,\nu)\in S^\circ} \gc_{\mu,\nu} b_{\mu,\nu}. $

\begin{lemma}\label{Ser03}
For any $0 \ne y \in A, $
\begin{enumerate}\eroman
\item
            $by = \ga_b \ga_y a+ \gc_0 b_0^2+\sum_{(\mu,\nu)\in S} \gc_{\mu,\nu} b_{\mu,\nu}^2$  \\
            $$+\sum_{(\mu,\nu)\in S}\left( \left(\frac{ \ga_b \gc _{\mu,\nu} +(1-\ga_b)\gc_0 }{\varepsilon_{\mu,\nu}}\right) \mu  +\left(\frac{   (1-\ga_b) \gc_{\mu,\nu} +\varepsilon_{\mu,\nu}\ga_y +\ga_b(1-\varepsilon_{\mu,\nu})\gc_0   }{\varepsilon_{\mu,\nu}} \right)\nu\right)b_{\mu,\nu}.$$

        \item  $y$ is a   left $\rho$-eigenvector of $b$, if and only if the following two conditions hold:
        \begin{enumerate}
            \item  $\ga_b\ga_y a +\gc_0 b_0 ^2+\sum \gc_{\mu,\nu} b_{\mu,\nu}^2 = \rho (\ga_y a +\gc_0 b_0) .$
        \item  $\left(\frac{ \ga_b \gc _{\mu,\nu} +(1-\ga_b)\gc_0 }{\varepsilon_{\mu,\nu}}\right) \nu  +\left(\frac{(1-\ga_b) \gc_{\mu,\nu} +\varepsilon_{\mu,\nu}\ga_y +\ga_b(1-\varepsilon_{\mu,\nu})\gc_0)}{\varepsilon_{\mu,\nu}}\right)\mu =\rho \gc_{\mu,\nu}  ,$ for all $(\mu,\nu)\in S^\circ$.
    \noindent  In particular, for $(\mu,\mu)\in S^{\circ}$ we get
        \[
        \gc_0   +\varepsilon_{\mu,\mu}(\ga_y -\ga_b\gc_0) =(2\rho-1) \gc_{\mu,\mu}  ,
        \]
         \end{enumerate}

        \item Furthermore, when $y$ is a   left $\rho$-eigenvector of $b$, if $\gc_{\mu,\nu} =0$ for $(\mu,\nu)\in S^\circ,$ then $ \ga_y  =  \left(  -  (1-\ga_b)\frac{  \mu }{\nu  }- \ga_b(1-\varepsilon_{\mu,\nu})   \right)\frac{\gc_0}{\varepsilon_{\mu,\nu}}.$

  \item  $y$ is a   right $\xi$-eigenvector of $b$, if and only if the following two conditions hold:
        \begin{enumerate}
            \item  $\ga_b\ga_y a +\gc_0 b_0 ^2+\sum \gc_{\mu,\nu} b_{\mu,\nu}^2 = \xi (\ga_y a +\gc_0 b_0) .$
        \item  $\left(\frac{ \ga_b \gc _{\mu,\nu} +(1-\ga_b)\gc_0 }{\varepsilon_{\mu,\nu}}\right) \mu  +\left(\frac{   (1-\ga_b) \gc_{\mu,\nu} +\varepsilon_{\mu,\nu}\ga_y +\ga_b(1-\varepsilon_{\mu,\nu})\gc_0   }{\varepsilon_{\mu,\nu}} \right)\nu =\xi \gc_{\mu,\nu}  ,$ for all $(\mu,\nu)\in S^\circ$.
   \end{enumerate}

    \end{enumerate}   \end{lemma}
\begin{proof} $A =  \ff a +\ff b_0+\sum_{(\mu,\nu)\in S} \ff b_{\mu,\nu}$  by Lemma~\ref{Ser23}.
     By the fusion rules, the $\ff a +A_0(a)$ part of $by$ is
\begin{equation}\begin{aligned}
    (\ga_b a +   b_0 )(\ga_y a +\gc_0 b_0 )+& (\sum b_{\mu,\nu})(\sum \gc_{\mu,\nu} b_{\mu,\nu})\\& = \ga_b\ga_y a+\gc_0 b_0^2 +\sum \gc_{\mu,\nu} b_{\mu,\nu}^2.
\end{aligned}
\end{equation}
Using Theorem~\ref{Ser02}(v) and \eqref{622}, the $A_{\mu,\nu}(a) $ part of $by$ is \begin{equation*}
    \begin{aligned}
        &(\ga_b  a +b_0)\gc_{\mu,\nu} b_{\mu,\nu} + b_{\mu,\nu}(\ga_y a + \gc_0 b_0) \\
    & = \left(\frac{  (1 -\ga_b)\nu +\ga_b  \mu}{\varepsilon_{\mu,\nu}}\gc_{\mu,\nu} +\ga_y    \nu+  \frac{(1 -\ga_b)\mu +\ga_b(1- \varepsilon_{\mu,\nu}) \nu}{\varepsilon_{\mu,\nu}}\gc_0  \right) b_{\mu,\nu} \\& = \left(\left(\frac{ \ga_b \gc _{\mu,\nu} +(1-\ga_b)\gc_0 }{\varepsilon_{\mu,\nu}}\right) \mu  +\left(\frac{   (1-\ga_b) \gc_{\mu,\nu} +\varepsilon_{\mu,\nu}\ga_y +\ga_b(1-\varepsilon_{\mu,\nu})\gc_0   }{\varepsilon_{\mu,\nu}} \right)\nu \right)b_{\mu,\nu} .
    \end{aligned}
\end{equation*}

    (i), (ii) are immediate.

    (iii) By (ii),
     $\left(\frac{(1-\ga_b)\gc_0 }{\varepsilon_{\mu,\nu}}\right) \mu  +\left(\frac{   \varepsilon_{\mu,\nu}\ga_y +\ga_b(1-\varepsilon_{\mu,\nu})\gc_0   }{\varepsilon_{\mu,\nu}} \right)\nu =\gc_{\mu,\nu}\rho=0 ,$ so
     \begin{equation}
         \begin{aligned}
             \varepsilon_{\mu,\nu} \ga_y  & =\frac{  -  (1-\ga_b)\gc_0\mu - \ga_b(1-\varepsilon_{\mu,\nu})\gc_0\nu}{\nu  }\\& =\left(  -  (1-\ga_b)\frac{  \mu }{\nu  }- \ga_b(1-\varepsilon_{\mu,\nu})   \right)\gc_0.
         \end{aligned}
     \end{equation}

     (iv) By left-right symmetry.

    \end{proof}

Here is a common sort of eigenvector.

\begin{lemma}
    \label{Ser041}     Assume that $  y=\sum_{(\mu,\nu)\in S^{\circ}}\gc_{\mu,\nu}b_{\mu,\nu}\in V $.
    \begin{enumerate}\eroman
        \item  Suppose that $y$ is a $(\rho,\xi)$-eigenvector of~$b$.  Then
\begin{equation}\label{896}
    \rho  = \frac{(1-\ga_b)   {  \nu} +\ga_b\mu  }{\varepsilon_{\mu,\nu}}, \qquad  \xi  = \frac{(1-\ga_b)   {  \mu} +\ga_b\nu  }{\varepsilon_{\mu,\nu}}.
\end{equation}
 for each $(\mu,\nu)$ such that $\gc_{\mu,\nu}  \ne 0$. (Such $(\mu,\nu)\in S ^\circ$ exists since $y\ne 0.$)   Consequently $\rho + \xi =1.$

 Also for each $(\mu,\nu)$ such that $\gc_{\mu,\nu}\ne 0,$ $\mu(\gr-\ga_b)= \nu(\xi- \ga_b).$  In addition, $\sum_{(\mu,\nu)\in S^{\circ}}\gc_{\mu,\nu}b_{\mu,\nu}^2=0.$ Furthermore,
 \begin{enumerate}\eroman
     \item If $\ga_b=\half,$ then $\rho=\xi=\half$.
     \item If $\gr=\ga_b\ne\half,$ then $\mu=0,$ for each $(\mu,\nu)\in S^{\circ}$ such that $\gc_{\mu,\nu}\ne 0.$
     \item  If $\xi=\ga_b\ne\half,$ then $\nu=0,$ for each $(\mu,\nu)\in S^{\circ}$ such that $\gc_{\mu,\nu}\ne 0.$
     \item If $\gr\ne\ga_b\ne\xi,$ then $\frac{\nu}{\mu}=\frac{\gr-\ga_b}{\xi-\ga_b},$ for each $(\mu,\nu)\in S^{\circ}$ such that $\gc_{\mu,\nu}\ne 0.$
 \end{enumerate}
\item Conversely, if   $\rho + \xi =1,$ and $\mu(\rho-\ga_b)= \nu(\xi-\ga_b),$ for all $(\mu,\nu)\in S^{\circ}$ such that $\gc_{\mu,\nu}\ne 0,$ and $\sum_{(\mu,\nu)\in S^{\circ}}\gc_{\mu,\nu}b_{\mu,\nu}^2=0,$ then $y$ is a $(\rho,\xi)$-eigenvector of~$b.$  In particular,

\item
The vector $ \sum_{(\mu,\nu)\in S} \gc'_{\mu,\nu}  b_{\mu,\nu}$ is a $(\half,\half) $ eigenvector of $b,$ if and only if, either $\ga_b=\half,$ or $\ga_b\ne\half$ and $\gc'_{\mu,\nu}=0,$ for all $(\mu,\nu)\in S^{\dagger},$ and $\sum_{(\mu,\mu)\in S}   \gc'_{\mu,\mu}  b_{\mu,\mu}^2 = 0.$
    \end{enumerate}

\end{lemma}
\begin{proof} (i) If $\ga_b = \half$ then, by Lemma \ref{Ser03}(ii)(b), and (iv)(b) $\rho =\xi=\frac{\half \mu +\half \nu}{\mu+\nu}= \half.$
 Also the equality of  Lemma~\ref{Ser03}(ii)(a) holds, so (a) follows.

So assume $\ga_b\ne \half.$
   By Lemma~\ref{Ser03}(ii)(b), if $\gc_{\mu,\nu}\ne 0,$
    \begin{equation}
        \label{Ser043} \rho  = \frac{(1-\ga_b)   {  \nu} +\ga_b\mu  }{\varepsilon_{\mu,\nu}} ,
    \end{equation}
    and by Lemma \ref{Ser03}(iv)(b),
\begin{equation}
    \label{Ser042}  \xi  = \frac{(1-\ga_b)   {  \mu} +\ga_b\nu  }{\varepsilon_{\mu,\nu}} .
\end{equation}
Adding yields $\rho + \xi = \frac{(1-\ga_b)  \varepsilon_{\mu,\nu}  +\ga_b \varepsilon_{\mu,\nu}} {\varepsilon_{\mu,\nu}}=1.$  Also, $\sum_{(\mu,\nu)\in S^{\circ}}\gc_{\mu,\nu}b_{\mu,\nu}^2=0,$ by Lemma~\ref{Ser03}(ii)(a).

 Now by \eqref{Ser043},
 $(\mu +\nu)\rho = (1-\ga_b)   {  \nu} +\ga_b\mu  $ so
 $\mu(\gr-\ga_b)=\nu  ((1-\ga_b) -\rho) = \nu(\xi- \ga_b).$  For the next assertions,
substitute into Lemma~\ref{Ser03}(i),(ii), canceling $\gc_{\mu,\nu}.$ The remaining parts of (i) follow from  this.

 (ii)  $\mu(\rho- \ga_b)=\nu(\xi-\ga_b) = \nu(1-\rho-\ga_b)$ implies
 $$\rho = \frac{(1 -\ga_b)\nu +\ga_b\mu}{\mu+\nu} , $$ and together with $\xi = 1-\rho$ we  have \eqref{896}.  Then the equalities given in Lemma~\ref{Ser03}(ii)(a), (ii)(b) and (iv)(b) are satisfied, so $y$ is a $(\gr,\xi)$-eigenvector of~$b.$

 (iii) This follows from (i) and (ii).
\end{proof}

\begin{thm}\label{Ser04} Take $V$ as in  \eqref{Vdef}, and assume that $$ y = \ga_y a +\gc_0 b_0 + \sum \gc_{\mu,\nu} b_{\mu,\nu}$$ is a nonzero  $(\rho,\xi)$-eigenvector of $b$.
    \begin{enumerate}\eroman

    \item If $\rho\ne \xi$, then $\gc_0 = 0 =  \ga_y $ so $y\in V$.

 \item Suppose that $\rho = \xi.$   Then
 \begin{enumerate}
     \item
 \begin{equation}\label{11}
 {(2\rho -1)} \gc_{\mu,\nu}    =    \varepsilon_{\mu,\nu}\ga_y+ (1-\varepsilon_{\mu,\nu}\ga_b) \gc_0= \varepsilon_{\mu,\nu} (\ga_y -\ga_b \gc_0) + \gc_0.
 \end{equation}
for all $(\mu,\nu)\in S^\circ,$ where $\varepsilon_{\mu,\nu}=\mu+\nu.$
  In particular, when $\rho= \half$, $\varepsilon :={\varepsilon_{\mu,\nu}}  $ is independent of $(\mu,\nu)\in S^\circ$,
\begin{equation}
    \label{141} \ga_y= (\ga_b -\frac{1}{\varepsilon})  \gc_0,
\end{equation}
and
  \begin{equation}\label{112}
  (\rho-\ga_b) \gc_{\mu,\nu}    = (1- \ga_b)\gc_0
 \end{equation}
 for all $(\mu,\nu)\in S^\dagger.$

 \item
 In particular, if $\gc_0=0,$ and $\gr\ne\ga_b,$ the $y$ belongs to the commutative subalgebra $\ff a+\ff b_0+\sum_{(\mu,\nu)\in S^{\circ}}\ff b_{\mu,\mu}.$
 \end{enumerate}

\item Suppose that $y\in V$. Then the following assertions are equivalent:
\begin{enumerate} \eroman
    \item  $\rho = \xi $.

   \item $\rho = \xi=\half$.

  \item $A$ is commutative or $\ga_b=\half$.
\end{enumerate}

  \item For $(\mu,\mu)\in S,$ $b_{\mu,\mu} $ is a $(\half,\half)$ eigenvector of $b$ if and only if $b_{\mu,\mu}^2 = 0.$
 \item  Suppose that $\rho = \xi$ and $y\notin V$.
\begin{enumerate}
 \item If $\gc_{\mu ,\nu }= 0 $ for some $(\mu,\nu)\in S^{\circ}$, then $\gc_0\ne 0$ and  $\varepsilon_{\mu,\nu} = \frac{\gc_0}{\ga_b \gc_0 -\ga_y}.$ Consequently, $\ga_y = 0$ if and only if $\varepsilon_{\mu,\nu} = \frac{1}{\ga_b}$.

   \item If    $\rho = \half$, then $\gc_0\ne 0$  and $\varepsilon =\mu'+\nu'$ is fixed for all $(\mu',\nu')\in S^{\circ}.$

   Let $\mu =\frac{\varepsilon}{2} $, and
 write $b_{\mu,\mu}^2 = \theta'_1 a +\theta'_0b_0.$
\begin{itemize}
\item [(1)] If $A$ is commutative, then either $A=\ff a+\ff b_0,$ or $A=\ff a+\ff b_0+\ff b_{\mu,\mu}.$

\item[(2)]
If $\ga_b=\half,$ then $S^{\dagger}=\emptyset$. Also, normalizing $\gc_0 =1,$ we have,
$$y=\left(\ga_b -\frac{1}{\varepsilon}  \right) a +b_0 + \gc_{\mu,\mu}b_{\mu,\mu}\in A_{\half,\half}(b)$$ iff
\begin{equation}\label{half2}
\textstyle{(\mu,\mu)\in S^{\circ},\quad \theta'_0 \ne 0,\quad \theta'_1= \half\theta'_0,\quad \gc_{\mu,\mu}  =1 - \frac{1}{2\theta'_0}.}
\end{equation}
Thus in this case, \eqref{half2} holds  iff  $A$ is commutative and $3$-dimensional and $A_{\half,\half}(b)=\ff y$. Otherwise $A_{\half,\half}(b)=0.$

    \item[(3)] If $\ga_b\ne\half,$ then
    \[
\textstyle{y= \left(\ga_b -\frac{1}{\varepsilon}  \right) a +b_0 + \gc_{\mu,\mu}b_{\mu,\mu}+2\frac{\ga_b-1}{2\ga_b-1}\sum_{(\mu,\nu)\in S^{\dagger}}b_{\mu,\nu}},
\]
and $y\in A_{\half,\half}(b)$ iff
    \begin{equation}\label{half13}
    \begin{aligned}
&\textstyle{(\mu,\mu)\in S^{\circ},\quad \theta'_0 \ne 0,\quad \theta'_1 \ne \ga_b \theta'_0,\quad \gc_{\mu,\mu}  =1 - \frac{1}{2\theta'_0}\text{ and }}\\
& \textstyle{\mu =  \frac{(1-2\ga_b)\theta'_0}{\theta'_1-\ga_b\theta'_0}.}
\end{aligned}
\end{equation}
Thus, in this case, $A_{\half,\half}(b) = \ff y+ V_{\half,\half}(b)$ iff \eqref{half13}  holds, and otherwise $A_{\half,\half}(b)=V_{\half,\half}(b).$ Furthermore $\bar S = \emptyset.$

    \end{itemize}
\end{enumerate}

    \end{enumerate}
\end{thm}
\begin{proof}
    (i)
The symmetric equation to Lemma~\ref{Ser03}(ii)(a), computing $yb = \xi y$,
yields $$ \ga_b\ga_y a +\gc_0 b_0^2 +\sum \gc_{\mu,\nu} b_{\mu,\nu}^2 = \xi (\ga_y a + \gc_0 b_0),$$ which, together with Lemma~\ref{Ser03}(ii)(a), yields  $\ga_y a + \gc_0 b_0 = 0,$  since $\rho\ne \xi;$ hence $\ga_y =\gc_0 =0$.

 (ii)
 There is some $(\mu,\nu)\in S^\circ$ for which $\gc_{\mu,\nu}\ne 0.$
We use Lemma~\ref{Ser03}(ii)(b) for $\rho$ and the analogous equation for $\xi=\rho$ on the right, to get:
\begin{equation}\label{same}
    \begin{aligned}
    &\textstyle{\left(\frac{ \ga_b \gc _{\mu,\nu} +(1-\ga_b)\gc_0 }{\varepsilon}\right) \mu  +\left(\frac{   (1-\ga_b) \gc_{\mu,\nu} +\varepsilon\ga_y +\ga_b(1-\varepsilon)\gc_0   }{\varepsilon} \right)\nu =\rho \gc_{\mu,\nu} }\\
    & = \textstyle{\left(\frac{ \ga_b \gc _{\mu,\nu} +(1-\ga_b)\gc_0 }{\varepsilon}\right) \nu  +\left(\frac{   (1-\ga_b) \gc_{\mu,\nu} +\varepsilon\ga_y +\ga_b(1-\varepsilon)\gc_0   }{\varepsilon} \right)\mu ,}
    \end{aligned}
\end{equation}
 Adding the left and right sides yields
 $$  \ga_b \gc _{\mu,\nu} +(1-\ga_b)\gc_0 +  (1-\ga_b) \gc_{\mu,\nu} +\varepsilon\ga_y +\ga_b(1-\varepsilon)\gc_0    =2\rho\gc_{\mu,\nu},$$ which yields
 \eqref{11}. Equation \eqref{141} is an immediate consequence.

When $\mu \ne \nu,$ equating the two sides of  \eqref{same} yields

$$  \ga_b \gc _{\mu,\nu} +(1-\ga_b)\gc_0 = (1-\ga_b) \gc_{\mu,\nu} +\varepsilon\ga_y +\ga_b(1-\varepsilon)\gc_0   ,$$

or
  \begin{equation}
     \label{111a}   {\varepsilon}\ga_y     =(2\ga_b -1)\gc_{\mu,\nu} +(1- 2\ga_b +\varepsilon \ga_b)\gc_0.
 \end{equation}

 Plugging into \eqref{11}  yields
 $$(2\rho-1) \gc_{\mu,\nu} -  (1-\varepsilon\ga_b)\gc_0    =(2\ga_b -1)\gc_{\mu,\nu} +(1- 2\ga_b +\varepsilon \ga_b)\gc_0,$$
 so $(2\gr-2\ga_b)\gc_{\mu,\nu}=(2-2\ga_b)\gc_0,$
or \eqref{112}.
\medskip


(iii) $(a) \Rightarrow (b)$ $\rho+\xi =1$ by Lemma~\ref{Ser041}.

$(b) \Rightarrow (c)$
 Rewrite \eqref{Ser043} as $\rho \mu + \rho \nu =(1-\ga_b)   {  \nu} +\ga_b\mu ,$ or
 \begin{equation}
     \label{Ser046}  (\rho - \ga_b)\mu = (1-\rho -\ga_b) \nu=(\xi-\ga_b)\nu.
 \end{equation}

Since $\xi = \rho = \half,$ we have either $\ga_b = \half$, or for  $\ga_b \ne \half$, $\mu=\nu$ for each $(\mu,\nu)\in S^\circ,$ so $A$ is commutative.

$(c) \Rightarrow (a)$ When $A$ is commutative, \eqref{Ser046} implies
$\rho - \ga_b= \xi-\ga_b ,$ so $\rho = \xi.$
When $\ga_b = \half,$
\eqref{Ser043} says
$\rho  = \frac{\half  {  \nu} +\half\mu  }{\varepsilon} = \half, $ and symmetrically $\xi = \half.$
\medskip

(iv) By Theorem \ref{Ser02}(vi).
\medskip

(v)
(a)  If $\gc_{\mu,\nu}=0$ for   $(\mu,\nu)\in S^{\circ}$ and $\gc_0=0,$ then $\ga_y=0$ by \eqref{11}, so $y\in V,$ a contradiction. By \eqref{11}, $\ga_b\gc_0-\ga_y\ne 0,$ and the equation for $\varepsilon _{\mu,\nu} $ follows.  Then if $\ga_y=0,$ we see that $\varepsilon_{\mu,\nu} =\frac{1}{\ga_b}.$
\medskip

(b) Since $y\notin V,$ \eqref{141} shows $\gc_0\ne 0$.
By \eqref{11}, $\varepsilon _{\mu',\nu'}  (\ga_y -\ga_b \gc_0) + \gc_0=0.$  Since $\gc_0\ne 0,$ normalizing $\gc_0=1,$ we see that that $\varepsilon =\frac{1}{\ga_b-\ga_y}$   is fixed for all $(\mu',\nu')\in S^{\circ},$ and also $\ga_y=\ga_b-\frac{1}{\varepsilon}.$
\medskip

(1) If $(\mu,\mu)\in S^\circ$ then $\mu = \frac{\varepsilon}{2},$ so of course (1) holds.
\medskip

(2)\&(3)
If $S^{\dagger}\ne\emptyset,$ we can take $(\mu,\nu)\in S^{\dagger}.$
 If $\ga_b=\half,$ then, by \eqref{112},
$\gc_0=0,$ a contradiction. Hence, in this case, $\ga_b\ne\half,$ so $\bar S = \emptyset$ by Lemma~\ref{b0sq}, that is $b_{\mu,\nu}^2=0,$ for each $(\mu,\nu)\in S^{\dagger}.$ This also shows that if $\ga_b=\half,$ then $S^{\dagger}=\emptyset.$ Also,  \eqref{112} shows $\gc_{\mu,\nu}=2\frac{\ga_b-1}{2\ga_b-1}$ when $\ga_b\ne \half.$

We need to check that $y$ satisfies the condition of Lemma~\ref{Ser03}(ii)(a). In case $\ga_b\ne \half,$ we have
$$\textstyle{\ga_b(\ga_b -\frac{1}{\varepsilon} ) a + b_0 ^2+ \gc_{\mu,\mu} b_{\mu,\mu}^2 +2\frac{\ga_b-1}{2\ga_b-1}\sum_{(\mu,\nu)\in S^\dagger} b_{\mu,\nu}^2 = \half \left((\ga_b -\frac{1}{\varepsilon}  ) a + b_0\right).}$$
Since $\sum_{(\mu,\nu)\in S^\dagger} b_{\mu,\nu}^2=0,$ both in case $\ga_b\ne \half$ and in case $\ga_b=\half$ (as $S^{\dagger}=\emptyset$),
 substituting for $b_0^2 = b_0   + (\ga_b-   \ga_b^2 )a - b_{\mu,\mu}^2   - \sum _{(\mu,\nu)\in S^\dagger}  b_{\mu,\nu}^2$ by Theorem~\ref{Ser02}(vii), 
$$\textstyle{  \ga_b(\ga_b -\frac{1}{\varepsilon}) a +  b_0   + (\ga_b-   \ga_b^2 )a +(\gc_{\mu,\mu}  -1) b_{\mu,\mu}^2  = \half \left((\ga_b -\frac{1}{\varepsilon}  ) a + b_0\right) ,}$$

or $$\textstyle{ -\frac{\ga_b}{\varepsilon} a + \half b_0   + \ga_b a  +(\gc_{\mu,\mu}  -1) b_{\mu,\mu}^2    = \half (\ga_b -\frac{1}{\varepsilon}  ) a,  }$$

or \begin{equation}\textstyle{
 \frac{1}{2\varepsilon}(\ga_b(\varepsilon-2)+1) a + \half b_0   +(\gc_{\mu,\mu}  -1) b_{\mu,\mu}^2  = 0.}
\end{equation}

If $(\mu,\mu)\notin S^{\circ},$ or $\gc_{\mu,\mu}  =1$ or $ b_{\mu,\mu}^2  = 0,$ we have a contradiction.
Thus
\begin{equation}
    \label{193} \frac{1}{2\varepsilon}(\ga_b(\varepsilon-2)+1) =    (1- \gc_{\mu,\mu}  ) \theta'_1, \qquad  \half   =  (1- \gc_{\mu,\mu}  ) \theta'_0 .
\end{equation} Hence $\theta'_0\ne 0,$ and $ \gc_{\mu,\mu}  =1 - \frac{1}{2\theta'_0},$ and dividing the two parts of
 \eqref{193} yields $ \ga_b+\frac{1-2\ga_b}{\varepsilon} = \frac{\theta'_1}{\theta'_0}.$  If $\ga_b\ne\half,$ then $\frac{1}{\gre} =  \frac{\frac{\theta'_1}{\theta'_0}-\ga_b}{1-2\ga_b}=\frac{\theta'_1-\ga_b\theta'_0}{(1-2\ga_b)\theta'_0}$. Hence $\theta'_1 \ne \ga_b \theta'_0,$ and $\mu =  \frac{(1-2\ga_b)\theta'_0}{2(\ga_b\theta'_0 - \theta'_1)}$. If $\ga_b=\half,$ then $\theta'_1=\half\theta'_0.$ \end{proof}

\begin{remark}\label{rhohalf}
    If $y\in V$ is a $(\rho,\rho) $-eigenvector of $b$ and $\ga_b =\half$ (which is the case when $A$ is not commutative, by Theorem~\ref{Ser04}(iii)), then $\rho = \half$ by Lemma~\ref{Ser041}, and then
$\dim A_{\rho,\rho}(b)$ could be rather large, since we could have several independent vectors in  Lemma~\ref{Ser041}(iii).
\end{remark}


In what follows,   assume that $$ y = \ga_y a +\gc_0 b_0 + \sum \gc_{\mu,\nu} b_{\mu,\nu}$$ is a nonzero  $(\rho,\rho)$-eigenvector of $b$.

\begin{thm}\label{eigcom}
Assume that $\gr\ne\half.$
Write  $\varepsilon_{\mu,\nu} = \mu +\nu$, $ \sum _{(\mu,\nu)\in S}{\mu}b_{\mu,\nu}^2 = \theta_1 a + \theta_0 b_0 $, and
 $ \sum_{(\mu,\nu)\in S} b_{\mu,\nu}^2 = \theta'_1 a + \theta'_0b_0$.  Then  $2( \theta_1 a + \theta_0 b_0 )= \sum _{(\mu,\nu)\in S}\varepsilon_{\mu,\nu}b_{\mu,\nu}^2 $.
\begin{enumerate}\eroman
\item Suppose that $ \gc_0 =0.$ Then
$\ga_y\ne 0$ and  $\gc_{\mu,\nu}= \frac{\varepsilon_{\mu,\nu}}{2\rho -1} \ga_y$ for all $(\mu,\nu)\in S^\circ,$ i.e., $y = \ga_y y_\rho,$ where $$y_\rho = a +  \frac{1}{2\rho -1}\sum_{(\mu,\nu)\in S}  \varepsilon_{\mu,\nu}  b_{\mu,\nu}.$$ Also, $\theta_0= 0,$ and $\rho-\ga_b = \frac{2}{2\rho -1}\theta_1.$  Furthermore if $S^{\dagger}\ne\emptyset,$ then $\gr=\ga_b.$

Conversely, the following three conditions together imply that   $y_\rho$ is an eigenvector of $b:$
\begin{itemize}
    \item[(1)] $\theta_0= 0,$

    \item[(2)]$\gr-\ga_b=\frac{2}{2\rho -1}\theta_1,$

    \item[(3)] $\gr=\ga_b$ when $A$ is not commutative.
\end{itemize}



\item Assume that  $\gc_0\ne 0,$  and  normalize $\gc_0=1.$
Let
         \begin{equation}\label{72} y_\rho'=\ga_b a +b_0 + \frac{1}{2\rho -1} \sum b_{\mu,\nu}.  \end{equation}
then
          \begin{equation}\label{71} y = \gc_\rho y_\rho + y'_\rho,
 \end{equation}
where $\gc_\rho = \ga _y- \ga_b $ implying $\dim A_{\rho,\rho}(b) \le 2,$
and
     \begin{enumerate}


\item
\begin{equation}\label{5.4}
\gc_{\mu,\nu}    =  \frac{\varepsilon_{\mu,\nu}(\ga_y-\ga_b   ) +1}{2\gr-1}.
\end{equation}
\begin{equation}\label{671}
    0   = \ga_y(\ga_b + \frac{2} {2\rho-1} \theta_1   -\rho )+(\ga_b-   \ga_b^2 )    -\frac  {2\ga_b}{2\rho-1} \theta_1  +\frac  {2(1-\rho )}{2\rho-1} \theta'_1
    \end{equation}
 or
    \begin{equation}\label{theta1}
2\theta_1(\ga_y-\ga_b)= (\ga_y-\ga_b)(\gr-\ga_b)(2\gr-1)+\ga_b(\gr-1)(2\gr-1)+2(\gr-1)\theta'_1.
    \end{equation}
    \begin{equation}\label{672}
    0   =  \frac{2\theta_0} {2\rho-1} \ga_y  + (1-\rho) -\frac  {2\ga_b\theta_0 }{2\rho-1}  +\frac  {2(1-\rho )}{2\rho-1} \theta'_0,
    \end{equation}
    or
    \begin{equation}\label{theta0}
2\theta_0(\ga_y-\ga_b)=(2\gr-1)(\gr-1)+2(\gr-1)\theta'_0.
    \end{equation}
Also,
\begin{equation}\label{NC}
(\gr-\ga_b)\gc_{\mu,\nu}=1-\ga_b,\quad\forall (\mu,\nu)\in S^{\dagger}.
\end{equation}

\item
For $\gr=1,$
$y_\gr' = b \in A_{1}(b)$.  For  $\gr\ne 1,$ $y_\rho'\in A_{\rho,\rho}(b)$ if and only if  $\theta_0' =  \half - \rho$ and $\theta_1' = \ga_b (\half -\rho) $, and  $\ga_b=\half$ when $A$ is not commutative.

\item
If $y_{\gr}\notin A_{\gr,\gr}(b),$ then $y$ is determined as follows.  If $\gr-\ga_b\ne \frac{2}{2\gr-1}\theta_1,$ then $\ga_y$ is given by \eqref{671} and $\gc_{\mu,\nu}$ by \eqref{5.4}.  If $\theta_0\ne 0,$ then $\ga_y$ is given by \eqref{672} and $\gc_{\mu,\nu}$ by \eqref{5.4}.  Finally if both $\gr-\ga_b= \frac{2}{2\gr-1}\theta_1,$ and $\theta_0=0,$ then by (i), $A$ is not commutative, and $\gr=\ga_b,$ so $\gc_{\mu,\nu}$ are given by \eqref{NC}, for all $(\mu,\nu)\in S^{\dagger},$ and by \eqref{5.4}, either $\ga_y=\ga_b,$ or $\ga_y$ is given by \eqref{5.4}, and again all $\gc_{\mu,\nu}$  are given by \eqref{5.4}.
\end{enumerate}

\item
\begin{enumerate}
\item
 $\dim A_{\gr,\gr}(b)=2,$ iff $A_{\gr,\gr}(b)=\ff y_{\gr}+ \ff y'_{\gr}.$

\item
$A_{\gr,\gr}=\ff y_{\gr}$ iff (i) (1), (2) and (3) hold, and $y'_{\gr}\notin A_{\gr,\gr}(b).$


\item
\begin{itemize}
\item[(1)]
For $A$   not commutative, $\dim A_{\gr,\gr}(b)=2$ implies that $\gr=\ga_b=1.$
    \item [(2)]
$A_1(b)=\ff b + \ff \left(a + \sum_{(\mu,\nu)\in S}  \varepsilon_{\mu,\nu}  b_{\mu,\nu}\right)$ iff $ \sum _{(\mu,\nu)\in S}{\mu}b_{\mu,\nu}^2 = \frac{1-\ga_b}{2}a.$

\end{itemize}
\item If $\gr\ne 1,$ then $\dim A_{\gr,\gr}(b)=2$ iff
$
\gr=\half-\theta'_0=\ga_b+\frac{2}{2\gr-1}\theta_1,
$
and, in case $\ga_b\ne 0,$ also $\gr=\half-\frac{\theta'_1}{\ga_b}.$
In particular $A$ is commutative, and there is at most one $\gr\not\in \{\half, 1\}$ for which  $\dim A_{\gr,\gr}(b)=2.$


         \item    If $\theta_0 \ne 0$, then there are at most three possible values of $\rho$  defining a $(\rho,\rho)$-eigenvector,
         and $\dim A_{\rho,\rho}(b) = 1$.


\item If $\theta_0 = 0,$ and $\gr\ne 1,$ then $\rho =  \half-\theta_0'$; if moreover    $\rho-\ga_b \ne \frac{2}{2\rho -1}\theta_1,$   then $\dim A_{\gr,\gr}(b)= 1.$
          \end{enumerate}


    \end{enumerate}

\end{thm}
\begin{proof}  \eqref{61} says $\sum _{(\mu,\nu)\in S}{\mu}b_{\mu,\nu}^2 =\sum _{(\mu,\nu)\in S}{\nu}b_{\mu,\nu}^2$, so
$$2 \sum _{(\mu,\nu)\in S}{\mu}b_{\mu,\nu}^2 = \sum _{(\mu,\nu)\in S}({\mu}+\nu)b_{\mu,\nu}^2 .$$

(i) $\ga_y\ne 0,$ else, by \eqref{11}, $y=0.$ Then, by \eqref{11}, $\gc_{\mu,\nu}= \frac{\varepsilon_{\mu,\nu}}{2\rho -1} \ga_y,$ for all $(\mu,\nu)\in S^{\circ}.$ In particular $\gc_{\mu,\nu}\ne 0.$ Then
  Lemma~\ref{Ser03}(ii)(a) says  $$\ga_b\ga_y a  +\sum  \ga_y\frac{\varepsilon_{\mu,\nu}}{2\rho -1}  b_{\mu,\nu}^2 = \rho \ga_y a  .$$

Since $\ga_y\ne 0$,
  $( \rho-\ga_b )a = \frac{2}{2\rho -1}  \sum  \mu b_{\mu,\nu}^2 .$ Hence $\theta_0 = 0$
and $\rho-\ga_b = \frac{2}{2\rho -1}\theta_1.$  Also, if $A$ is not commutative, then given $(\mu^{\dagger},\nu^{\dagger})\in S^{\dagger}$, $\gc_{\mu^{\dagger},\nu^{\dagger}}\ne 0,$ so by \eqref{112}, $\gr=\ga_b,$ since $\gc_0=0.$

Conversely, when (1), (2), and (3) are satisfied,   we get the conditions of ~Lemma~\ref{Ser03}(ii), so $y_\rho$ is a $(\rho,\rho)$ eigenvector.

 \medskip



(ii)
(a)
By \eqref{11},
\[
\gc_{\mu,\nu}    =  \frac{\varepsilon_{\mu,\nu}(\ga_y-\ga_b   ) +1}{2\gr-1},
\]
for all $(\mu,\nu)\in S^{\circ}.$ so
\[
    \begin{aligned}
 y =\ga_y a +b_0 +   \frac{1}{2\gr-1}\sum _{(\mu,\nu) \in S} (\varepsilon_{\mu,\nu}(\ga_y-\ga_b   ) +1) b_{\mu,\nu}\\ =  \ga_y a  +\frac{1}{2\gr-1}\ga_y\sum _{(\mu,\nu) \in S}\!  \varepsilon_{\mu,\nu} b_{\mu,\nu} +b_0 +   \frac{1}{2\gr-1}\sum _{(\mu,\nu) \in S} (-\varepsilon_{\mu,\nu} \ga_b    +1) b_{\mu,\nu},
    \end{aligned}
\] yielding \eqref{71}. In particular $A_{\rho,\rho}(b) \subseteq \ff y_\rho + \ff y'_\rho,$ implying $\dim A_{\rho,\rho}(b) \le 2. $

Now
Lemma~\ref{Ser03}(ii)(a), together with    Theorem~\ref{Ser02}(vii), says  \begin{equation}\label{364}
      \begin{aligned}
           \rho (& \ga_y a + b_0) = \ga_b\ga_y a + b_0 ^2+\sum  \frac{\varepsilon_{\mu,\nu}(\ga_y-\ga_b   ) +1 }{2\rho-1}b_{\mu,\nu}^2 \\ & = \ga_y(\ga_b a + \frac{2} {2\rho-1} \sum {\mu}b_{\mu,\nu}^2 )+ b_0 ^2+\frac  {1}{2\rho-1} \sum  (-\varepsilon_{\mu,\nu} \ga_b    +1 )b_{\mu,\nu}^2 \\ & = \ga_y(\ga_b a + \frac{2} {2\rho-1} \sum {\mu}b_{\mu,\nu}^2 )+ b_0 + (\ga_b-   \ga_b^2 )a  \\& \qquad  - \sum  b_{\mu,\nu}^2+\frac  {1}{2\rho-1} \sum  (-\varepsilon_{\mu,\nu} \ga_b    +1 )b_{\mu,\nu}^2 \\ & = \ga_y(\ga_b a + \frac{2} {2\rho-1} \sum {\mu}b_{\mu,\nu}^2 )+ b_0 + (\ga_b-   \ga_b^2 )a  \\& \qquad +\frac  {2}{2\rho-1} \sum  (-\mu \ga_b    +1-\rho )b_{\mu,\nu}^2.
      \end{aligned}  \end{equation}

    Substituting for $ \sum {\mu}b_{\mu,\nu}^2  $ and
 $ \sum b_{\mu,\nu}^2  $ gives us \begin{equation}\label{6780}
      \begin{aligned} 0 & = \ga_y(\ga_b a + \frac{2} {2\rho-1} (\theta_1 a + \theta_0 b_0) -\rho a)+ (1-\rho) b_0 + (\ga_b-   \ga_b^2 )a  \\& \qquad -\frac  {2 \ga_b}{2\rho-1} (\theta_1 a + \theta_0 b_0   )  +\frac  {2(1-\rho )}{2\rho-1} (\theta'_1 a + \theta'_0b_0)
      \end{aligned}  \end{equation}
 Matching components of $a$ and $b_0$
 yields
 \[
    0   = \ga_y(\ga_b + \frac{2} {2\rho-1} \theta_1   -\rho )+(\ga_b-   \ga_b^2 )    -\frac  {2\ga_b}{2\rho-1} \theta_1  +\frac  {2(1-\rho )}{2\rho-1} \theta'_1
    \]
    and
    \[
    0   =  \frac{2\theta_0} {2\rho-1} \ga_y  + (1-\rho) -\frac  {2\ga_b\theta_0 }{2\rho-1}  +\frac  {2(1-\rho )}{2\rho-1} \theta'_0.
\]
This show \eqref{671} and \eqref{672}. Finally \eqref{5.4} comes from \eqref{112}.
\medskip

\indent (b)
    Taking $y=y'_{\rho},$ we have $\ga_y=\ga_b.$  If $\gr=1,$ then $y'_{\rho}=b.$  Suppose $\rho\ne 1.$ By \eqref{671},
    \[
    (1-\rho)\ga_b+\frac  {2(1-\rho )}{2\rho-1} \theta'_1=0.
    \]
    Cancelling $\gr-1,$ we get $\theta_1'=\frac{\ga_b(1-2\gr)}{2},$ as asserted.  Also, by \eqref{672}, $$1+\frac  {2}{2\rho-1} \theta'_0=0,$$ so $\theta'_0=\half-\gr,$ as asserted.  Suppose $A$ is not commutative.  Since $\gc_{\mu,\nu}=\frac{1}{2\gr-1},$ for all $(\mu,\nu)\in S^{\circ},$ \eqref{5.4} shows that
    $\frac{\gr-\ga_b}{2\gr-1}=1-\ga_b,$ hence $\ga_b=\half.$
    \medskip

\indent (c)
This is clear.
\medskip



(iii)
(a)  This follows from (i) and \eqref{71}.
\medskip

\indent (b)  By (i) and (a), $\gr=\ga_b=1.$
\medskip





\indent (c)
(1)  Suppose that $A$ is not commutative and $\dim A_{\gr,\gr}(b)=2.$ By (a), $y_{\gr}, y'_{\gr}\in A_{\gr,\gr}(b).$ Since $y_{\gr}\in A_{\gr,\gr}(b),$  (i)(3) says that $\gr=\ga_b.$
But also $y'_{\gr}\in A_{\gr,\gr}(b),$ so applying \eqref{NC} to $y'_{\gr},$ we get $\frac{\gr-\ga_b}{2\gr-1}=1-\ga_b.$  But $\gr=\ga_b,$ so $\gr=\ga_b=1.$
\medskip

\indent \indent (2) If $\dim A_1(b)=2,$ then $A_1(b)=\ff b\oplus \ff \left( a+\sum_{(\mu,\nu)\in S^{\circ}}\varepsilon_{\mu,\nu}b_{\mu,\nu}\right)$  by~(a). The rest of part (2) follows from (i).




\indent (d)   holds by (a),  (i)(b) and (ii)(b) and (c)(1).


\indent (e)  If $\theta_0 \ne 0,$ then solving for $\ga_y$ in \eqref{672}, and plugging  into \eqref{671} gives a cubic equation for~$\rho.$ Since by (i), $y_{\gr}\notin A_{\gr,\gr}(b),$ (a) shows that $\dim A_{\gr,\gr}(b)=1.$

\indent (f) If $\theta_0 = 0,$ \eqref{672} yields $\rho =  \half-\theta_0'.$
Again $\dim A_{\gr,\gr}(b)=1,$ as in the proof of (e).
\medskip

\end{proof}

\begin{remark}\label{t0t1}
The following observation will be used a number of times.  In the notation of Theorem \ref{eigcom}, if there exists a unique $(\mu,\mu)\in S^{\circ}(a)$ such that $b_{\mu,\mu}^2\ne 0,$ then $b_{\mu,\mu}^2=\theta'_1a+\theta'_0b_0,$ $\theta_1=\mu\theta'_1$ and $\theta_0=\mu\theta'_0.$
\end{remark}

\begin{cor}\label{rho0}
       Notation as in Theorem~\ref{eigcom}, write $\rho_0 = 0.$    $$y_{\rho_0}  = a - \sum_{(\mu,\nu)\in S} \varepsilon_{\mu,\nu}  b_{\mu,\nu}, \qquad y_{\rho_0}' = \ga_b a +b_0 - \sum_{(\mu,\nu)\in S}   b_{\mu,\nu}.$$
       There are four mutually exclusive types for the axial algebra $A$:

       \begin{enumerate}\eroman
        \item  $ A_0(b) = \ff y_{\rho_0} +\ff y_{\rho_0}' $, the only case  where  $\dim A_0(b) = 2.$ This occurs if and only if  $A$ is commutative and 
       both  $ \sum  \mu b_{\mu,\nu}^2 =  \half \ga_b a$ and  $ \sum   b_{\mu,\nu}^2 =  \half(\ga_b a +b_0).$
 \item  $ A_0(b) = \ff y_{\rho_0}' $. This occurs if and only if
 \begin{itemize}
 \item[(1)] $ \sum  \mu b_{\mu,\nu}^2 \ne  \half \ga_b a$, and $\ga_b\ne 0$ when $A$ is not commutative, and
 \item[(2)]
 $ \sum  b_{\mu,\nu}^2 =  \half(\ga_b a +b_0),$ and $\ga_b=\half$ when $A$ is not commutative.
 \end{itemize}
  Then  $\dim A_0(b) = 1.$
  \item  $ A_0(b) = \ff y_{\rho_0}$. This occurs if and only if
  \begin{itemize}
\item[(1)]
   $ \sum  \mu b_{\mu,\nu}^2 = \half \ga_b a  $, and $\ga_b=0$ when $A$ is not commutative, and
   \item[(2)]$ \sum   b_{\mu,\nu}^2 \ne  \half(\ga_b a +b_0)$ or $\ga_b\ne \half$, when $A$ is not commutative.
   \end{itemize}
   \item    $$ A_0(b) = \ff (\ga a +b_0+   \sum_{(\mu,\nu)\in S^\circ}  \gc_{\mu,\nu}  b_{\mu,\nu})  = \ff((\ga-\ga_b )y_{\rho_0} +y_{\rho_0}') .$$ This occurs when $\gc_{\mu,\nu}=\varepsilon_{\mu,\nu} (\ga-\ga_b)      -1,$ and if $A$ is not commutative then $\ga_b\ne 0,$ $\gc_{\mu,\nu}=1-\frac{1}{\ga_b},$ for all $(\mu,\nu)\in S^{\dagger}\,$ and $\ga=\ga_b+\frac{1-2\ga_b}{\varepsilon_{\mu,\nu}\ga_b}.$  In particular $\varepsilon_{\mu,\nu}$ is determined, for all $(\mu,\nu)\in S^{\dagger}.$  Also
   \begin{itemize}
       \item[(1)] for $\theta_0\ne 0,$  $\ga=\frac{ 1 +  {2\ga_b\theta_0 }   - {2 }\theta'_0}{2\theta_0},$ and
       \[0   = \ga (\ga_b -2 \theta_1   )+(\ga_b-   \ga_b^2 )    +{2\ga_b} \theta_1 -2 \theta'_1.
       \]

       \item[(2)] for $\theta_0=0,$ and $\theta_1\ne\half\ga_b,$ then $\ga=\frac{(\ga_b-   \ga_b^2 ) +2\ga_b \theta_1 -2\theta'_1}{\ga_b -2 \theta_1},$ and
       \[
       0   =  -2\theta_0 \ga  + 1 +  {2\ga_b\theta_0 }   - {2 }\theta'_0.
       \]
      \item[(3)] If  both $\theta_0=0$ and $\theta_1=\half\ga_b,$ then $\theta'_0=\half,$ and $\theta'_1=\half\ga_b,$ $A$ is not commutative, $\ga_b\notin\{0,\half\},$ $\gc_{\mu,\nu}=1-\frac{1}{\ga_b},$ for all $(\mu,\nu)\in S^{\dagger},$ and
      $\ga=\ga_b-\frac{1-2\ga_b}{\ga_b\varepsilon_{\mu,\nu}},$ for all $(\mu,\nu)\in S^{\dagger}.$  In particular, $\varepsilon:=\varepsilon_{\mu,\nu}$ is independent of $(\mu,\nu),$ for $(\mu,\nu)\in S^{\dagger}.$
      \item[(4)] If $\ga_b=1,$ and $A$ is not commutative, then $\ga=1+\frac{1}{\varepsilon},$ with $(\mu^{\dagger},\nu^{\dagger})\in S^{\dagger},$
      \[
      y=\ga a+b_0+\sum_{(\mu,\mu)\in S^{\circ}}\left(1-\frac{2\mu}{\varepsilon}\right)b_{\mu,\mu}.
      \]
   \end{itemize}
    \end{enumerate}
\end{cor}
\begin{proof}
    We substitute ${\rho_0} =0$ in Theorem~\ref{eigcom}.
\medskip

(i)  $\theta_0 =0$ by Theorem~\ref{eigcom}(i)(1), and $\theta_1=\half\ga_b,$ by  Theorem~\ref{eigcom}(i)(2), so   $ \sum  \mu b_{\mu,\nu}^2 =  \half \ga_b a$.
 $\theta'_1=\half{\ga_b}$ and $\theta'_0=\half,$ by Theorem~\ref{eigcom}(ii)(b), so   $ \sum   b_{\mu,\nu}^2 =  \half (\ga_b a+b_0)$. Finally  $\dim A_0(b)=2,$ comes from Theorem~\ref{eigcom}(iii)(a).

(ii)\&(iii)  By Theorem~\ref{eigcom}(i) and Theorem~\ref{eigcom}(ii)(b),  (ii)  occurs iff $y'_{\gr_0}\in A_0(b)$ and $y_{\gr_0}\notin A_0(b),$ and (iii) occurs iff $y'_{\gr_0}\notin A_0(b),$ and $y_{\gr_0}\in A_0(b).$

(iv) Take $y\in A_0(b)$ and write $y=\ga_ya+\gc_0b_0+\sum_{(\mu,\nu)\in S^{\circ}}\gc_{\mu,\nu}b_{\mu,\nu}.$ First note that $\gc_0\ne 0$, for otherwise, by Theorem \ref{eigcom}(i), $y_{\gr_0}\in A_0(b),$ so we would have~(i) or (iii).  We normalize $\gc_0 = 1, $ so $\ga_y = \ga.$ Note that if $\ga_y=\ga_b,$ then $y=y_{\gr_0}'$ and we are in case (ii).

$(\ga_y-\ga_b )y_{\rho_0} +y_{\rho_0}'   = (\ga_y-\ga_b )a - \sum_{(\mu,\nu)\in S} \varepsilon_{\mu,\nu}  (\ga_y-\ga_b ) b_{\mu,\nu}+ \ga_b a +b_0 - \sum_{(\mu,\nu)\in S}   b_{\mu,\nu}  $.

Suppose $A$ is not commutative. Then, by \eqref{NC},  $-\ga_b\gc_{\mu,\nu}=(1-\ga_b),$ for all $(\mu,\nu)\in S^{\dagger}.$ Thus $\ga_b\ne 0,$ and $\gc_{\mu,\nu}=1-\frac{1}{\ga_b},$ for all $(\mu,\nu)\in S^{\dagger}.$  Then by \eqref{11}, the equality $\frac{1}{\ga_b}-1=\varepsilon_{\mu,\nu}(\ga_y-\ga_b)+1$ holds and yields the value of $\ga_y,$ and  $\varepsilon_{\mu,\nu}$ is fixed.

Substituting $\rho = 0$ in \eqref{671} and \eqref{672} yields \begin{equation}\label{6711}
    0   = \ga_y (\ga_b -2 \theta_1   )+(\ga_b-   \ga_b^2 )    +{2\ga_b} \theta_1 -2 \theta'_1,
    \end{equation} and
    \begin{equation}\label{6721}
    0   =  -2\theta_0 \ga_y  + 1 +  {2\ga_b\theta_0 }   - {2 }\theta'_0,
    \end{equation}
 implying (1) or (2), unless both $\theta_0=0,$ and $\theta_1=\half\ga_b.$ But then we get $\theta'_0=\half$ and $\theta'_1=\half\ga_b.$

 Thus, if $A$ is commutative, we are in case (i). Suppose that $A$ is not commutative. If $\ga_b=\half,$ we are in case (ii), while if $\ga_b=0,$ we are in case~(iii), and    (3) holds for $\ga_b\notin\{0,\half\}$. (4) is obtained by plugging in $\gc_{\mu,\nu}=\varepsilon_{\mu,\nu} (\ga-\ga_b)      -1.$
\end{proof}

\begin{lemma}\label{calcu}
Notation  as in Theorem \ref{eigcom}, assume that $A$ is commutative, and let $\gr,\xi\in \ff\setminus\{\half\}.$  Then
\begin{enumerate}\eroman
\item
$y_{\gr}y_{\xi}=a+\left(\frac{1}{2\gr-1}+\frac{1}{2\xi-1}\right)\sum_{(\eta,\eta)\in S^{\circ}}2\eta^2 b_{\eta,\eta}+\frac{1}{(2\gr-1)(2\xi-1)}\sum_{(\eta,\eta)\in S^{\circ}}4\eta^2b_{\eta,\eta}^2.$

\item
$y_{\gr}y'_{\xi}=\ga_b a+\left(\frac{1}{2\gr-1}+\frac{1}{2\xi-1}\right)\sum_{(\eta,\eta)\in S^{\circ}}\eta b_{\eta,\eta}+\frac{1}{(2\gr-1)(2\xi-1)}\sum_{(\eta,\eta)\in S^{\circ}}2\eta b_{\eta,\eta}^2.$

\item
$y'_{\gr}y'_{\xi}=\ga_ba+b_0-\sum b_{\eta,\eta}^2+\left(\frac{1}{2\gr-1}+\frac{1}{2\xi-1}\right)\sum_{(\eta,\eta)\in S^{\circ}}\half b_{\eta,\eta}\\ \qquad  \qquad  \qquad \qquad\qquad +\frac{1}{(2\gr-1)(2\xi-1)}\sum_{(\eta,\eta)\in S^{\circ}}b_{\eta,\eta}^2.$

\item
Let $y=\ga y_{\gr}+y'_{\gr}$ and $z=\gb y_{\xi}+y'_{\xi}.$ Then
\begin{enumerate}
\item
$y=(\ga_b+\ga)a+b_0+\frac{1}{2\gr-1}\sum_{(\eta,\eta)\in S^{\circ}}(2\ga\eta+1)b_{\eta,\eta}.$

\item
\begin{equation}\label{yz}
\begin{aligned}
yz=(\ga&\gb+(\ga+\gb+1)\ga_b)a+b_0-\sum_{(\eta,\eta)\in S^{\circ}}b_{\eta,\eta}^2\\
&+ \frac{1}{(2\gr-1)(2\xi-1)}\sum_{(\eta,\eta)\in S^{\circ}}(2\ga\eta+1)(2\gb\eta+1)b_{\eta,\eta}^2\\
&+\left(\frac{1}{2\gr-1}+\frac{1}{2\xi-1}\right)\sum_{(\eta,\eta)\in S^{\circ}}\half(2\ga\eta+1)(2\gb\eta+1)b_{\eta,\eta}.
\end{aligned}
\end{equation}
\end{enumerate}

\item
Let $z=\gb y_{\xi}+y'_{\xi}.$  Then
\begin{equation}\label{ygr z}
\begin{aligned}
y_{\gr}z=&\textstyle{(\ga_b+\gb)a+\left(\frac{1}{2\gr-1}+\frac{1}{2\xi-1}\right)\sum_{(\eta,\eta)\in S^{\circ}}(2\gb\eta^2+\eta)b_{\eta,\eta}}\\
  &\qquad+\textstyle{\frac{1}{(2\gr-1)(2\xi-1)}\sum_{(\eta,\eta)\in S^{\circ}}(4\gb\eta^2+2\eta)b_{\eta,\eta}^2.}
\end{aligned}
\end{equation}

\item
Let $y=\ga y_{\gr}+y'_{\gr}.$ Then
\begin{equation*}
\begin{aligned}
by&=(\ga+1)\ga_ba+b_0-\sum_{(\eta,\eta)\in S^{\circ}}b_{\eta,\eta}^2
+\frac{1}{2\gr-1}\sum_{(\eta,\eta)\in S^{\circ}}(2\ga\eta+1)b_{\eta,\eta}^2\\
&\qquad +\frac{\gr}{2\gr-1}\sum_{(\eta,\eta)\in S^{\circ}}(2\ga\eta+1)b_{\eta,\eta}.
\end{aligned}
\end{equation*}

\item Let $\xi=\xi_0=0,$ and take $w=\gt y_{\xi_0}+y'_{\xi_0},$ in place of $y_{\xi}$ in \eqref{yz}, Then
\begin{enumerate}
\item We have
\begin{equation}\label{wy}
\begin{aligned}
wy=(\gt\ga&+(\ga+\gt+1)\ga_b)a+b_0-\sum_{(\eta,\eta)\in S^{\circ}}b_{\eta,\eta}^2\\
& -\frac{1}{2\gr-1}\sum_{(\eta,\eta)\in S^{\circ}}(2\ga\eta+1)(2\gt\eta+1)b_{\eta,\eta}^2\\
&-\frac{\gr}{2\gr-1}\sum_{(\eta,\eta)\in S^{\circ}}(2\ga\eta+1)(2\gt\eta+1)b_{\eta,\eta}.
\end{aligned}
\end{equation}

\item
We have
\begin{equation}\label{zero0}
bw=(\gt+1)\ga_ba+b_0-\sum_{(\eta,\eta)\in S^{\circ}}(2\gt\eta+2)b_{\eta,\eta}^2.
\end{equation}


\end{enumerate}
\item
Suppose that $\gr+\xi\ne 1,$ let $y=\ga y_{\gr}+y'_{\xi}$ and $z=\gb y_{\xi}+y'_{\xi}.$  Then,
\begin{enumerate}
\item
if $yz=0,$ and $\gr+\xi\ne 1,$ then $\{\eta\mid (\eta,\eta)\in S^{\circ}(a)\}\subseteq \{-\frac{1}{2\ga}, -\frac{1}{2\gb}\},$ and $\sum_{(\eta,\eta)\in S^{\circ}}b_{\eta,\eta}^2=(\ga\gb+(\ga+\gb+1)\ga_b)a+b_0.$

\item
If $y_{\gr}z=0,$ and $\gr+\xi\ne 0,$ then $\{\eta\mid (\eta,\eta)\in S^{\circ}\}\subseteq \{-\frac{1}{2\gb}\}.$

\item
If $ay_{\gr}=\half y_{\gr}+\gl y_{\gr}^2,$ then $(1-4\gl)\eta=\half,$ for all $(\eta,\eta)\in S^{\circ}(a),$ in particular, $\mid S^{\circ}(a)\mid\le 1.$
\end{enumerate}
\end{enumerate}
\end{lemma}
\begin{proof}
(i) We have
\begin{gather*}
\textstyle{y_{\gr}y_{\xi}=\left(a+\frac{1}{2\gr-1}\sum_{(\eta,\eta)\in S^{\circ}}2\eta b_{\eta,\eta}\right)\left(a+\frac{1}{2\xi-1}\sum_{(\eta,\eta)\in S^{\circ}}2\eta b_{\eta,\eta}\right)}\\
\textstyle{=a+\left(\frac{1}{2\gr-1}+\frac{1}{2\xi-1}\right)\sum_{(\eta,\eta)\in S^{\circ}}2\eta^2 b_{\eta,\eta}+\frac{1}{(2\gr-1)(2\xi-1)}\sum_{(\eta,\eta)\in S^{\circ}}4\eta^2b_{\eta,\eta}^2.}
\end{gather*}
\medskip

(ii) We have
\begin{gather*}
\textstyle{y_{\gr}y'_{\xi}=\left(a+\frac{1}{2\gr-1}\sum_{(\eta,\eta)\in S^{\circ}}2\eta b_{\eta,\eta}\right)\left(\ga_ba+b_0+\frac{1}{2\xi-1}\sum_{(\eta,\eta)\in S^{\circ}}b_{\eta,\eta}\right)}\\
\textstyle{=\ga_b a+\left(\frac{1}{2\gr-1}+\frac{1}{2\xi-1}\right)\sum_{(\eta,\eta)\in S^{\circ}}\eta b_{\eta,\eta}+\frac{1}{(2\gr-1)(2\xi-1)}\sum_{(\eta,\eta)\in S^{\circ}}2\eta b_{\eta,\eta}^2.}
\end{gather*}
\medskip

(iii) We have
\begin{gather*}
\textstyle{y'_{\gr}y'_{\xi}=\left(\ga_ba+b_0+\frac{1}{2\gr-1}\sum_{(\eta,\eta)\in S^{\circ}}b_{\eta,\eta}\right)\left(\ga_ba+b_0+\frac{1}{2\xi-1}\sum_{(\eta,\eta)\in S^{\circ}} b_{\eta,\eta}\right)}\\
\textstyle{=\ga_b^2a+b_0^2+\left(\frac{1}{2\gr-1}+\frac{1}{2\xi-1}\right)\sum_{(\eta,\eta)\in S^{\circ}}\half b_{\eta,\eta}+\frac{1}{(2\gr-1)(2\xi-1)}\sum_{(\eta,\eta)\in S^{\circ}}b_{\eta,\eta}^2}\\
\textstyle{=\ga_ba+b_0-\sum b_{\eta,\eta}^2+\left(\frac{1}{2\gr-1}+\frac{1}{2\xi-1}\right)\sum_{(\eta,\eta)\in S^{\circ}}\half b_{\eta,\eta}}\\ +\textstyle{\frac{1}{(2\gr-1)(2\xi-1)}\sum_{(\eta,\eta)\in S^{\circ}}b_{\eta,\eta}^2.}
\end{gather*}
\medskip

(iv) (a) This is by definition.

(b) Since $(\ga y_{\gr}+y'_{\gr})(\gb y_{\gr}+y'_{\gr})=\ga\gb y_{\gr}y_{\xi}+\ga y_{\gr}y'_{\xi}+\gb y_{\xi}y'_{\gr}+y'_{\gr}y'_{\xi},$ we have
\begin{align*}
yz=&(\ga\gb+(\ga+\gb+1)\ga_b)a+b_0-\sum_{(\eta,\eta)\in S^{\circ}}b_{\eta,\eta}^2\\
&+ \frac{1}{(2\gr-1)(2\xi-1)}\sum_{(\eta,\eta)\in S^{\circ}}(4\ga\gb\eta^2+2(\ga+\gb)\eta+1)b_{\eta,\eta}^2\\
&+\left(\frac{1}{2\gr-1}+\frac{1}{2\xi-1}\right)\sum_{(\eta,\eta)\in S^{\circ}}(2\ga\gb \eta^2+(\ga+\gb)\eta+\half)b_{\eta,\eta},
\end{align*}
so \eqref{yz} holds.

\medskip

(v) We have
\begin{gather*}
\textstyle{y_{\gr}(\gb y_{\xi}+y'_{\xi})=\gb y_{\gr}y_{\xi}+y_{\gr}y'_{\xi}}\\
\textstyle{=\gb a+\left(\frac{\gb}{2\gr-1}+\frac{\gb}{2\xi-1}\right)\sum_{(\eta,\eta)\in S^{\circ}}2\eta^2 b_{\eta,\eta}+\frac{\gb}{(2\gr-1)(2\xi-1)}\sum_{(\eta,\eta)\in S^{\circ}}4\eta^2b_{\eta,\eta}^2}\\
\textstyle{+\ga_b a+\left(\frac{1}{2\gr-1}+\frac{1}{2\xi-1}\right)\sum_{(\eta,\eta)\in S^{\circ}}\eta b_{\eta,\eta}+\frac{1}{(2\gr-1)(2\xi-1)}\sum_{(\eta,\eta)\in S^{\circ}}2\eta b_{\eta,\eta}^2}\\
\textstyle{=(\ga_b+\gb)a+\left(\frac{1}{2\gr-1}+\frac{1}{2\xi-1}\right)\sum_{(\eta,\eta)\in S^{\circ}}(2\gb\eta^2+\eta)b_{\eta,\eta}}\\
\textstyle{+\frac{1}{(2\gr-1)(2\xi-1)}\sum_{(\eta,\eta)\in S^{\circ}}(4\gb\eta^2+2\eta)b_{\eta,\eta}^2.}
\end{gather*}
\medskip

(vi)  We take $\xi =1$ and $\gb=0$ in \eqref{yz}, so that $z=b.$ By \eqref{yz},
\begin{equation*}
\begin{aligned}
yb&=(\ga+1)\ga_ba+b_0-\sum_{(\eta,\eta)\in S^{\circ}}b_{\eta,\eta}^2
+\frac{1}{2\gr-1}\sum_{(\eta,\eta)\in S^{\circ}}(2\ga\eta+1)b_{\eta,\eta}^2\\
&\qquad +\left(\frac{1}{2\gr-1}+1\right)\sum_{(\eta,\eta)\in S^{\circ}}\half(2\ga\eta+1)b_{\eta,\eta}.
\end{aligned}
\end{equation*}

(vii) We take $\xi=0$ in \eqref{yz}, and replace $\gb$ by $\gt.$

(a) By \eqref{yz}
\begin{equation*}
\begin{aligned}
yw=(\gt\ga &+(\ga+\gt+1)\ga_b)a+b_0-\sum_{(\eta,\eta)\in S^{\circ}}b_{\eta,\eta}^2\\
&- \frac{1}{2\gr-1}\sum_{(\eta,\eta)\in S^{\circ}}(2\ga\eta+1)(2\gt\eta+1)b_{\eta,\eta}^2\\
&+\left(\frac{1}{2\gr-1}-1\right)\sum_{(\eta,\eta)\in S^{\circ}}\half(2\ga\eta+1)(2\gt\eta+1)b_{\eta,\eta}.
\end{aligned}
\end{equation*}

(b) Taking $\gr=1,$ and $\ga=0,$ in \eqref{yz}, then $y=b,$ so
\[
bw=(\gt+1)\ga_ba+b_0-\sum_{(\eta,\eta)\in S^{\circ}}b_{\eta,\eta}^2
- \sum_{(\eta,\eta)\in S^{\circ}}(2\gt\eta+1)b_{\eta,\eta}^2.
\]

(viii)  (a) follows from \eqref{yz}. and (b) from \eqref{ygr z}.

(c)  By (i) and by the fusion rules for $a,$ we get $y_{\gr}^2=w+\frac{1}{2\gr-1}\sum_{(\eta,\eta)\in S^{\circ}}\varepsilon_{\eta,\eta}^2 b_{\eta,\eta},$ with $w\in \ff a+\ff b_0.$  Hence, by the definition of $y_{\gr},$
\[
\sum_{(\eta,\eta)\in S^{\circ}}\eta\varepsilon_{\eta,\eta}b_{\eta,\eta}=\half\sum_{(\eta,\eta)\in S^{\circ}}\varepsilon_{\eta,\eta}b_{\eta,\eta}+\frac{\gl}{2\gr-1}
\sum_{(\eta,\eta)\in S^{\circ}}\varepsilon_{\eta,\eta}^2b_{\eta,\eta}.
\]
Thus, after cancelling $\varepsilon_{\eta,\eta}$
we get $\eta=\half+\frac{2\gl}{2\gr-1}\eta.$
\end{proof}

\section{The classification when the axis $b$  satisfies the fusion rules}\label{sb}

Throughout this section we continue to assume that $a$ is a weakly primitive $S$-axis satisfying the fusion rules, with $\dim A_0(a)=1.$  Here we also assume that $b$ is an axis satisfying the fusion rules.  Recall that
\[
b=\ga_b+b_0+\sum_{(\mu,\nu)\in S^{\circ}(a)}b_{\mu,\nu},
\]
and that $A=\lan\lan a,b\ran\ran = \ff a+\ff b_0+\sum_{(\mu,\nu)\in S^{\circ}}\ff b_{\mu,\nu}.$  

\text{\bf Throughout, we use the notation of Theorem \ref{eigcom}.}

\subsection{The case when $A$ is commutative}$ $
\medskip

In this subsection we assume that $A$ is commutative.  We start with a number of examples with the axis $b$ non-primitive but satisfying the fusion rules, the first two examples having dimension 3.

\begin{example}\label{4dim2}
Let $A=\ff a+\ff b_0+\ff b_{\half,\half},$ with $b_{\half,\half}^2=0,\ b_0^2=b_0,  b_0b_{\half,\half}=0.$
$b=a+b_0+b_{\half,\half},$ so $\ga_b=1.$
$A_1(b)={\rm Span}\{b, a+b_{\half,\half}\},$ $A_{\half,\half}(b)=\ff b_{\half,\half}.$
\end{example}

\begin{example}\label{3dim01}
Let $A=\ff a+\ff b_0+\ff b_{\mu,\mu}.$
\[
\textstyle{b_{\mu,\mu}^2=\frac{1}{4\mu}a,\quad b_0^2=b_0+\frac{\mu-1}{4\mu}a,\quad b_0 b_{\mu,\mu} =\frac{1-\mu}{2}b_{\mu,\mu},}
\]
$b=\half a+b_0+b_{\mu,\mu},$ and
\[
A_0(b)={\rm Span}\{ a-2\mu b_{\mu,\mu}\},\quad A_1(b)={\rm Span}\{b, a+2\mu b_{\mu,\mu}\}.
\]
\end{example}

\begin{example}\label{4dim}
Let $A$ be the following algebra.  $A$ is spanned over $\ff,$ by $a, b_0, b_{\mu,\mu}$ and $b_{\nu,\nu},$ where $\mu,\nu\in\ff\setminus\{0,1\}, \mu\ne\nu,$ so it is $4$-dimensional.  $A$ is commutative, and multiplication in $A$ is defined as follows.
\[
a^2=a,\quad ab_0=0,\quad ab_{\mu,\mu}=\mu b_{\mu,\mu},\quad ab_{\nu,\nu}=\nu b_{\nu,\nu}.
\]
To define $b_{\mu,\mu}^2$ and $b_{\nu,\nu}^2$ we use the following equation:
\[
\textstyle{b_{\mu,\mu}^2+b_{\nu,\nu}^2=\frac{1}{4}a+\half b_0,\quad\mu b_{\mu,\mu}^2+\nu b_{\nu,\nu}^2=\frac{1}{4} a.}
\]
We also let
\[
\textstyle{b_0^2=\half b_0,\quad b_0b_{\mu,\mu}=\frac{1-\mu}{2}b_{\mu,\mu},\quad b_0b_{\nu,\nu}=\frac{1-\nu}{2}b_{\nu,\nu},\quad b_{\mu,\mu}b_{\nu,\nu}=0.}
\]
Note that
\[
\textstyle{(\half a+b_0)b_{\mu,\mu}=\half\mu b_{\mu,\mu}+\frac{1-\mu}{2}b_{\mu,\mu}=\half b_{\mu,\mu},}
\]
and similarly $(\half a+b_0)b_{\nu,\nu}=\half b_{\nu,\nu}.$ Also
\[
\textstyle{(a+b_0)b_{\mu,\mu}=(\half a+\half a+b_0)b_{\mu,\mu}=\half b_{\mu,\mu}+\half\mu b_{\mu,\mu}=\frac{1+\mu}{2}b_{\mu,\mu},}
\]
and similarly $(a+b_0)b_{\nu,\nu}=\frac{1+\nu}{2} b_{\nu,\nu}.$
Let $b=\half a+b_0+b_{\mu,\mu}+b_{\nu,\nu},$ so $\ga_b = \half.$   Then
\begin{align*}
b^2&=\textstyle{\frac{1}{4}a+b_0^2+b_{\mu,\mu}^2+b_{\nu,\nu}^2+b_{\mu,\mu}+b_{\nu,\nu}}\\
&\textstyle{=\frac{1}{4}a+\half b_0+\frac{1}{4}a+\half b_0+b_{\mu,\mu}+b_{\nu,\nu}=b.}
\end{align*}
So $b$ is an idempotent.  Next
\begin{align*}
&\textstyle{b(a+2\mu b_{\mu,\nu}+2\nu b_{\nu,\nu})=(\half a+b_0+b_{\mu,\mu}+b_{\nu,\nu})(a+2\mu b_{\mu,\mu}+2\nu b_{\nu,\nu})}\\
&\qquad \textstyle{=\half a+\mu b_{\mu,\mu}+\nu b_{\nu,\nu}+\mu b_{\mu,\mu}+\nu b_{\nu,\nu}+2\mu b_{\mu,\mu}^2+2\nu b_{\nu,\nu}^2}\\
&\qquad =a+2\mu b_{\mu,\mu}+2\nu b_{\nu,\nu},
\end{align*}

Thus $A_1(b)={\rm Span}\{ b, a+2\mu b_{\mu,\mu}+2\nu b_{\nu,\nu}\}.$ Next
\begin{align*}
&\textstyle{b(a-2\mu b_{\mu,\mu}-2\nu b_{\nu,\nu})=(\half a+b_0+b_{\mu,\mu}+b_{\nu,\nu})(a-2\mu b_{\mu,\mu}-2\nu b_{\nu,\nu})}\\
&\qquad \textstyle{=\half a+\mu b_{\mu,\mu}+\nu b_{\nu,\nu}-\mu b_{\mu,\mu}-\nu b_{\nu,\nu}-2\mu b_{\mu,\mu}^2-2\nu b_{\nu,\nu}^2}=0.
\end{align*}
Finally,
\begin{align*}
&\textstyle{b(\half a+b_0-b_{\mu,\mu}-b_{\nu,\nu})=(\half a+b_0+b_{\mu,\mu}+b_{\nu,\nu})(\half a+b_0-b_{\mu,\mu}-b_{\nu,\nu})}\\
&\qquad =\textstyle{\frac{1}{4}a+\half b_0-b_{\mu,\mu}^2-b_{\nu,\nu}^2=0. }
\end{align*}
Thus $A_0(b)={\rm Span}\left\{y_{\rho_0} = a-2\mu b_{\mu,\mu}-2\nu b_{\nu,\nu},\ y'_{\rho_0} = (\ga_b a +b_0)-b_{\mu,\mu}-b_{\nu,\nu}\right\},$
so $A$ is as in Corollary \ref{rho0}(i), and $b$ is an axis satisfying the fusion rules.
\end{example}

\begin{thm}\label{yrho}
Suppose $A$ is commutative and let $\gr\ne\half.$ Let
\[
y_{\gr}=a+\frac{1}{2\gr-1}\sum_{(\eta,\eta)\in S^{\circ}}\varepsilon_{\eta,\eta}b_{\eta,\eta},
\]
as in Theorem \ref{eigcom}.  Assume that $y_{\gr}\in A_{\gr,\gr}(b);$  then
\begin{enumerate}\eroman
\item
If $\gr\in\{0,1\},$ then $|S^{\circ}(a)|\le 2.$  More precisely, letting $\xi=1-\gr,$ if we write using the fusion rules (and using \eqref{71}):
\begin{equation}\label{2b0}
y_{\gr}^2=\ga y_{\gr}+\gb y_{\xi}+\gc y'_{\gr}+\gd y'_{\xi},
\end{equation}
then $\varepsilon_{\eta,\eta}$ is a solution of the equation $x^2+(\ga-\gb) x+\gc-\gd,$ for all $(\eta,\eta)\in S^{\circ}(a).$ In particular $\dim A\le 4.$



\item 
 If $\dim A_0(b)=\dim A_1(b)=2,$ then $A$ is as in Example \ref{4dim}.


 \item
 Assume that $\gr\in\{0,1\}$ and let  $0\ne z:=\sum_{(\mu,\mu)\in S^{\circ}}\gc'_{\mu,\mu}b_{\mu,\mu}\in A_{\half,\half}(b).$
 \begin{enumerate}
\item
$A_{\half,\half}(b)\subseteq V(a),$ and $\sum_{(\mu,\mu)\in S^{\circ}}\gc'_{\mu,\mu}\varepsilon_{\mu,\mu}^n b^2_{\mu,\mu}=0,$ for all $n\ge 0;$  in particular $b_{\mu,\mu}^2=0,$ for each $\mu$ such that $\gc'_{\mu,\mu}\ne 0.$

\item  If $\dim A_{0}(b)=2,$ then $A$ is as in Example \ref{4dim}.

\item Suppose $\dim A_1(b)=2.$ 
\begin{itemize}
\item[(1)]
Assume that $A_{\xi,\xi}(b)\ne 0,$ for some $\xi\in \ff\setminus\{\half, 1\};$ then $\xi=0,$ and either $\dim A=3,$ and $A$ is as in Example~\ref{3dim01}, or $A$ is as in Example \ref{4dim}.

\item[(2)] The only remaining possibility is that $S^{\circ}(b)=\{(\half,\half)\}.$ Then $A$ is as in Example~\ref{4dim2}.
\end{itemize}



 \end{enumerate}

\end{enumerate}
\end{thm}
\begin{proof}
We use the notation of Theorem \ref{eigcom}.

(i)  By Lemma \ref{calcu}(i), for $\gr\in\{0,1\},$
\begin{equation}\label{111}
y_{\gr}^2=a+(2\gr-1)\sum_{(\eta,\eta)\in S^{\circ}(a)} \varepsilon_{\eta,\eta}^2b_{\eta,\eta}+\sum_{(\eta,\eta)\in S^{\circ}(a)} \varepsilon_{\eta,\eta}^2b_{\eta,\eta}^2.
\end{equation}
Since $b$ satisfies the fusion rules, $y_{\gr}^2$ is a linear combination
\[
y_{\gr}^2=\ga y_{\gr}+\gb y_{\xi}+\gc y'_{\gr}+\gd y'_{\xi},
\]
by Theorem \ref{eigcom}. Thus, comparing the coefficient of $b_{\eta,\eta}$ in each side,
\[
(2\gr-1)\varepsilon_{\eta,\eta}^2=\ga(2\gr-1)\varepsilon_{\eta,\eta}+\gb(2\xi-1)\varepsilon_{\eta,\eta}+(2\gr-1)\gc+(2\xi-1)\gd, \quad \forall (\eta,\eta)\in S^{\circ}.
\]
Since $2\gr-1=-(2\xi-1),$ we get that $\varepsilon_{\eta,\eta}$ is a solution to the equation
\[
x^2+(\gb-\ga) x+\gd- \gc=0,
\]
for all $(\eta,\eta)\in S^{\circ}.$
\medskip

(ii)  By Theorem \ref{eigcom},  $y_{\gr}, y_{\gr}'\in A_{\gr}(b),$ for $\gr\in\{0,1\}.$ Also applying Theorem \ref{eigcom}(i), we get $-\ga_b=-2\theta_1,$ and $1-\ga_b=2\theta_1.$ Hence $\ga_b=\half,$ and $\theta_1=\frac{1}{4}.$ Also, $\theta_0=0,$ that is $\mu b_{\mu,\mu}^2+\nu b_{\nu,\nu}^2=\half a.$  By Theorem \ref{eigcom}(iii)(d), applied to $\gr=0,$ $\theta_0'=\half,$ and $\theta_1'=\frac{1}{4},$ that is $b_{\mu,\mu}^2+b_{\nu,\nu}^2=\frac{1}{4}a+\half b_0.$  Also, by Theorem \ref{Ser02}(iii), $b_0^2=b_0+\frac{1}{4}a-\frac{1}{4}a-\half b_0=\half b_0.$  Hence $A$ is as in Example~\ref{4dim}.
\medskip

(iii) (a)  
Set $\gr =0.$   Note first that $\theta_0=0,$ since $y_{\gr}\in A_{\gr,\gr}(b).$  Suppose $A_{\half,\half}(b)\nsubseteq V(a).$  Then, by Theorem \ref{Ser04}(v), $S^{\circ}(a)=\{(\mu,\mu)\},$ for at most one~$\mu,$ and  $\theta'_0\ne 0,$ by Theorem \ref{Ser04}(v)(b)(2\&3).  But $\theta_0=\mu\theta'_0,$ by Remark \ref{t0t1}, a contradiction.

The proof of the second part is by induction.  For $n=0,$ this follows from \ref{Ser041}(iii). Assume $\sum_{(\mu,\mu)\in S^{\circ}}\gc'_{\mu,\mu}\varepsilon_{\mu,\mu}^kb_{\mu,\mu}^2=0.$  Then
$\sum \gc'_{\mu,\mu}\varepsilon_{\mu,\mu}^kb_{\mu,\mu}\in A_{\half,\half}(b),$ by~Theorem~\ref{Ser041}(iii). But then, also
$y_{\gr}\sum \gc'_{\mu,\mu}\varepsilon_{\mu,\mu}^kb_{\mu,\mu}\in A_{\half,\half}(b).$  But
\[
y_{\gr}\sum \gc'_{\mu,\mu}\varepsilon_{\mu,\mu}^kb_{\mu,\mu}=\sum \gc'_{\mu,\mu}\mu\varepsilon_{\mu,\nu}^k b_{\mu,\mu}+(2\gr-1)\sum\gc'_{\mu,\mu}\varepsilon_{\mu,\mu}^{k+1}b_{\mu,\mu}^2.
\]
Since  $A_{\half,\half}(b)\subseteq V(a),$
$\sum\gc'_{\mu,\mu}\varepsilon_{\mu,\mu}^{k+1}b_{\mu,\mu}^2=0,$
and the induction step is complete. By a Vandermonde argument, since $\varepsilon_{\mu,\mu}\ne 1$ for  $\mu\ne\half,$  $b_{\mu,\mu}^2=0,$ for all $\mu$ such that $\gc'_{\mu,\mu}\ne 0.$
\medskip

(b)  Since $y_{\gr}\in A_0(b),$ Theorem \ref{eigcom}(i) implies that $\theta_0=0.$ Since $y'_{\gr}\in A_0(b),$ Theorem \ref{eigcom}(d) implies $\theta'_0=\half.$

If $\dim A=3,$ then $\theta_0=\mu\theta'_0$ for $(\mu,\mu)\in S^{\circ}(a),$ by Remark \ref{t0t1}, a contradiction.  Hence $\dim A=4.$

If $A_{\half,\half}(b)\ne 0,$ then, by (a), $b_{\mu,\mu}^2=0$ for some $(\mu,\mu)\in S^{\circ}(a).$  But then for $\nu\ne\mu$ such that $(\nu,\nu)\in S^{\circ},$ $\theta_0=\nu\theta'_0,$ by Remark \ref{t0t1}, a contradiction.

Hence $A_{\half,\half}(b)=0.$  But now, since $\theta'_0=\half,$  Theorem \ref{eigcom}(iii)(f), $A_{\xi,\xi}(b)=0,$ for all $\xi\notin\{0,1\},$ so by (iv), $A$ is as in Example \ref{4dim}.
\medskip

(c)
(1) Set $\gr=1$.
\medskip

{\bf Case I.}  $y_{\xi}\in A_{\xi,\xi}(b).$
\medskip

We have
\begin{equation}\label{222}
y_{\gr}y_{\xi}=a+\left(\frac{\xi}{2\xi-1}\right)\sum_{(\eta,\eta)\in S^{\circ}}4\eta^2 b_{\eta,\eta}+\frac{1}{2\xi-1}\sum_{(\eta,\eta)\in S^{\circ}}4\eta^2b_{\eta,\eta}^2,
\end{equation}

{\bf Subcase Ia.} $\xi\ne 0.$
\medskip

We claim that $y'_{\xi}\notin A_{\xi,\xi}(b).$ Indeed, suppose $y'_{\xi}\in A_{\xi,\xi}(b);$ then, since $\dim A=4,$ $A_0(b)=0.$  Now by \eqref{111} and since $\theta_0=0$ (because $y_{\gr}\in A_0(b)$), $u':=\sum_{(\eta,\eta)\in S^{\circ}(a)}\varepsilon_{\eta,\eta}^2b_{\eta,\eta}^2\in\ff a.$  Hence, since $y_{\gr}^2\in A_1(b)$ by the fusion rules for $b,$ and since the coefficient in $b$ of $b_0$ is $1,$ $y_{\gr}^2\in\ff y_{\gr}.$  But then considering the coefficient of $b_{\eta,\eta},$ we get $\varepsilon_{\eta,\eta}^2=\gl\varepsilon_{\eta,\eta},$ for a fixed $\gl\in \ff,$ for all $(\eta,\eta)\in S^{\circ}(a).$  Hence $|S^{\circ}(a)|\le 1,$ so $\dim A\le 3,$ a contradiction.

Next, by the fusion rules for $b,$ $y_{\gr}y_{\xi}\in A_{\xi,\xi}(b)=\ff y_{\xi},$ so let $\gl\in\ff$ such that $y_{\gr}y_{\xi}=\gl y_{\xi}.$  Since $y_{\xi}=a+\frac{1}{2\xi-1}\sum_{(\eta,\eta)\in S^{\circ}(a)}2\eta b_{\eta,\eta},$ \eqref{222} shows that $\frac{2\gl\eta}{2\xi-1}=\frac{4\xi\eta^2}{2\xi-1},$ for all $(\eta,\eta)\in S^{\circ}(a).$  This shows that $|S^{\circ}(a)|=1,$ so $\dim A=3.$  Hence $A_0(b)=0.$ Set $S^{\circ}(a)=\{(\mu,\mu)\}.$

By Theorem \ref{eigcom}(i), since $y_{\zeta}\in A_{\zeta,\zeta}(b),$ for $\zeta\in \{\gr,\xi\},$ we get $1-\ga_b=2\theta_1$ and $\xi-\ga_b=\frac{2}{2\xi-1}\theta_1.$  Hence $\theta_1=\frac{1-2\xi}{4}.$  By Remark \ref{t0t1}, $\theta'_1=\frac{1-2\xi}{4\mu},$ and $\theta'_0=0,$ that is, $b_{\mu,\mu}^2=\frac{1-2\xi}{4\mu}a.$ Plugging this into \eqref{222} we get
\[
(2\xi-1)y_{\gr} y_{\xi}=(2\xi-1)a+4\xi\mu^2b_{\mu,\mu}+\frac{4\mu^2}{2\xi-1}\cdot\frac{1-2\xi}{4\mu}a=(2\xi-1-\mu)a+4\xi\mu^2b_{\mu,\mu}
\]
But $(2\xi-1)y_{\gr} y_{\xi}=\gd y_{\xi},$ for some $\gd\in\ff.$  Hence $\gd=2\xi-1-\mu,$ and hence, considering the coefficient of $b_{\mu,\mu},$ we get
\begin{equation}\label{444}
(2\xi-1-\mu)\frac{2\mu}{2\xi-1}=4\xi\mu^2\implies\frac{2\xi-1-\mu}{2\xi-1}=2\xi\mu.
\end{equation}
Now by \eqref{111},
\[
y^2_{\gr}=a+4\mu^2b_{\mu,\mu}+4\mu^2\frac{1-2\xi}{4\mu}a=a+4\mu^2b_{\mu,\mu}+\mu(1-2\xi)a.
\]
By the fusion rules for $b,$ $y_{\gr}^2\in A_1(b)+A_0(b).$  Since $A_0(b)=0,$ $y_{\gr}^2\in A_1(b),$ so $y_{\gr}^2$ is a linear combination of $b$ and $y_{\gr}.$  However, the coefficient of $b_0$ in $y_{\gr}^2$ is $0,$ so $y_{\gr}^2\in\ff y_{\gr}.$  Since $y_{\gr}=a+2\mu b_{\mu,\mu},$ we see that $y_{\gr}^2=(\mu(1-2\xi)+1)y_{\gr}.$  Comparing the coefficient of $b_{\mu,\mu}$ we get
\begin{equation}\label{555}
((1-2\xi)\mu+1)2\mu=4\mu^2\implies \mu (1-2\xi)+1=2\mu \implies 2\mu\xi=1-\mu
\end{equation}
By \eqref{444} and \eqref{555},
\[\
2\xi-1-\mu=(1-\mu)(2\xi-1)\implies -\mu=-\mu(2\xi-1)
\]
Since $\mu\ne 0,$ we get $\xi=1,$ a contradiction.
\medskip

{\bf Subcase Ib.} $\xi=0.$
\medskip

We have

\begin{equation}\label{333}
y_{\xi}^2=a-\sum_{(\eta,\eta)\in S^{\circ}}4\eta^2 b_{\eta,\eta}+\sum_{(\eta,\eta)\in S^{\circ}}4\eta^2b_{\eta,\eta}^2,
\end{equation}
Let $u:=\sum_{(\eta,\eta)\in S^{\circ}} \varepsilon_{\eta,\eta}^2 b_{\eta,\eta},$ and $u':=\sum_{(\eta,\eta)\in S^{\circ}} 4\eta^2 b_{\eta,\eta}^2.$ By  the fusion rules for $b$ and by \eqref{222}, \eqref{333}, \eqref{111}, respectively,
\begin{align*}
a-u'&\in A_1(b)+A_0(b)\\
a-u+u'&\in A_1(b)+A_0(b)\\
a+u+u'&\in A_1(b)+A_0(b)
\end{align*}
Hence $2a-u, 2u\in A_1(b)+A_0(b),$ so $a,u\in A_1(b)+A_0(b).$
Multiplying $u$ by $a$ any number of times, we see that $\sum_{\eta,\eta)\in S^{\circ}} \eta^k b_{\eta,\eta}\in A_1(b)+A_0(b),$ for all $k\ge 2.$  This implies that $b_{\eta,\eta}\in A_1(b)+A_0(b),$ for all $(\eta,\eta)\in S^{\circ}(a).$  But since $b\in A_1(b)+A_0(b),$ we see that also $b_0\in A_1(b)+A_0(b),$ so $A=A_1(b)+A_0(b).$

If $\dim A=3,$ then, since $y_{\zeta}\in A_{\zeta,\zeta}(b),$ for $\zeta\in \{\gr,\xi\},$ Theorem \ref{eigcom}(i) implies that $\theta_0=0,\ \theta_1=\frac{1}{4},$ and $\ga_b=\half.$  Let $S^{\circ}(a)=\{(\mu,\mu)\}.$  Then by~Remark~\ref{t0t1}, $\theta'_0=0$ and $\theta'_1=\frac{1}{\mu}\theta_1=\frac{1}{4\mu}.$  Hence $b_{\mu,\mu}^2=\frac{1}{4\mu}a.$  Also by Theorem \ref{Ser02}(vii),
\[
\textstyle{b_0^2=b_0+(\ga_b-\ga_b^2)a-b_{\mu,\mu}^2=b_0+\frac{1}{4}a-\frac{1}{4\mu}a=\frac{\mu-1}{4\mu}a+b_0,}
\]
So $A$ is as in Example \ref{3dim01}.

If $\dim A=4,$ then, by (b), $A$ is as in example \ref{4dim}.
\medskip

{\bf Case II.} $y_{\xi}\notin A_{\xi,\xi}(b),$ for all $\xi\ne 1.$
\medskip

By \eqref{theta0} (and since $\theta_0=0$ because $y_{\gr}\in A_{\gr}(b)$), we get $0=2\xi-1+2\theta'_0.$  Hence $\theta'_0\ne 0.$  If $\dim A=3$ then, by Remark \ref{t0t1}, $\theta'_0=0$ since $\theta_0=0,$ a contradiction.  Hence $\dim A=4.$ Also, $A_{\half,\half}(b)=0,$ or else, by part (a) and by Remark \ref{t0t1}, $\theta'_0= 0.$  Thus, since $\dim A=4,$ there exist $\zeta\notin\{\half,1,\xi\},$ with $A_{\zeta,\zeta}(b)\ne 0.$  But then, as above, $0=2\zeta-1+2\theta'_0,$ a contradiction, since $\zeta\ne\xi.$
\medskip

(2)
By (a), $b_{\eta,\eta}^2=0,$ for all $(\eta,\eta)\in S^{\circ}(a),$ and $b_{\eta,\eta} \in A_{\half,\half}(b).$  Also, by Theorem \ref{eigcom}(i), $\gr-\ga_b=0$ since $\theta_1=0,$  so $\ga_b=1.$ This is Example~\ref{4dim2}. \end{proof}

\begin{example}\label{egc4} (A 4-dimensional CPAJ)
Let
\[
A=\ff a+\ff b_0+\ff b_{2,2}+\ff b_{\half,\half},
\]
with $A$ commutative, and
\[
\textstyle{b_{\half,\half}^2=0,\quad b_{2,2}^2=-\half b_0,\quad b_{\half,\half}b_{2,2}=0.}
\]
\[
\textstyle{b_0^2=\frac{3}{2}b_0,\qquad  b_0b_{\half,\half}=0,\qquad b_0b_{2,2}=-\frac{3}{2}b_{2,2}.}
\]
\[
\text{Let }b=a+b_0+b_{2,2}+b_{\half,\half}.
\]
Write
\[
\textstyle{a=b+a_0+a_{2,2}+a_{\half,\half},\quad a_0\in A_0(b),\quad a_{\eta,\eta}\in A_{\eta}(b).}
\]
The following hold.
\begin{enumerate}
\item
$a_0=-\half(b_0+3b_{2,2}),\quad a_{2,2}=-\half(b_0-b_{2,2}),\quad a_{\half,\half}=-b_{\half,\half}.$  The proof will be given in Example \ref{egnon} below.

\item

\begin{gather*}
b+a_0+a_{2,2}+a_{\half,\half}=a+b_0+b_{2,2}+b_{\half,\half}\\
\textstyle{(-\half b_0-\frac{3}{2}b_{2,2})+(-\half b_0+\half b_{2,2})+(-b_{\half,\half})=a.}
\end{gather*}
\item
It is easy to check that $b$ satisfies the fusion rules.  In fact it will be done in Example \ref{egnon}.
\end{enumerate}

\end{example}


\begin{thm}\label{comm1}
Assume that $A$ is commutative,   $\dim A_0(b)=\dim A_1(b)=1,$ and   $b$ satisfies the fusion rules. Then
\begin{enumerate}\eroman
\item
For all $(\gr,\gr)\in S^{\circ}(b),$ $\dim A_{\gr,\gr}(b)=1.$

\item
If $b_{\mu,\mu}^2=0,$ for some $(\mu,\mu)\in S^{\circ}(a),$ then $\mu=\half.$

\item
Suppose $\dim A\ge 4,$ so there exists $\gr\notin\{\half,0,1\}$ such that $A_{\gr,\gr}(b)\ne 0,$ and let
\[
y:=\ga y_{\gr}+\gb y'_{\gr}\in A_{\gr,\gr}(b).
\]
Then
\begin{enumerate}
    \item If $0\ne z:=\sum_{(\mu,\mu)\in S^{\circ}}\gc'_{\mu,\mu}b_{\mu,\mu}\in A_{\half,\half}(b),$ then $z=\gc'_{\half,\half}b_{\half,\half},$ so $b_{\half,\half}^2=0.$

    \item
Assume $A_{\half,\half}(b)\ne 0;$ then $A_{\half,\half}(a)\ne 0,$ and $b_{\half,\half}^2=0.$

    \item If $b_{\half,\half}^2=0,$ then $\ga+\gb=0;$ thus we can take $\ga=-1$ and $\gb=1,$ that is $-y_{\gr}+y'_{\gr}\in A_{\gr,\gr}(b).$

    \item
    Assume that $A_{\half,\half}(a)\ne 0.$  Then $A$ is as in Example \ref{egc4}.

\item $\dim A =  4$.
\end{enumerate}



\end{enumerate}
\end{thm}
\begin{proof}
(i) This follows from Lemma \ref{Ser23} applies to $b.$
\medskip

(ii)  By Theorem \ref{Ser041}(iii), $b_{\mu,\mu}\in A_{\half,\half}(b).$  By Theorem \ref{Ser041}(iii) applied to $b,$ $b_{\mu,\mu}\in A_{\half,\half}(a),$ so $\mu=\half.$

(iii) (a)
We have
\[
y_{\gr}z=\sum_{(\mu,\mu)\in S^{\circ}}\mu\gc'_{\mu,\mu}b_{\mu,\mu}+w,\quad w\in A_1(a)+A_0(a),
\]
and by Theorem \ref{Ser02}(v),
\[
\textstyle{y'_{\gr}z=\sum_{(\mu,\mu)\in S^{\circ}}\half\gc'_{\mu,\mu}b_{\mu,\mu}+w',\quad w'\in A_1(a)+A_0(a).}
\]
Thus,
\[
\textstyle{0=yz=(\ga y_{\gr}+\gb y'_{\gr})z=\sum_{(\mu,\mu)\in S^{\circ}}\left(\ga\mu+\half\gb\right)\gc'_{\mu,\mu}b_{\mu,\mu}+\ga w+\gb w'.}
\]
It follows that $\ga\mu+\half\gb=0,$ for each $(\mu,\mu)\in S^{\circ}(a)$ such that $\gc'_{\mu,\mu}\ne 0.$ This shows that there is a unique $\mu$ such that $\gc'_{\mu,\mu}\ne 0.$  Thus by Theorem~\ref{Ser041}(iii), $b_{\mu,\mu}^2=0,$ so by (ii), $\mu=\half,$ and (a) follows.
\medskip

(b) If $A_{\hal,\half}(b)\subseteq V(a),$ then by (a), $A_{\half,\half}(a)\ne 0,$ and $b_{\half,\half}^2=0.$ Otherwise, by Theorem \ref{Ser04}(v)(b), $\dim A\le 3,$ a contradiction.

(c)   We saw in the proof (a) that $0=yb_{\half,\half}=(\half\ga+\half\gb)b_{\half,\half},$ so $\ga+\gb=0.$

(d) Since $A_{\half,\half}(a)\ne 0,$   by symmetry and by (b),   $A_{\half,\half}(b)\ne 0$ and $b_{\half,\half}^2=0.$
We prove a series of claims.

\begin{itemize}
    \item[(1)] $b_{\half,\half}\in A_{\half,\half}(b).$ This follows from Theorem \ref{Ser02}(vi).

    \item[(2)]  $\dim A=4.$ Indeed, suppose there exists $\xi\in\ff\setminus\{0,1,\half,\gr\},$ and $z\in A_{\xi,\xi}(b),$ with $z\ne 0.$ Then, as above,  $z=-y_{\xi}+y'_{\xi}.$ Hence, by~Lemma~\ref{calcu}(viii)(a), $S^{\circ}(a)\subseteq \{(\half,\half)\}$ (because $\ga=\gb=-1$ there), implying $\dim A\le 3,$ a contradiction.

    \item[(3)]  Set $S^{\circ}(a)=\{(\half,\half),(\mu,\mu)\},$ with $\mu\ne\half.$ Then $a_{\gr,\gr}=-\half y,$ $\gr=\mu,$ $\theta'_0=\half(1-\mu)$ and $\theta'_1=\half(1-\ga_b)\mu.$

    Proof of (3):  By Lemma \ref{calcu}(iv)(a) (taking $\ga=-1$),
\begin{equation}\label{yyy}
y=(\ga_b-1)a+b_0+\frac{1}{2\gr-1}\sum_{(\eta,\eta)\in S^{\circ}}(1-2\eta)b_{\eta,\eta}.
\end{equation}
Let $\gl\in\ff,$ such that $a_{\gr,\gr}=\gl y.$  By Theorem \ref{Ser02}(v) (applied to $b$),
\begin{equation}\label{lambda}
\textstyle{ay=\half y+\gl y^2.}
\end{equation}
By \eqref{yz} (noting $\gr =\xi$ and    taking $\ga=\gb =-1$ ),
\begin{equation}\label{y2y2}
y^2=(1-\ga_b)a+b_0-b_{\mu,\mu}^2+\frac{(1-2\mu)^2}{(2\gr-1)^2}b_{\mu,\mu}^2+\frac{(1-2\mu)^2}{2\gr-1}b_{\mu,\mu}.
\end{equation}
Comparing the coefficient of $b_{\mu,\mu}$ in \eqref{lambda} we get
\[
(1-2\mu)\mu=\half(1-2\mu)+\gl(1-2\mu)^2,
\]
Cancelling $1-2\mu,$ we get $\mu-\half=\gl(1-2\mu),$ so $\gl=-\half.$  Note now that $\theta'_0\ne 0,$ else, by Remark \ref{t0t1}, $\theta_0=\mu\theta'_0=0,$ so by \eqref{theta0}, $(2\gr-1)(\gr-1)=0,$ a contradiction.

Comparing the coefficient of $b_0$ in \eqref{lambda} we get
\[
0=\half b_0-\half \left(b_0-\theta'_0+\frac{(1-2\mu)^2}{(2\gr-1)^2}\theta'_0\right).
\]
Hence cancelling $\theta'_0,$ we get $(1-2\mu)^2=(1-2\gr)^2,$ or
\begin{equation}\label{gr=mu}
(\gr-\mu)(2-2(\gr+\mu))=0.
\end{equation}
By \eqref{yz},
\begin{equation*}
\begin{aligned}
by&=b_0-\sum_{(\eta,\eta)\in S^{\circ}}b_{\eta,\eta}^2
+\frac{1}{2\gr-1}\sum_{(\eta,\eta)\in S^{\circ}}(1-2\eta)b_{\eta,\eta}^2\\
&+\frac{\gr}{2\gr-1}\sum_{(\eta,\eta)\in S^{\circ}}(1-2\eta)b_{\eta,\eta}.
\end{aligned}
\end{equation*}
Since $by=\gr y,$ we get
 \[
b_0-b_{\mu,\mu}^2+\frac{1-2\mu}{2\gr-1}b_{\mu,\mu}^2=(\ga_b-1)\gr a+\gr b_0.
\]
or $b_0+\frac{2-2\mu-2\gr}{2\gr-1}b_{\mu,\mu}^2=(\ga_b-1)\gr a+\gr b_0.$ Note that if $\gr+\mu=1,$ we get $\gr=1,$ a contradiction. Hence by \eqref{gr=mu}, $\gr=\mu,$ and
\[
\textstyle{b_{\mu,\mu}^2=\frac{2\gr-1}{2-2\gr-2\mu}\left((\ga_b-1)\gr a+(\gr-1)b_0\right)=-\half((\ga_b-1)\mu a+(\mu-1)b_0).}
\]
Thus
\[
\textstyle{\theta'_0=\half(1-\mu),\qquad \theta'_1=\half(1-\ga_b)\mu.}
\]

    \item[(4)]
    $y=(\ga_b-1)a+b_0-b_{\mu,\mu},\quad y^2=(1-\ga_b)a+b_0+(2\mu-1)b_{\mu,\mu}.$ This follows from \eqref{yyy} and \eqref{y2y2}, since $\gr=\mu.$


    \item[(5)]  Let $w\in A_0(b).$ Then $w=\tau y_{\gr_0}+y'_{\gr_0}\in A_0(b),$ for $\gr_0 =0,$ also, $\tau=\frac{1}{1-\mu}$ and
    \[
    w=(\ga_b+\tau)a+b_0-\frac{2-\mu}{1-\mu}b_{\half,\half}-\frac{\mu+1}{1-\mu}b_{\mu,\mu}.
    \] Proof of (5):  Since, by Remark \ref{t0t1}, $\theta_0=\mu\theta'_0,$ and since $\theta'_0\ne 0,$ by (3), Theorem \ref{eigcom}(i) implies  that $w\ne y_{\gr_0}.$ So we can take $w=\tau y_{\gr_0}+y'_{\gr_0}.$ By \eqref{theta0},
\begin{align*}
2\tau\mu\theta'_0&=1-2\theta'_0\quad\iff\\
\tau\mu(1-\mu)&=1-(1-\mu)=\mu.
\end{align*}
The formula for $w$ comes from Lemma \ref{calcu}(iv)(a), replacing  $y$ by $w$,  and $\ga$ by $\tau$, and noticing that since $\tau=\frac{1}{1-\mu},$ $2\tau\mu+1=\frac{\mu+1}{1-\mu}$ and $\tau+1=\frac{2-\mu}{1-\mu}.$

    \item[(6)] $\mu=2,\ \ga_b=1,\ \theta'_0=\half,\, \theta'_1=0,\ b_0^2=\frac{3}{2}b_0, \ b_{2,2}^2=\half b_0.$ 

Proof of (6):
 By the fusion rules, $y^2=\gl b+\gd w,\ \gl,\gd\in\ff.$  Comparing the coefficient of $b_0,$ we get $\gl+\gd=1.$  Comparing the coefficient of $b_{\half,\half}$ we get
\begin{gather*}
\textstyle{\gl-\gd\frac{2-\mu}{1-\mu}=0\implies 1-\gd\left(1+\frac{2-\mu}{1-\mu}=0\right)}\\
\implies 1-\mu-\gd(3-2\mu)=0.
\end{gather*}
So $\gl=\frac{\mu-2}{2\mu-3},\quad \gd=\frac{\mu-1}{2\mu-3}.$

Comparing now the coefficient of $b_{\mu,\mu}$ we get
\[
\textstyle{2\mu-1=\frac{\mu-2}{2\mu-3}+\frac{\mu-1}{2\mu-3}\!\cdot\!\frac{\mu+1}{\mu-1}=\frac{2\mu-1}{2\mu-3},}
\]
Hence $2\mu-3=1,$ so $\mu=2.$

Next, using \eqref{theta1},   since $\theta'_1=1-\ga_b$ (recalling that $\ga_y-\ga_b=-1$),
\[
-4(1-\ga_b)= -3(2-\ga_b)+3\ga_b+2(1-\ga_b)\implies -6(1-\ga_b)=-6+6\ga_b,
\]
so $\ga_b=1.$  The formula for $\theta'_0$ and $\theta'_1$ comes from (3), then, since $b_{\half,\half}^2=0,$ we have $b_{2,2}^2=\theta'_1 a+\theta'_0b_0=\half b_0.$  Finally, the formula for $b_0^2$ comes from Theorem \ref{Ser02}(viii).



To see (d), note that the multiplication table follows from (6) and  Theorem \ref{Ser02}(ii) (which gives the product $b_0b_{\eta,\eta}=\frac{1-2\eta}{2}$).
\end{itemize}

(e)   Otherwise $\dim A\ge 5,$ so $| S^{\circ}(a)|\ge 3.$ Then by Theorem \ref{Ser041}(iii), $A_{\half,\half}(a)\ne ~0,$ and hence by~Claim (2) of (d), $\dim A = 4,$ a contradiction \end{proof}





\begin{prop}\label{gr+xine1}
Assume the hypotheses of Theorem \ref{comm1}, with $\dim A\ne 3,$ and $A_{\half,\half}(b)=0.$ Then $\dim A=4.$ Put $S^{\circ}(a)=\{(\mu,\mu), (\nu,\nu)\},$ and $S^{\circ}(b)=\{(\gr,\gr), (\xi,\xi)\}.$
%
%
\begin{enumerate}\eroman
\item
$A_{\chi,\chi}(b)\ne \ff y_{\chi},$ for  $\chi\in \{\gr_0, \gr,\xi\},$ where $\gr_0=0.$
\medskip

Let $y=\ga y_{\gr}+y'_{\gr}\in A_{\gr,\gr}(b), z=\gb y_{\xi}+y'_{\xi}\in A_{\xi,\xi}(b),$ and $w=\tau y_{\gr_0}+y'_{\gr_0}\in A_0(b).$

\item we cannot have both $\ga=-\frac{1}{2\mu}$ and $\gb=-\frac{1}{2\nu}.$

\end{enumerate}

\end{prop}
\begin{proof}

(i) Let $\chi\in\{\gr_0,\gr,\xi\},$ and assume that $A_{\chi,\chi}(b)=\ff y_{\chi}.$ By Theorem \ref{eigcom}(i), $\theta_0=0.$  If $\chi=\gr_0,$ then by Theorem \ref{eigcom}(iii)(f), $\gr=\xi=\half-\theta'_0,$ a contradiction.

Suppose $\chi=\gr.$  Let $\gl\in\ff,$ with  $a_{\gr,\gr}=\gl y_{\gr}.$ Of course $\gl\ne 0.$ Then, by Theorem \ref{Ser02}(v) applied to $b,$ we have
\[
\textstyle{a(\gl y_{\gr})=\half (\gl y_{\gr})+\gl^2y_{\gr}^2\implies ay_{\gr}=\half y_{\gr}+\gl y_{\gr}^2.}
\]
This contradicts Lemma \ref{calcu}(viii)(c).
\medskip

(ii)  Assume that $\ga=-\frac{1}{2\mu}$ and $\gb=-\frac{1}{2\nu}.$
We obtain a series of claims, that culminate in a contradiction.

\begin{itemize}

    \item[(1)] $b_{\mu,\mu}^2+b_{\nu,\nu}^2=(\ga\gb+(\ga+\gb+1)\ga_b)a+b_0=\theta'_1a+\theta'_0b_0.$









\medskip

This  follows  from \eqref{yz}, since $yz=0$ by the fusion rules for $b.$
\medskip

    \item[(2)] $y=(\ga_b-\frac{1}{2\mu})a+b_0+ \frac{1}{2\gr-1}(1-\frac{\nu}{\mu})b_{\nu,\nu}.$

\medskip

This follows from Lemma \ref{calcu}(iv)(a).
\medskip

     \item[(3)]  $w=(\ga_b+\tau)a+b_0-(2\tau\nu+1)b_{\nu,\nu}-(2\tau\mu+1)b_{\mu,\mu}.$


    \medskip


 This comes from Lemma \ref{calcu}(iv)(a), replacing $y$ by $w$ and $\ga$ by~$\tau.$



\medskip

\item[(4)] Write $b_{\nu,\nu}^2=\theta''_1a+\theta''_0b_0$ and $b_{\mu,\mu}^2=\theta'''_1a+\theta'''_0b_0.$ Then $\theta''_0+\theta'''_0=1,$ and
$(2\gr-1)\gr=(1-\frac{\nu}{\mu})\theta''_0,\qquad (2\xi-1)\xi=(1-\frac{\mu}{\nu})\theta'''_0.$

 \medskip

Indeed, by (1), $\theta'_0=1,$ and so $1=\theta'_0=\theta''_0+\theta'''_0.$ Note that since $\theta'_0=1,$ Theorem \ref{Ser02}(vii) shows that the $(0,0)$-component of $b_0^2$ is $0.$ Hence the $(0,0)$-component of $by$ is $\frac{1}{2\gr-1}(1-\frac{\nu}{\mu})\theta''_0.$ Since $by=\gr y,$ comparing the components in the $(0,0)$-eigenspace of $A$ we get (4) for $\gr,$ and, by symmetry, we get (4) for $\xi.$
\medskip

   \item[(5)]
$-\frac{\theta_0}{\mu}=(2\gr+1)(\gr-1),\quad -\frac{\theta_0}{\nu}=(2\xi+1)(\xi-1).$

\medskip

Indeed, since $y\in A_{\gr,\gr}(b),$ we can use \eqref{theta0}.  Note that by (1), $\theta_0'=1.$ Also $\ga_y-\ga_b=-\frac{1}{2\mu}.$ Hence (5) follows from \eqref{theta0}.

\item[(6)] $\theta_0=-\frac{1}{2\tau}.$

\medskip
  Since $w\in A_0(b),$ and since $\theta'_0=1,$ we get (6) from \eqref{theta0}.
\medskip

\item[(7)]  $\theta''_0=\frac{\frac{1}{2\tau}+\mu}{\mu-\nu}.$
\medskip


Follows from
$-\frac{1}{2\tau}=\theta_0=\nu\theta_0''+\mu(1-\theta''_0).$


%


\item[(8)] $\tau=\frac{\frac{3}{2}-2\gr}{2\mu(\gr-1)},$ so $\theta_0=\frac{\mu(\gr-1)}{2\gr-\frac{3}{2}}.$

\medskip

By \eqref{yz} (taking $z=y$) we get
\begin{equation*}
\begin{aligned}
y^2&\textstyle{=(\ga^2+(2\ga+1)\ga_b)a+b_0-\sum_{(\eta,\eta)\in S^{\circ}}b_{\eta,\eta}^2+ \frac{1}{(2\gr-1)^2}(1-\frac{\nu}{\mu})^2b_{\nu,\nu}^2}\\
&\qquad+\textstyle{\frac{1}{2\gr-1}(1-\frac{\nu}{\mu})^2b_{\nu,\nu}.}
\end{aligned}
\end{equation*}
Using the fusion rules write $y^2=\gl w+\gd b,\ \gl,\gd\in\ff.$ Now
\begin{align*}
&\gl w+\gd b=\gl(\ga_b+\tau)a+\gl b_0-\gl(2\tau\nu+1)b_{\nu,\nu}-\gl(2\tau\mu+1)b_{\mu,\mu}+\gd b\\
&=(\gl(\ga_b+\tau)+\gd\ga_b)a+(\gl+\gd)b_0-(\gl(2\tau\nu+1)-\gd)b_{\nu,\nu}-(\gl(2\tau\mu+1)-\gd)b_{\mu,\mu}.
\end{align*}
Since the coefficient of $b_{\mu,\mu}$ in $y^2$ is $0,$ $\gd=\gl(2\tau\mu+1).$  Comparing the coefficients of $b_{\nu,\nu}$ we get
\[
\textstyle{\frac{1}{2\gr-1}(1-\frac{\nu}{\mu})^2=-(\gl(2\tau\nu+1)-\gd),}
\]
so
\begin{equation}\label{2gl}
\textstyle{\frac{1}{2\gr-1}(1-\frac{\nu}{\mu})^2=2\gl\gt(\mu-\nu).}
\end{equation}

Comparing the $(0,0)$-components (using $\theta'_0=1)$ gives
\[
\textstyle{\gl+\gd=\frac{1}{(2\gr-1)^2}(1-\frac{\nu}{\mu})^2\theta''_0.}
\]
or
\[
\textstyle{2\gl(\tau\mu+1)=\frac{1}{(2\gr-1)^2}(1-\frac{\nu}{\mu})^2\theta''_0.}
\]
Substituting for $2\gl$ from \eqref{2gl} and cancelling we get
\[
\textstyle{\frac{\tau\mu+1}{\tau(\mu-\nu)}=\frac{1}{(2\gr-1)}\theta''_0.}
\]
By (8), $\theta_0''=\frac{\frac{1}{2\tau}+\mu}{\mu-\nu}.$ Using this and canceling  we get
\[
(\tau\mu+1)(2\gr-1)=\half+\tau\mu
\]
so $2\tau\mu(\gr-1)=\frac{3}{2}-2\gr.$  This show the first equality of (8).  The second equality follows from (5).
\end{itemize}
\medskip

We can now obtain the final contradiction.  By Claims (8) and  (5),
\[
-\frac{(\gr-1)}{2\gr-\frac{3}{2}}=(2\gr+1)(\gr-1).
\]
Hence $-1=(2\gr+1)(2\gr-\frac{3}{2})=4\gr^2-\gr-\frac{3}{2}.$ Hence $8\gr^2-2\gr-1=0,$ that is $(4\gr+1)(2\gr-1)=0.$  Since $\gr\ne\half,$ we conclude that $\gr=-\frac{1}{4}.$  By symmetry, we conclude that $\gr=\xi,$ a contradiction.
\end{proof}

\begin{thm}\label{gr+xi=1}
Suppose $A =\lan\lan a,b \ran\ran$ is commutative, where $a,b$ are weakly primitive axes satisfying the fusion rules, and $\dim A_0(a)= \dim A_0(b) = 1$.  Then either $A$ is as in Example \ref{egc4}, or $A$  is an HRS algebra of dimension $\le 3$.

\end{thm}
\begin{proof} By Theorem \ref{comm1}(iii)(e), $\dim A\le 4.$
 Assume $\dim A=4.$ By Theorem~ \ref{comm1}(b,d) we may assume that $A_{\half,\half}(b) = A_{\half,\half}(a) =  0,$  for otherwise $\dim A\le 3,$ so by \S\ref{3d}, $A$ is an HRS algebra. We use the notation of Proposition \ref{gr+xine1}.  We obtain a contradiction by showing  that necessarily $\ga=-\frac{1}{2\mu}$ and $\gb=-\frac{1}{2\nu}.$

 If $\gr+\xi\ne 1,$ this follows from Lemma \ref{calcu}(viii)(a), since $yz=0,$ by the fusion rules. So suppose $\gr+\xi=1.$
  We shall  obtain a contradiction  by a series of assertions.
\begin{enumerate}\eroman
\item
$y=(\ga_b+\ga)a+b_0+\frac{1}{2\gr-1}\sum_{(\eta,\eta)\in S^{\circ}}(2\ga\eta+1)b_{\eta,\eta},$
and
\begin{equation*}
\begin{aligned}
y^2=&(\ga^2+(2\ga+1)\ga_b)a+b_0-\sum_{(\eta,\eta)\in S^{\circ}}b_{\eta,\eta}^2 \\
&+\frac{1}{(2\gr-1)^2}\sum_{\eta\in \{\mu,\nu\}}(2\ga\eta+1)^2b_{\eta,\eta}^2+\frac{1}{2\gr-1}\sum_{\eta\in\{\mu,\nu\}}(2\ga\eta+1)^2b_{\eta,\eta}.
\end{aligned}
\end{equation*}

 This follows from Lemma \ref{calcu}(iv) and \eqref{yz}.
\item
$\ga,\gb\notin\{-1\},$ and $\ga\ne\gb.$
\medskip

To see this, write Write $a=\gb_ab+a_0 +\sum_{(\mu,\nu)\in S^\circ(b)}a_{\mu,\nu},$ with $a_0=\varphi w.$  By Theorem~\ref{Ser02}(iii) applied to $b,$ we have
\begin{equation}\label{a0y}
\varphi wy=\frac{1-2\gr\gb_a}{2}y,\quad \varphi wz=\frac{1-2\xi\gb_a}{2}z
\end{equation}
Recall from \eqref{wy} that
\begin{equation*}
\varphi wy=u-\frac{\varphi\gr}{2\gr-1}\sum_{(\eta,\eta)\in S^{\circ}}(2\ga\eta+1)(2\gt\eta+1)b_{\eta,\eta},\quad u\in\ff a+\ff b_0.
\end{equation*}
Then  by \eqref{a0y} applied to  $y$ and then to $z,$ comparing the coefficient of $b_{\eta,\eta},$ we get for $\eta\in\{\mu,\nu\},$
\begin{equation}\label{phi}
-\varphi\gr(2\ga\eta+1)(2\tau\eta+1)=\frac{1-2\gr\gb_a}{2}(2\ga\eta+1)\end{equation}
\begin{equation}\label{phi1}
-\varphi\xi(2\gb\eta+1)(2\tau\eta+1)=\frac{1-2\xi\gb_a}{2}(2\gb\eta+1).
\end{equation}
Next, for some fixed $\eta'\in\{\mu,\nu\},$ we want to cancel $2\ga\eta'+1$ in \eqref{phi} and $2\gb\eta'+1$ in \eqref{phi1}, since that would lead to   $2\tau\eta'+1=-\frac{1-2\gr\gb_a}{2\varphi\gr}$ and   $2\tau\eta'+1=-\frac{1-2\xi\gb_a}{2\varphi\xi},$  yielding
$\frac{1-2\gr\gb_a}{2\varphi\gr}=\frac{1-2\xi\gb_a}{2\varphi\xi},$ or $\frac{1}{\gr}=\frac{1}{\xi},$ a contradiction. Thus we need $\eta'$ such that  $2\ga\eta'+1,\  2\gb\eta'+1 \ne 0.$ Clearly  $2\gb\eta'+1\ne 0,$  for {\it some} $\eta'\in\{\mu,\nu\}.$

If $\ga=-1,$ then $2\ga\eta'+1\ne 0,$ because $\eta'\ne\half,$ since $A_{\half,\half}(a)=0.$
Similarly, when $\ga=\gb,$     $2\ga\eta'+1\ne 0$.
\medskip

\end{enumerate}
We  now obtain the final contradiction, by showing that $\ga=-\frac{1}{2\mu},$ and $\gb= -\frac{1}{2\nu}.$  Indeed write $a_{\gr,\gr}=\gl y;$  $\gl\ne 0$ since $\dim A >3.$ By Theorem \ref{Ser02}(v), applied to $b,$
\[
\textstyle{a(\gl y)=\half(\gl y)+\gl^2y^2\implies ay=\half y+\gl y^2,}
\]
so comparing the coefficient of $b_{\eta,\eta},$
\[
\eta(2\ga\eta+1)=\half(2\ga\eta+1)+\gl(2\ga\eta+1)^2,\quad\eta\in\{\mu,\nu\}.
\]
In case $2\ga\eta+1\ne 0,$ we get $\eta=\half+\gl(2\ga\eta+1),$ so
$\eta( 1 -2\gl\ga) = \half +\gl.$ If $\gl=-\half$ then, since $\eta\ne 0,$ $\ga=-1,$  contradicting (ii).  Hence $\gl\ne -\half,$ so $\eta  = \frac{ \half +\gl}{1 -2\gl\ga}.$

This implies that $2\ga\eta+1= 0$ for some $\eta \in \{\mu,\nu\}$.    Without loss we may assume that $\eta = \mu,$ so $\ga=-\frac{1}{2\mu}.$ By (ii), and by symmetry, $\gb=-\frac{1}{2\nu}.$

Hence we may assume that $\dim A\le 3,$ so by Remark \ref{rem03}(2), $A$ is an HRS algebra.

\end{proof}

In particular,  this completes the proof Theorem~A.
\subsection{The case when $A$ is not commutative}$ $
\medskip

In this subsection we assume that $A$ is noncommutative.  We start with an example.  Note that by Theorem \ref{eigcom}(iii)(c)(1), if $\dim A_{\gr,\gr}(b)=2,$ then $\gr=\ga_b=1.$

\begin{example}\label{noncom2} (Generalizing Example~\ref{4dim2})
Let $A=\ff a+\ff b_0+\sum_{(\mu,\nu)\in S^{\circ}}\ff b_{\mu,\nu},$ with $\varepsilon_{\mu,\nu}=\mu+\nu=1,$ for all  $(\mu,\nu)\in S^{\circ}.$
\[
b_{\mu,\mu}b_{\mu',\nu'}=0,\quad\text{for all }(\mu,\nu), (\mu',\nu')\in S^{\circ}.
\]
\[
b_0^2=b_0,\quad b_0b_{\mu,\nu}=b_{\mu,\nu}b_0=0,\ \forall (\mu,\nu)\in S^{\circ}.
\]
Let
\[
b=a+b_0+\sum_{(\mu,\nu)\in S^{\circ}}b_{\mu,\nu}.
\]
Clearly $b^2=b.$ Also
\begin{enumerate}
\item
$bb_{\mu,\nu}=\mu b_{\mu,\nu}$ and $b_{\mu,\nu}b=\nu b_{\mu,\nu}.$

\item
$A_0(b)=0,$ $A_1(b)=\ff b+\ff y,$ with $y=a+\sum_{(\mu,\nu)\in S^{\circ}}b_{\mu,\nu}.$

\item
$by=y^2=y.$

\item
$b$ satisfies the fusion rules.
\end{enumerate}
\end{example}

\begin{thm}\label{rho1}
Suppose that $A$ is not commutative and $\dim A_1(b)=2.$ Set $\gr=1.$ Recall from Theorem \ref{eigcom} that
\[
\textstyle{y_{\gr}=a+\sum_{(\mu,\nu)\in S^{\circ}}\varepsilon_{\mu,\nu}b_{\mu,\nu},}
\]
and that $A_1(b)=\ff b+\ff y_{\gr}.$
\begin{enumerate}\eroman
 \item \begin{enumerate}
    \item
$\ga_b=1,$  $b_{\mu,\nu}^2=0,$ for all $(\mu,\nu)\in S^{\dagger},$ and $\sum_{(\mu,\nu)\in S^{\circ}}\varepsilon_{\mu,\nu} b_{\mu,\nu}^2=0.$ Thus, in the notation of Theorem \ref{eigcom}, $\theta_0=\theta_1=0.$  

\item
$A_{\half,\half}(b)\subseteq V(a).$ 

\item
If $A_1(b)A_{\half,\half}(b)\subseteq A_{\half,\half}(b),$ then $\sum_{(\mu,\nu)\in S^{\circ}}\varepsilon_{\mu,\nu}^nb_{\mu,\nu}^2=0$ for all $n\ge 1.$

\item
$b_{\mu,\mu}^2=0,$ for all $(\mu,\mu)\in S^{\circ}.$


\item
$\sum_{(\mu,\nu)\in S^{\circ}} b_{\mu,\nu}^2 =0.$  Thus, in the notation of Theorem \ref{eigcom}, $\theta'_0=\theta'_1=0.$




\item
$A_{\gr,\gr}(b)=0,$ for $\gr\notin\{1,\half\}.$ Also for $(\mu,\nu)\in S^{\dagger},$ $b_{\mu,\nu}$ is a $\left(\frac{\mu}{\mu+\nu},\frac{\nu}{\mu+\nu}\right)$-eigenvector of $b.$

\item
$A_0(b)=0$ and if $A_1(b)$ is a subalgebra of $A,$ then $\varepsilon_{\mu,\nu}=1,$ for all $(\mu,\nu)\in S^{\circ}.$

\end{enumerate}
\item
If $b$ satisfies the fusion rules, $A$ is as in Example \ref{noncom2}.



\end{enumerate}
\end{thm}
\begin{proof}
(i) (a)
By Theorem \ref{eigcom}(iii)(c)(1), $\ga_b=1,$  and, in the notation of Theorem \ref{eigcom}, $\theta_0=\theta_1=0,$ by Theorem \ref{eigcom}(i) (1)--(3), i.e.~$\sum_{(\mu,\nu)\in S^{\circ}}\varepsilon_{\mu,\nu}b_{\mu,\nu}^2=0.$    By Lemma \ref{b0sq}, $\bar S=\emptyset.$
\medskip

(b)
Assume that $A_{\half,\half}(b)\nsubseteq V(a).$ By Theorem \ref{Ser04}(v)(b) there exists at most one $\mu$ such that $(\mu,\mu)\in S^{\circ},$ and  if $(\mu,\mu)\in S^{\circ},$ then $b_{\mu,\mu}^2\ne 0,$ since $\ga_b=1.$  However, by (i), $b_{\mu,\mu}^2=0,$ a contradiction.
\medskip

(c)
The proof is by induction.  For $k=1,$ this follows from (i). Assume $\sum_{(\mu,\nu)\in S^{\circ}}\varepsilon_{\mu,\nu}^kb_{\mu,\nu}^2=\sum_{(\mu,\mu)\in S^{\circ}}\varepsilon_{\mu,\mu}^kb_{\mu,\mu}^2=0.$  Then
$\sum \varepsilon_{\mu,\mu}^kb_{\mu,\mu}\in A_{\half,\half}(b),$ by~Theorem \ref{Ser041}(iii). But then, also
$y_{\gr}\sum \varepsilon_{\mu,\mu}^kb_{\mu,\mu}\in A_{\half,\half}(b).$  But
\[
y_{\gr}\sum \varepsilon_{\mu,\mu}^kb_{\mu,\mu}=\sum \mu\varepsilon_{\mu,\nu}^k b_{\mu,\mu}+\sum\varepsilon_{\mu,\mu}^{k+1}b_{\mu,\mu}^2.
\]
Since $A_{\half,\half}(b)\subseteq V,$ the induction step is complete.
\medskip

(d)
By (i) and (iii), $\sum_{(\mu,\mu)\in S^{\circ}}\varepsilon_{\mu,\mu}^nb_{\mu,\mu}^2=0,$ for all $n\ge 1,$ so the claim  holds by a Vandermonde argument.
\medskip


(e)
This follows from (a) and (d).

\medskip

(f)  The first assertion is by Theorem \ref{eigcom}(iii)(f), since $\theta'_0=0.$  The second assertion holds since $b_{\mu,\nu}^2=0,$ and by Theorem \ref{Ser02}(vi), since $\ga_b=1.$

\medskip

(g)
By (f), $A_0(b)=0.$ Now
 $y_{\gr}^2=a+\sum_{(\mu,\nu)\in S^{\circ}}\varepsilon_{\mu,\nu}^2b_{\mu,\nu}\in A_1(b),$ implying $y_{\gr}^2-y_{\gr}\in A_1(b).$  But $y_{\gr}^2-y_{\gr}\in V(a),$ and $A_1(b)\cap V(a)=0$ by Theorem \ref{Ser04}(iii), since $\rho=\xi=1.$  Hence $y_{\gr}^2=y_{\gr},$ so $\varepsilon_{\mu,\nu}=1,$ for all $(\mu,\nu)\in S^{\circ}.$
\medskip

(ii)
This follows from (i)(a)--(g).
\end{proof}

\begin{lemma}\label{nonchalf}
Assume that $A$ is not commutative, that $b$ satisfies the fusion rules and that $b$ is weakly primitive. Then $\ga_b\ne\half.$
\end{lemma}
\begin{proof}
By Theorem \ref{Ser04}(i), if $(\gr,\xi)\in S^{\dagger}(b),$ then $A_{\gr,\xi}(b)\subseteq V.$  But if $\ga_b=\half,$ then, by Theorem \ref{Ser041}(i), $\gr=\xi=\half,$ a contradiction, since $S^{\dagger}(b)\ne\emptyset.$
\end{proof}
\begin{example}\label{egnon}
Let $T\subseteq\{(\mu,\nu)\in\ff\times\ff:\mu+\nu=1\},$ and let $S^{\circ}=T\cup\{(2,2)\},$ with $S^{\dagger}\ne\emptyset.$ Set
\[A=\ff a+\ff b_0+\sum_{(\mu,\nu)\in S^{\circ}}\ff b_{\mu,\nu}.
\]
and
\[
\textstyle{b=a+b_0+\sum_{(\mu,\nu)\in S^{\circ}}b_{\mu,\nu},\quad
b_{\mu,\nu}^2=0, \forall (\mu,\nu)\in T,\quad b_{2,2}^2=-\half b_0.}
\]
\[
b_{\mu,\nu}b_{\mu',\nu'}=0,\text{ for }(\mu,\nu)\ne (\mu',\nu').
\]
\[
\textstyle{b_0^2=\frac{3}{2} b_0,\quad b_0b_{2,2}=b_{2,2}b_0=-\frac{3}{2}b_{2,2}, \quad b_0b_{\mu,\nu}=b_{\mu,\nu}b_0=0,\ \forall (\mu,\nu)\in T.}
\]
Clearly $bb_{\mu,\nu}=\mu b_{\mu,\nu},$ and $b_{\mu,\nu}b=\nu b_{\mu,\nu},\ \forall (\mu,\nu)\in T.$

Write $a=\gb_ab+a_0+a_{2,2}+\sum_{(\mu,\nu)\in T}a_{\mu,\nu},$ with $a_0\in A_0(b),\ a_{\mu,\nu}\in A_{\mu,\nu}(b).$  The following hold:
\begin{enumerate}

\item
$b^2=a+b_0^2+b_{2,2}^2+(a+b_0)b_{2,2}+b_{2,2}(a+b_0)+\sum_{(\mu,\nu)\in T}b_{\mu,\nu}\\
= a+\frac{3}{2}b_0-\half b_0+\half b_{2,2}+\half b_{2,2}+\sum_{(\mu,\nu)\in T}b_{\mu,\nu}=b.$

Hence, $b$ is an idempotent.



\item
$A_0(b) =\ff (b_0+3b_{2,2}),\quad A_{2,2}(b)=\ff(b_0-b_{2,2}).$ Indeed
\begin{align*}
b(b_0+3b_{2,2})&=(a+b_0+b_{2,2})(b_0+3b_{2,2})\\
&\textstyle{=6b_{2,2}+\frac{3}{2}b_0-\frac{9}{2}b_{2,2}-\frac{3}{2}b_{2,2}-\frac{3}{2}b_0=0.}
\end{align*}
\begin{align*}
b(b_0-b_{2,2})&=(a+b_0+b_{2,2})(b_0-b_{2,2})\\
&=-2b_{2,2}+\frac{3}{2}b_0+\half b_0=2( b_0-b_{2,2}).
\end{align*}
\item
$a=b-\half(b_0+3b_{2,2})-\half(b_0-b_{2,2})-\sum_{(\mu,\nu)\in T} b_{\mu,\nu}.$ Consequently $a_0=-\half(b_0+3b_{2,2}),$  $a_{2,2}=-\half(b_0-b_{2,2}),$ and $a_{\mu,\nu}=-b_{\mu,\nu},$ for all $(\mu,\nu)\in T.$
\item
$(b_0+3b_{2,2})^2=b_0^2+9b_{2,2}^2+6b_0b_{2,2}=\frac{3}{2}b_0-\frac{9}{2}b_0-9b_{2,2}=-3b_0-9b_{2,2}=-3(b_0+3b_{2,2}).$ That is, $a_0^2=\frac{3}{2}a_0.$
\item
$(b_0+3b_{2,2})(b_0-b_{2,2})=b_0^2+2b_0b_{2,2}-3b_{2,2}^2=\frac{3}{2}b_0-3b_{2,2}+\frac{3}{2}b_0=3(b_0-b_{2,2}).$   That is $a_0a_{2,2}=-\frac{3}{2}a_{2,2}.$
\item
$(b_0-b_{2,2})^2=b_0^2+b_{2,2}^2-2b_0b_{2,2}=\frac{3}{2}b_0-\half b_0+3b_{2,2}=b_0+3b_{2,2}.$ That is $a_{2,2}^2=-\half a_{0}.$
\item $a_{2,2}b_{\mu,\nu}=-\half(b_0-b_{2,2})b_{\mu,\nu}=0.$
\item By (2), $A$ is a sum of paired eigenspaces of $b,$ so $b$ is an axis in $A.$ And (3)--(7), $a$ and $b$ act symmetrically, so $b$ satisfies all fusion rules.
\end{enumerate}
\end{example}

\begin{example}\label{egnon1}
Let $S^{\circ}\subseteq\{(\mu,\nu)\in\ff\times\ff: \mu+\nu=1\},$ with $(\half,\half)\in S^{\circ},$ and with  $S^{\dagger}\ne\emptyset.$
Set
\[A=\ff a+\ff b_0+\sum_{(\mu,\nu)\in S^{\circ}}\ff b_{\mu,\nu}.
\]
and
\[
\textstyle{b=a+b_0+\sum_{(\mu,\nu)\in S^{\circ}}b_{\mu,\nu},\quad
b_{\mu,\nu}^2=0, \forall (\mu,\nu)\in S^{\dagger},\quad b_{\half,\half}^2= b_0.}
\]
\[
b_{\mu,\nu}b_{\mu',\nu'}=0,\text{ for }(\mu,\nu)\ne (\mu',\nu').
\]
\[
\textstyle{b_0^2=0, \quad b_0b_{\mu,\nu}=b_{\mu,\nu}b_0=0,\ \forall (\mu,\nu)\in S^{\circ}.}
\]
Clearly $bb_{\mu,\nu}=\mu b_{\mu,\nu},$ and $b_{\mu,\nu}b=\nu b_{\mu,\nu},\ \forall (\mu,\nu)\in S^{\dagger}.$

Write $a=\gb_ab+a_0+a_{\half,\half}+\sum_{(\mu,\nu)\in S^{\dagger}}a_{\mu,\nu},$ with $a_0\in A_0(b),\ a_{\mu,\nu}\in A_{\mu,\nu}(b).$  The following hold:
\begin{enumerate}
\item
$A_0(b)=\ff b_0,$ $A_{\half,\half}(b)=b_0+\half b_{\half,\half}.$
\item
$a=b+b_0-2(b_0+\half b_{\half,\half})-\sum_{(\mu,\nu)\in S^{\dagger}}b_{\mu,\nu},$ so $a_0=b_0,$ $a_{\half,\half}=-2(b_0+\half b_{\half,\half}),$ and $a_{\mu,\nu}=-b_{\mu,\nu},\ \forall (\mu,\nu)\in S^{\dagger}.$
\item
$\ff a+\ff b_0+\ff b_{\half,\half}$ is the algebra $B(\half,1)$ in the notation of \cite{HRS}.
\end{enumerate}
\end{example}

\begin{thm}\label{Yoav1}
Assume that $A$ is not commutative and that $b$ is weakly primitive and satisfies the fusion rules, and $A_0(b)\ne 0.$  Then
\begin{enumerate}\eroman
\item
$\dim A_{\gr,\xi}(b)=1,$ for all $\gr,\xi.$

\item
By Lemma \ref{nonchalf}, $\ga_b\ne \half$ and $\gb_a\ne\half,$ so that $b_{\mu,\nu}^2=0$ for all $(\mu,\nu)\in S^{\dagger}.$ Furthermore,
\begin{enumerate}
\item $\mu+\nu=1,$ for all $\mu,\nu\in S^{\dagger}.$

\item
If $\ga_b\ne 0,$ then $A_0(b)=\ff w,$ where
\[
w=\frac{(\ga_b-1)^2}{\ga_b}a+b_0+\sum_{(\mu,\mu)\in S^{\circ}}\left(\frac{\varepsilon_{\mu,\mu}(2\ga_b-1)}{\ga_b}-1\right)b_{\mu,\mu}+\sum_{(\mu,\nu)\in S^{\dagger}}\frac{\ga_b-1}{\ga_b}b_{\mu,\nu}
\]


\item Suppose $\gb_a\ne 1.$ Then
\begin{itemize}
\item[(1)]
$A_{\mu,\mu}(a)=0,$ for all $(\mu,\mu)\in S^\circ(a)\setminus\{(\half,\half)\},$ and either $A_{\half,\half}(a)=0,$ or $A_{\half,\half}(a)=\ff a_{\half,\half},$ with $a_{\half,\half}^2=0.$
    \item[(2)] $\ga_b=0,$ and
    letting $\gr_0=0,$ $A_0(b)=\ff y_{\gr_0}.$
    \item[(3)] $bb_{\mu,\nu}=\nu b_{\mu,\nu}$ and $b_{\mu,\nu}b=\mu b_{\mu,\nu}.$
    \item[(4)] $A$ is as in Example \ref{E1}.
\end{itemize}

\item Suppose $\gb_a= 1.$ Then $\ga_b=1,$ and $A$ is as in Example \ref{egnon} or \ref{egnon1}.
\end{enumerate}

\end{enumerate}
\end{thm}
\begin{proof}
(i)  By Theorem \ref{eigcom}(iii)(c), $\dim A_0(b)=1,$ so (i) follows from Lemma \ref{Ser23} applied to $b.$
\medskip

(ii) (a) By Theorem \ref{Ser041}, $b_{\mu,\nu}\in A_{\gr,\xi}(b)$ (with $\gr+\xi=1$).  Hence $b_{\mu,\nu}\in V(b),$ so by Theorem \ref{Ser041} applied to $b,$ $\mu+\nu=1.$
\medskip

(b) Let $\gr_0=0.$  By Corollary \ref{rho0}(iii), since $\ga_b\ne 0,$  $A_0(b)\ne \ff y_{\gr_0}.$ Also $A_0(b)\ne\ff y'_{\gr_0},$ by Corollary \ref{rho0}(ii), since $\ga_b\ne\half.$ Hence (b) follows from  Corollary \ref{rho0}(iv), replacing $y$ by $w$ and using the fact that $\varepsilon_{\mu,\nu}=1,$  for  $(\mu,\nu)\in S^{\dagger}.$
\medskip

(c)
(1)  By Lemma  \ref{fu1}(ii)(a), applies to $b,$ since $\bar S(b)=\emptyset$ (by (ii) and by symmetry), $A_{\mu,\mu}(a)\subseteq V(b),$ for all $(\mu,\mu)\in S^{\circ}(a).$    By  Lemma \ref{Ser041} (applied to $b$), $S^{\circ}(a)\setminus S^{\dagger}\subseteq\{(\half,\half)\}.$
It follows that if $A_{\half,\half}(a)\ne 0,$ then, by Lemma \ref{Ser041}(iii) (applied to $b$), $A_{\half,\half}(a)=\ff a_{\half,\half},$ with $a_{\half,\half}^2=0.$
\medskip


(2)  First note that by (1), $b_{\mu,\nu}^2=0,$ for all $(\mu,\nu)\in S^{\circ}(a),$ so, in the notation of Corollary \ref{rho0}, $\theta_0=\theta_1=\theta'_0=\theta'_1=0.$ Assume $\ga_b\ne 0.$  Then the cases of Corollary \ref{rho0} (i), (ii) and (iii), do not hold, so we are in the case of Corollary \ref{rho0}(iv).  Thus we are in Corollary \ref{rho0}(iv)(2).  But then we get $0=1,$ a contradiction.  Hence $\ga_b=0,$ and by Corollary \ref{rho0}(iii), $A_0(b)=y_{\gr_0}.$
\medskip

(3)  This follows from Theorem \ref{Ser02}(v), since $\varepsilon_{\mu,\nu}=1,$ for all $(\mu,\nu)\in S^{\circ},$ and $\ga_b=0.$
\medskip

(4)  By (1), (2) and (3).
\medskip

(d)  Assume $\gb_a= 1.$ By (c)(4), and by symmetry, $\ga_b=1.$

We prove part (d) with a number of claims.
\begin{itemize}
\item[({\bf d1)}]
$w=b_0+\sum_{(\mu,\mu)\in S^{\circ}(a)}(\varepsilon_{\mu,\mu}-1)b_{\mu,\mu}.$
\smallskip

This follows from part (b).
\smallskip

\item[({\bf d2)}]
$b_0b_{\mu,\nu}=b_{\mu,\nu}b_0=0,$ for $(\mu,\nu)\in S^{\dagger}.$
\smallskip

This follows from Theorem \ref{Ser02}(ii), since $\ga_b=1,$ and $\varepsilon_{\mu,\nu}=1,$ for $(\mu,\nu)\in S^{\dagger}.$
\smallskip

\item[({\bf d3)}] $S^{\circ}(c)\setminus S^{\dagger}(c)\ne\emptyset,$ and $b_{\eta,\eta}^2\ne 0,$ for some $(\eta,\eta)\in S^{\circ}(c),$ for $c\in\{a,b\}.$
\smallskip

{\bf Proof of (d3).}
We prove the claim for $a,$ by symmetry it follows for $b$ as well. Suppose false.  Then, by (ii), $b_{\mu,\nu}^2=0,$ for all $(\mu,\nu)\in S^{\circ}(a).$  Thus, in the notation of Corollary \ref{rho0}, $\theta_1=\theta_0=\theta'_1=\theta'_0=0.$ Hence we are in Case (iv)(2) of Corollary \ref{rho0}, and then $1=0,$ a contradiction.
\medskip

\item[({\bf d4)}]
\begin{itemize}
\item[(1)]
$S^{\circ}(a)\setminus S^{\dagger}(a)\subseteq \{(\half,\half),(\mu,\mu)\},$ and $S^{\circ}(b)\setminus S^{\dagger}(b)\subseteq \{(\half,\half), (\gr,\gr)\},$ with $\mu,\gr\neq \half .$
\item[(2)]
$w=b_0+(2\mu-1)b_{\mu,\mu},$ if $(\mu,\mu)\in S^{\circ},$ and $w=b_0$ otherwise.
\item[(3)]
If $(\mu,\mu)\in S^{\circ}(a),$ then $b_0+4\mu(\mu-1)b_{\mu,\mu}^2-b_{\half,\half}^2=(1-2\mu)b_0.$
\end{itemize}
\smallskip
{\bf Proof of (d4)}.  (1) Since $\ga_b=1,$ and since $b_0b_{\eta,\eta}=b_{\eta,\eta}b_0=\frac{1-2\eta}{2},$ for all $(\eta,\eta)\in S^{\circ},$ we get by using Theorem \ref{Ser02}(ii)(a\&b), Theorem \ref{Ser02}(viii), and   part (ii)(b), that
\begin{gather*}
w^2=b_0^2+\sum_{(\eta,\eta)\in S^{\circ}(a)}(\varepsilon_{\mu,\mu}-1)^2b_{\mu,\mu}^2+2\sum_{(\eta,\eta)\in S^{\circ}(a)}(\varepsilon_{\mu,\mu}-1)b_0b_{\mu,\mu}\\
\qquad \qquad =b_0+\sum_{(\eta,\eta)\in S^{\circ}(a)}(\varepsilon_{\mu,\mu}^2-2\varepsilon_{\mu,\mu})b_{\mu,\mu}^2-\sum_{(\eta,\eta)\in S^{\circ}(a)}(\varepsilon_{\mu,\mu}-1)^2b_{\mu,\mu}.
\end{gather*}
Now, by the fusion rules, $w^2\in\ff b+\ff w.$  Thus $w^2\in\ff w$, seen by checking the coefficient of $a.$  Write $w^2=\gd w.$  Then $\gd (\varepsilon_{\eta,\eta}-1)=-(\varepsilon_{\eta,\eta}-1)^2,$ so $\gd=-(\varepsilon_{\eta,\eta}-1),$ for each $(\eta,\eta)\in S^{\circ}(a)\setminus\{(\half,\half)\}.$ We conclude that $S^\circ(a) \setminus S^\dagger(a)\subseteq \{(\mu,\mu),(\half,\half)\},$ with $\mu\ne\half.$  By symmetry $S^{\circ}(b)\subseteq \{(\half,\half), (\gr,\gr)\},$ with $\gr\ne\half.$
\smallskip

(2) Follows from (1), using part (ii)(b).
\smallskip

(3)  By the proof of (1), $w^2=(1-2\mu)w.$ Comparing the $(0,0)$ component, we get (3).
\smallskip

\item[({\bf d5)}]

Assume that $A_{\gr,\gr}(b)\ne 0.$  Let $0\ne y\in A_{\gr,\gr}(b).$  If, in the notation of Theorem \ref{eigcom}, $\theta_1=\theta'_1=0,$ then $y\in\ff b_0+\ff b_{\mu,\mu}.$
\medskip

{\bf Proof of (d5).}  First note that $y\notin\ff y_{\gr}.$ This is by Theorem \ref{eigcom}(i)(2\&3), for if $\theta_1=0,$ then $\gr\ne\ga_b.$  Normalizing we may write
\[
y=\ga_y a+b_0+\sum_{(\mu',\nu')\in S^{\circ}(a)}b_{\mu',\nu'}.
\]
By \eqref{671}, and since $\ga_b-\gr\ne 0,$ we see that $\ga_y=0.$  Since $\varepsilon_{\mu',\nu'}=1,$ for $(\mu',\nu')\in S^{\dagger},$ we see, using \eqref{5.4}, that $\gc_{\mu',\nu'}=0=\gc_{\hal,\half},$ for all $(\mu',\nu')\in S^{\dagger},$ proving (d5).
\medskip

\item[({\bf d6)}]

 Let
\[
\mathcal{A}=\ff a+\ff b_0+\sum_{(\eta,\eta)\in S^{\circ}(a)}\ff b_{\eta,\eta}.
\]
and let
\[
b^{\natural}=a+b_0+\sum_{(\eta,\eta)\in S^{\circ}(a)}b_{\eta,\eta}.
\]
Then $\mathcal{A}$ is a subalgebra of $A,$ and $b^{\natural}$ is an idempotent in $\mathcal{A}.$

If $A_{\eta,\eta}(b)\subseteq \mathcal{A}$ for all $(\eta,\eta) \in S^{\circ}(a),$ then $b^{\natural}$ is an axis  in $\mathcal{A}$ satisfying the fusion rules.
\medskip

{\bf Proof of (d6).} The fact that $\mathcal{A}$ is a subalgebra is obvious.  Also, by the fusion rules, by (ii)(a) and by (d2),
\[
\textstyle{b^2=(b^{\natural})^2+\sum_{(\mu',\nu')\in S^{\dagger}(a)}b_{\mu',\nu'},}
\]
Since $b^2=b,$ we see that $(b^{\natural})^2=b^{\natural}.$

Suppose $A_{\eta,\eta}(a)\subseteq \mathcal{A},$ for all $(\eta,\eta)\in S^{\circ}(a).$ Then for $z\in A_{\eta,\eta}(a),$
\[
\textstyle{bz\in b^{\natural}z+\sum_{(\mu',\nu')\in S^{\dagger}(a)}\ff b_{\mu',\nu'}.}
\]
Since $bz=\eta z\in\mathcal{A},$ we see that $\eta z=bz=b^{\natural}z.$  Similarly $b^{\natural}w=0,$ because $w\in\mathcal{A},$ by (d1).  Hence $\mathcal{A}$ is a direct sum of paired eigenspaces of $b^{\natural},$ so $b^{\natural}$ is an axis in $\mathcal{A},$ and obviously satisfies the fusion rules.
 \medskip

\item[({\bf d7)}]
If $\mid S^{\circ}(a)\setminus S^{\dagger}(a)\mid =1,$ then $A$ is as in Example \ref{egnon} or \ref{egnon1}.
\medskip

{\bf  Proof of (d7).}
\smallskip

Set $S^{\circ}(a)\setminus S^{\dagger}(a)=\{(\eta,\eta)\}.$  By (d3), $b_{\eta,\eta}^2\ne 0,$ so by Lemma~\ref{Ser041}(iii),
\begin{equation}\label{haha}
\text{If }A_{\half,\half}(b)\ne 0,\text{ then }A_{\half,\half}(b)\nsubseteq V(a).
\end{equation}
Suppose first that $\eta=\mu.$
By Theorem \ref{Ser04}(v)(b), $A_{\half,\half}(b)=0,$ since $\varepsilon_{\mu,\mu}\ne 1.$ By (1) of Claim (d4), $S^{\circ}(b)\setminus S^{\dagger}(b)=\{(\gr,\gr)\},$ with $\gr\ne\half.$   By part (3) of Claim (d4) and by Claim (d5), $A_{\gr,\gr}(b)\subseteq\mathcal{A}.$ By Claim (d6), $b^{\natural}$ is an axis in $\mathcal{A}$ satisfying the fusion rules.  By Theorem A, $\mathcal{A}$ is an HRS algebra.  By Theorem~(1.1) in~\cite{HRS} (see also the proof on p.~102), $\mathcal{A}=B(\mu,1),$
 because $\mu\ne\half,$ and $\gvp,$ in the notation of \cite{HRS}, is $\ga_b$ in our notation; hence $\mu=2,$ and $A$ is as in Example~\ref{egnon} (where $(\half,\half)\notin S^{\circ}(a)$).

Suppose next that $\eta=\half.$  Assume first that $A_{\half,\half}(b)\ne 0.$ Using~\eqref{haha} and Theorem \ref{Ser04}(v)(b)(3), $A_{\half,\half}(b)\subseteq \ff b_0+\ff b_{\half,\half}\subseteq\mathcal{A}.$  We thus see that $\mathcal{A}$ is the algebra $B(\half,1)$ of \cite{HRS}, so $A$ is as in~Example \ref{egnon1}.

Finally, if $A_{\half,\half}(b)=0,$ then by (d4), $S^{\circ}(b)\setminus S^{\dagger}(b)=\{(\gr,\gr)\}.$  Interchanging the roles of $a$ and $b$ we get a contradiction by the first part of this proof.
\medskip

\item[({\bf d8)}]
Finally assume that $S^{\circ}(a)=S^{\dagger}(a)\cup\{(\half,\half), (\mu,\mu)\},$ and $S^{\circ}(b)=S^{\dagger}(b)\cup\{(\half,\half), (\gr,\gr)\}.$ Then $A$ is as in Example \ref{egnon}, with $(\half,\half)\in S^{\circ}(a)$.
\medskip

{\bf Proof of (d8).}
By (d6) we need only  show that $A_{\eta,\eta}(a) \subseteq \mathcal{A}$ for all $\eta \in \{\half,\rho\}.$ Then, by Theorem A, $\mathcal{A}$ is as in Example \ref{egc4}, so $A$ is as in Example~\ref{egnon}.

Since $\varepsilon_{\mu,\mu}\ne 1,$ Theorem \ref{Ser04}(v)(b) shows that $A_{\half,\half}(b)\subseteq V(a).$  By Lemma \ref{Ser041}(iii) $A_{\half,\half}(b)\subseteq \ff b_{\half,\half}+\ff b_{\mu,\mu},$ and if $\zeta\in\{\half,\mu\},$ is such that $b_{\zeta,\zeta}^2\ne 0$ (see (d3)), then $b_{\zeta',\zeta'}^2=\gd b_{\zeta,\zeta}^2,$ for some $\gd\in\ff,$ where $S^{\circ}(a)=\{(\zeta,\zeta),(\zeta',\zeta')\}.$  But now, using part (3) of (d4), we see that $b_{\eta,\eta}^2\in \ff b_0,$ for $\eta\in\{\half,\mu\},$ so $\theta_1=\theta'_1=0,$ where $\theta_1,\theta'_1$ are as in (d5). Hence, by (d5), $A_{\gr,\gr}(b)\subseteq\mathcal{A}.$
\end{itemize}

Note that we have covered all cases, since if $|S^{\circ}(b)\setminus S^{\dagger}(b)|=1,$ we can interchange the role of $a$ and $b.$ This completes the proof of Theorem \ref{Yoav1}.
\end{proof}

    If $b$ is weakly primitive and satisfies the fusion rules and $A_0(b)=0$, then we can switch $a$ and $b$ and appeal to Theorem~\ref{b0zero}. Thus we have concluded the classification of 2-generated weak PAJ's where $\dim A_0(a)\le 1.$

\section{Some examples of   algebras generated by  idempotents}

Since our previous results included the description of $2$-generated axial algebras of dimension $<4,$ the following four-dimensional examples may be of interest. First note that in any associative algebra an idempotent has   only has left and right eigenvalues in  $\{0,1\}$ and obviously satisfies the fusion rules, by the Peirce decomposition. Here is one natural example.

\begin{example} Recall from \cite{W}
the {\it left-regular band} semigroup comprised of idempotents, satisfying $ab x = ab$ and  $ ba x= ba$ for all elements $a,b,x$.
   The   semigroup algebra $A$ of the left-regular band is associative, and in particular, for 2 generators $a,b$,  the  left-regular band  $\{ a,b,ab,ba\}$ yields a 4-dimensional 2-generated axial algebra with $a,b$ satisfying the fusion rules (but $a,b$ are not weakly primitive).
\end{example}

Here is a noncommutative, nonassociative 4-dimensional example, where $b$ is not an axis.

\begin{example}\label{side}$ $  A 4 dimensional  algebra $A$ generated by a non-weakly primitive axis $a$ satisfying the fusion rules,
and an idempotent  $b$ which is not an axis.  Imitating Example~\ref{E1}, we define $A := \lan\lan a,b \ran\ran   =
\ff a + \ff b +\ff y +\ff y'$ satisfying the relations
$$ ab = \mu y +\nu y' =ay  = y'a,  \quad   ba = \nu y +\mu y' =ay' =ya,    $$ $$     by = \nu y,\quad  yb = \mu y, \quad  y'b = \mu y',  \quad  by' =  \mu y',$$ $$  \quad   y^2 = 0 =  yy'= y'y= (y')^2 ,$$
where $\mu +\nu = 1.$ We check the eigenspaces: $$a(y+y') =  \mu y +\nu y'+  \nu y +\mu y' = y+ y' = (y+y')a,$$
$$a(b-y) = ab-ay=0,$$ $$ (b-y)a = ba -ya= 0,$$
$$a(y-y') = (\mu y +\nu y')-(\nu y +\mu y')=  (\mu - \nu)(y-y'),$$
$$(y-y')a = (\nu y +\mu y')-(\mu y +\nu y') =  (\nu - \mu)(y-y').$$

Hence $A_1(a) = \ff a + \ff (y+y'), \qquad   A_0(a)= \ff(b-y),$ and $$ A_{\mu - \nu,\nu - \mu}(a)= \ff(y-y').$$

$b(y+y') = \nu y+\mu y',$ whereas $(y+y')b = \mu (y+y'), $ so $b$ is not an axis.

We check the fusion rules for $a$:
$$A_0(a)^2= \ff(b-y)^2= \ff(b-by-yb)= \ff(b-y) = A_0(a),$$
$$A_1(a)^2 = ( \ff a + \ff (y+y'))^2 = \ff a + \ff (y+y') + 0 = A_1(a),$$
$$A_0(a)A_1(a) = \ff (b-y)(y+y') = \ff b(y+y') = \ff(y+y') \subseteq A_1(a),$$
$$A_1(a)A_0(a) = \ff (y+y')(b-y) = \ff (y+y')b = \ff(y+y') \subseteq A_1(a),$$
$$A_0(a)A_{\mu - \nu,\nu - \mu}(a) = \ff (b-y)(y-y') = \ff (\nu - \mu)(y-y') = A_{\mu - \nu,\nu - \mu}(a),$$
$$A_{\mu - \nu,\nu - \mu}(a)A_0(a) = \ff (y-y')(b-y) =  \ff (\mu - \nu)(y-y') = A_{\mu - \nu,\nu - \mu}(a),$$
$$A_1(a)A_{\mu - \nu,\nu - \mu}(a) =   a A_{\mu - \nu,\nu - \mu}(a) + \ff (y+y')(y-y') =  A_{\mu - \nu,\nu - \mu}(a)+0 ,$$
$$A_{\mu - \nu,\nu - \mu}(a)A_1(a)= A_{\mu - \nu,\nu - \mu}(a),$$
$$A_{\mu - \nu,\nu - \mu}(a)^2 = \ff (y-y')^2 =0.$$
\end{example}

\end{document}